\documentclass[p1]{elsarticle}

\usepackage{booktabs} 
\usepackage[T1]{fontenc}
\usepackage[utf8]{inputenc}
\usepackage[finnish, british]{babel}
\usepackage{amsmath}
\usepackage[pdfencoding=auto]{hyperref}				
\PassOptionsToPackage{hyphens}{url}
\DeclareUrlCommand\doi{\def\UrlLeft##1\UrlRight{doi:\href{http://dx.doi.org/##1}{##1}}\urlstyle{rm}}

\usepackage{cleveref}
\usepackage{graphicx}
\usepackage{caption}
\usepackage{subcaption}
\usepackage{epstopdf}
\usepackage[numbers]{natbib}				
\usepackage{enumerate}			
\usepackage{amsfonts}			
\usepackage{listings}
\usepackage{tikz}
\usepackage{tikz-cd}
\usepackage{pifont}
\usepackage{bbm}
\usepackage[linesnumbered,ruled,vlined]{algorithm2e} 
\usepackage{algpseudocode}
\usepackage{amsthm} 
\usepackage{amssymb} 
\usepackage{wasysym} 
\usetikzlibrary{trees}
\usepackage{float}
\usepackage{libertine}

\usepackage{apxproof} 
\usepackage{thmtools,thm-restate}

\usepackage{epigraph} 
\makeatletter\@addtoreset{chapter}{part}\makeatother%
\usepackage{xcolor} 

\def\HiLi{\leavevmode\rlap{\hbox to \hsize{\color{yellow!50}\leaders\hrule height .8\baselineskip depth .5ex\hfill}}}

\usepackage{blindtext}
\usepackage{enumitem}
	\newlist{condenum}{enumerate}{1} 
	\setlist[condenum]{label=\bfseries Condition \arabic*., 
                   ref=\arabic*, wide}
\makeatletter
\newcommand{\removelatexerror}{\let\@latex@error\@gobble}
\makeatother

\declaretheorem[style=definition,numberwithin=section]{definition}

\declaretheorem[style=definition,numberwithin=section]{remark}
\declaretheorem[style=definition,numberwithin=section]{example}
\declaretheorem[style=definition,numberwithin=section]{theorem}
\declaretheorem[style=definition,numberwithin=section]{proposition}

\DeclareMathOperator*{\argmin}{arg\,min}

\newtheoremstyle{case}{}{}{}{}{}{:}{ }{}
\theoremstyle{case}
\newtheorem{case}{Case}

\newenvironment{claim}[1]{\par\noindent\underline{Claim:}\space#1}{}
\newenvironment{claimproof}[1]{\par\noindent\underline{Proof:}\space#1}{\hfill $\blacksquare$}

\makeatletter
\def\@author#1{\g@addto@macro\elsauthors{\normalsize%
    \def\baselinestretch{1}%
    \upshape\authorsep#1\unskip\textsuperscript{%
      \ifx\@fnmark\@empty\else\unskip\sep\@fnmark\let\sep=,\fi
      \ifx\@corref\@empty\else\unskip\sep\@corref\let\sep=,\fi
      }%
    \def\authorsep{\unskip,\space}%
    \global\let\@fnmark\@empty
    \global\let\@corref\@empty  
    \global\let\sep\@empty}%
    \@eadauthor={#1}
}
\makeatother
\usepackage{lineno,hyperref}
\modulolinenumbers[5]
\journal{Theoretical Computer Science}

\usepackage{lineno,hyperref}
\modulolinenumbers[5]

\begin{document}

\begin{frontmatter}

\title{A Framework for Population-Based Stochastic Optimization on Abstract Riemannian Manifolds
}

\author{Robert Simon Fong\corref{cor1}}
\ead{r.s.fong@cs.bham.ac.uk}
\cortext[cor1]{Corresponding author}

\author{Peter Ti\v no}
\ead{p.tino@cs.bham.ac.uk}

\address{Department of Computer Science, University of Birmingham, Birmingham, B15 2TT, United Kingdom}

\begin{abstract}

We present Extended Riemannian Stochastic Derivative-Free Optimization (Extended RSDFO), a novel population-based stochastic optimization algorithm on Riemannian manifolds that addresses the local restrictions and implicit assumptions of manifold optimization in the literature. 



We begin by investigating the Information (statistical) Geometrical structure of statistical model over Riemannian manifolds. This establishes a geometrical framework to construct Extended RSDFO incorporating both the statistical geometry of the decision space and the Riemannian geometry of the search space. We first construct locally inherited probability distribution via an orientation-preserving diffeomorphic bundle morphism, and then extend the information geometrical structure to mixture densities over totally bounded subsets of manifolds. The former relates the information geometrical structure of the decision space and the local point estimations on the search space manifold. The latter overcomes the local restrictions of parametric probability distributions on Riemannian manifolds.

Using the geometrical framework we construct Extended RSDFO and study its evolutionary step and properties from a geometrical perspective. We show that Extended RSDFO's expected fitness improves monotonically and that it converges globally eventually in finitely many steps on connected compact Riemannian manifolds.

Extended RSDFO is compared to Riemannian Trust-Region method, Riemannian CMA-ES and Riemannian Particle Swarm Optimization on a set of multi-modal optimization problems over a variety of Riemannian manifolds.

In particular, we perform a novel synthetic experiment on Jacob's ladder to motivate and necessitate manifold optimization. Jacob's ladder is a non-compact manifold of countably infinite genus, which cannot be expressed as polynomial constraints and does not have a global representation in an ambient Euclidean space. Optimization problems on Jacob's ladder therefore cannot be addressed by traditional (constraint) optimization techniques on Euclidean spaces, which necessitates the development of manifold optimization algorithms.
\end{abstract}

\end{frontmatter}

\section{Introduction}
\label{sec:intro}
In this paper we consider multi-modal, black-box optimization problems over Riemannian manifolds $M$:
\begin{align*}
    \max_{x\in M} f(x) \quad ,
\end{align*}
where $f:M \rightarrow \mathbb{R}$ denotes the objective function over the search space Riemannian manifold $M$. The search space manifold $M$ is regarded as an abstract, stand-alone non-linear search space, that is free from an ambient Euclidean space.

Multi-modal, black-box optimization problems over Euclidean spaces are approached by population-based optimization methods. Notable examples include meta-heuristics such as Particle Swarm Optimization \citep{eberhart1995particle}, and population-based Stochastic Derivative-Free Optimization (SDFO) methods such as Estimation of Distribution Algorithms (EDA) \cite{larranaga2001estimation} and Covariance Matrix Adaptation Evolution Strategy \cite{hansen2006cma}.


Population-based optimization methods are adapted to tackle optimization problems on Riemannian manifolds in two distinctive fashions, each with their own drawbacks.


In manifold optimization methods developed in recent literature \cite{gabay1982minimizing,absil2019collection}, the search space manifold is considered as an abstract, stand-alone non-linear search space. The work in the literature focuses on adapting existing optimization methods from Euclidean spaces to Riemannian manifolds. Examples include gradient-based optimization methods \cite{absil2009optimization}, population-based meta-heuristics \cite{borckmans2010modified} and model-based stochastic optimization \cite{colutto2010cma}. Whilst the structure of the pre-adapted algorithm is preserved, the computations and estimations are locally confined by the normal neighbourhood. Furthermore, additional assumptions on the search space manifold has to be made to accommodate the adaptation process.

In the classical optimization literature, optimization on Riemannian manifolds falls into the category of constraint optimization methods \cite{absil2009optimization}.  Indeed, due to Whitney's embedding theorem \cite{whitney1944self,whitney1944singularities}, all manifolds can be embedded in a sufficiently large ambient Euclidean space. As traditional optimization techniques are more established and well-studied in Euclidean spaces, one would be more inclined to address optimization problems on Riemannian manifold by first finding an embedding onto the manifold, and then applying familar classical optimization techniques. The search space manifold is therefore considered as a subset of an ambient Euclidean space, which is then described by a set of functional constraints. However, the global structure of Riemannian manifolds is generally difficult to determine. Therefore a set of functional constraints that describes general search space manifolds can often be difficult, or even impossible to obtain.

The paper advances along these two directions. We construct a stochastic optimization method on Riemannian manifolds that overcomes the local restrictions and implicit assumptions of manifold optimization methods in the literature. We address multi-modal, black-box optimization problems on Riemannian manifolds using only the \textit{intrinsic} statistical geometry of the decision space and the \textit{intrinsic} Riemannian geometry of the search space \footnote{That is, the search space manifold does not have to be embedded into an ambient Euclidean space.}. To motivate and necessitate manifold optimization, we perform a novel synthetic experiement on Jacob's ladder, a search space manifold that cannot be addressed by classical constraint optimization techniques.


To this end, we take the long route and investigate information geometrical structures of statistical models over manifolds. We describe the statistical geometry of locally inherited probability densities on smooth manifolds using a local orientation-preserving diffeomorphic bundle morphism (Proposition \ref{naturality}, Section \ref{Subsection:localDensity}), this generalizes the use of Riemannian exponential map described in both the manifold optimization \cite{colutto2010cma} and manifold statistics literature \cite{Pennec2004ProbabilitiesAS,oller1993intrinsic}. 

To overcomes the local restrictions of manifold optimization algorithms and parametric probability distributions on Riemannian manifolds in the literature, we require a family of parametrized probability
densities defined beyond the normal neighbourhoods of Riemannian manifolds. We therefore construct a family of mixture densities on totally bounded subsets of $M$ as a mixture of the locally inherited densities (Section \ref{sec:mixturedefn}). We show that the family of mixture densities has a product statistical manifold structure (Theorem \ref{prop:lvdstruct}, Remark \ref{rmk:lvisprod}), this allows us to handle statistical parameter estimations and computations of mixture coefficients and mixture components independently.


This constitutes a geometrical framework for stochastic optimization on Riemannian manifolds by combining the information geometrical structure of the decision space and the Riemannian geometry of the search space, which:
\begin{enumerate}
\item Relates the statistical parameters of the statistical manifold decision space and local point estimations on the Riemannian manifold, and
\item Overcomes the local restrictions of manifold optimization algorithms and parametric probability distributions on Riemannian manifolds in the literature.
\end{enumerate}

Using the product statistical geometry of mixture densities, we propose Extended  Riemannian Stochastic Derivative-Free Optimization (Extended RSDFO) (Algorithm \ref{alg:ereda}), a novel population-based stochastic optimization algorithm on Riemannian manifolds, which addresses the local restrictions and implicit assumptions (Section \ref{sec:manopt:assmption})of manifold optimization in the literature.

The geometrical framework also allows us to study the more general properties of Extended RSDFO, previously unavailable to population-based manifold optimization algorithms due to the local restrictions (Section \ref{sec:shortcoming:rsdfo}). We discuss the geometry and dynamics of the evolutionary steps of Extended RSDFO (Section \ref{sec:theo}) using a novel metric (Equation \eqref{eqn:new-metric}) on the simplex of mixture coefficients. We show that expected fitness obtained by Extended RSDFO improves monotonically (Proposition \ref{rmk:monoincreasing}), and show that Extended RSDFO converges globally eventually in finitely many steps on connected compact Riemannian manifolds (Theorem \ref{thm:conv}).



We wrap up our investigation by comparing Extended RSDFO with state-of-the-art manifold optimization methods in the literature (Section \ref{sec:experiment}), such as Riemannian Trust-Region method \cite{absil2007trust,absil2009optimization}, Riemannian CMA-ES \cite{colutto2010cma} and Riemannian Particle Swarm Optimization  \cite{borckmans2010modified,borckmans2010oriented}, using optimization problems defined on the $n$-sphere, Grassmannian manifolds, and Jacob's ladder. 

Jacob's ladder is a non-compact manifold of countably infinite genus, which cannot be expressed as polynomial constraints and does not have a global representation in an ambient Euclidean space. Optimization problems on Jacob's ladder therefore cannot be addressed by traditional (constraint) optimization techniques on Euclidean spaces, which necessitates the development of manifold optimization algorithms.

 
%

The organization of the rest of the paper is summarized as follows:

In Section \ref{sec:dg:prelim}, we review the definition and properties of the essential elements of Differential Geometry. This provides the foundations for the discussions in this paper. 

In Section \ref{sec:principle:manopt}, we discuss the general principle of manifold optimization in the literature, which we then use to describe a generalized framework, the Riemannian Stochastic Derivative-Free Optimization (RSDFO) algorithms, for adapting SDFO algorithms from Euclidean spaces to Riemannian manifolds. We discuss the main drawback, i.e. the local restrictions of the Riemannian adaptation approach.

In Section \ref{app:localdistn}, we describe the notion of locally inherited probability densities on smooth manifolds and derive its statistical geometrical structure via an orientation-preserving diffeomorphic bundle morphism (Diagram \eqref{tikz:OPbundle}). The use of Riemannian exponential map described in both the manifold optimization \cite{colutto2010cma} and manifold statistics literature \cite{Pennec2004ProbabilitiesAS,oller1993intrinsic} is then recovered as special case of this construction.

In Section \ref{Subsection:MixtureDensity}, we construct a family of mixture densities on totally bounded subsets of $M$ as a mixture of the locally inherited densities, and its product statistical geometry is derived therein (Theorem \ref{prop:lvdstruct}, Remark \ref{rmk:lvisprod}). 

In Section \ref{ch:paper2}, we propose a novel algorithm -- Extended RSDFO using the product Riemannian structure of mixture densities, which extends and augments an RSDFO core and addresses the local restriction of RSDFO (and manifold optimization algorithms). 

In Section \ref{sec:theo}, we describe the geometry of the evolutionary steps of Extended RSDFO using a novel metric (Equation \eqref{eqn:new-metric}) on the simplex of mixture coefficients.

In Section \ref{sec:conv}, we show that Extended RSDFO converges globally eventually in finitely many steps on connected compact Riemannian manifolds (Theorem \ref{thm:conv}).

In Section \ref{sec:experiment}, we compare Extended RSDFO against state-of-the-art manifold optimization methods such as Riemannian Trust-Region method, Riemannian CMA-ES and Riemannian Particle Swarm Optimization on the $n$-sphere, Grassmannian manifold, and Jacob's ladder.
 


In \ref{ch:litreview:manopt}, we reviews manifold optimization algorithms in the literature: Riemannian adaptations of trust-region method \cite{absil2007trust}, Particle Swarm Optimization (PSO) \cite{borckmans2010modified} and Covariance Matrix Adaptation Evolution Strategy (CMA-ES) \cite{colutto2010cma}. This illustrates the Riemannian adaptation approach described in Section \ref{sec:principle:manopt}. \ref{app:proof:induceddualisticgeo} details proofs of some results described in Section \ref{Section:InducedDualisticGeo}. In \ref{app:proof:OPcover}, we proof Lemma \ref{lemma:makediffeoOP} in detail.

\section{Preliminaries  from Riemannian geometry}
\label{sec:dg:prelim}
In this section we review notions and machinery in Riemannian geometry required for the subsequent sections. See textbooks \cite{lee2006riemannian,petersen2006riemannian} for further details.

Let $(M,g)$ be a \textbf{Riemannian manifold} with its corresponding Riemannian metric. For each point $x\in M$, let the vector space $T_x M$ denote the \textbf{tangent space} at $x$. Tangent vectors $v$ on a tangent space $T_x M$ can be thought of as directional derivatives or velocity vectors at $x \in M$.

Given a point $x$ in Riemannian manifold $(M,g)$ and a tangent vector $v\in T_x M$, there exists a unique \textbf{geodesic} $\gamma_v: \left[0,I_v \right) \subset \mathbb{R} \rightarrow M$ with initial point $x$ and initial velocity $v\in T_x M$. For sufficiently small $I_v$, the geodesic is the ``curve of minimal length" on $(M,g)$ that connects $x$ and $\gamma_v(I_v)$ (with respect to the distance generated by the Riemannian metric $g$), analogous to a straight line in $M = \mathbb{R}^n$. The \textbf{Riemannian exponential map} $\exp_x: T_x M \rightarrow M$ is a function that maps a tangent vector $v$ in the tangent space $T_x M$ to $M$ by tracing along the geodesic $\gamma_v$ starting at $x$ with initial velocity $v \in T_x M$ for time $1$.
%
%

For each $x\in M$, there exists a open star-shaped neighbourhood  $U_x$ around $\vec{0} \in T_x M$ where the Riemannian exponential map is a \textit{local} diffeomorphism. The image of $U_x$ under $\exp_x$ is the open neighbourhood $N_x := \exp_x(U_x) \subset M$ of $x\in M$ called the \textbf{normal neighbourhood} of $x$. 

Moreover, since $\exp_x : U_x \rightarrow N_x$ is a local diffeomorphism within $U_x$, the Riemannian exponential map is invertible \textit{within} the normal neighbourhood and the inverse is called the \textbf{Riemannian logarithm map} $\log_x := \exp_x^{-1} : N_x \rightarrow U_x$.


\subsection{Parallel Transport in Geodesic balls/Normal neighborhoods}
\label{sec:paralleltransport}
For any point $ x \in M$, normal neighborhood $N_x$ of $x$ is diffeomorphic to a star-shaped open neighbourhood of $\vec{0} \in T_x M$ via the Riemannian exponential map. Computations can therefore be performed on the tangent space in a similar fashion as in the Euclidean case, and the results subsequently translated back onto the manifold locally via the Riemannian exponential map. 

The algorithms discussed in this paper will therefore focus on normal neighbourhoods centered around the search iterates in the manifold $M$. In particular, we will use \textbf{parallel transport} to transfer search information from (the tangent space of) the current iterate to the next, within the normal neighbourhood of the current iterate.

Given a point $x \in M$ and fixing a normal neighbourood $N_x$ of $x$, let $\left\{e_1(x), \ldots, e_(x) \right\}$ denote an orthonormal basis for $T_x M$. \footnote{This can be generated using the Gram–Schmidt process.} This induces a linear isomorphism:
$E: \mathbb{R}^n \rightarrow T_x M$ mapping $(v^1,\ldots, v^n) \mapsto \sum_{i=1}^n v^i e_i(x)$. Together with the Riemannian exponential map, we obtain a coordinate function within $N_x$, called \textbf{normal coordinates centered at $x$}, given by:
$E^{-1} \circ \exp_x^{-1}: N_x \rightarrow \mathbb{R}^n$.

Since $\exp_x$ is a diffeomorphism with range $N_x$, for any point $y$ in the normal neighbourhood $N_x$ centered at $x$ can be connected to $x$ via a unique geodesic with some initial velocity $v$. Within the normal coordinate centered at $x$, this geodesic is represented by a radial line segment radiating from $x$ in $N_x$: let $v = \sum_{i=1}^n v^i e_i(x) \in T_x M$, the geodesic $\gamma_v(t): \left[0,1 \right) \rightarrow M$ with initial velocity $v$ starting at $x$ is given by the radial line segment \cite{lee2006riemannian}: 
\begin{align}
\label{eqn:radialgeodesic}
\gamma_v(t) = \left(tv^1,\ldots,tv^n \right) \quad .
\end{align}


The tangent space of any point in $N_x$ thus admits an orthonormal basis corresponding to the $N_x$ and it's Riemannian exponential map. Let $P_{x,y}: T_x M \rightarrow T_y M$ denote the parallel transport from $T_x M$ to $T_y M$ along $\gamma_v$, and $\left\{e_1(x), \ldots, e_n(x) \right\}$ denote an orthonormal basis for $T_x M$. Since parallel transport is an isometry, the set of tangent vectors $\left\{e_1(y), \ldots, e_n(y) \right\} := \left\{P_{x,y} e_1(x), \ldots, P_{x,y} e_n(y) \right\}$ is an orthonormal basis for $T_y M$. Together with Equation \eqref{eqn:radialgeodesic}, the parallel transport of any tangent vector $w:= \sum_{i=1}^n w^i e_i(x) \in T_x M$ from $T_x M$ to $T_y M$ \textit{within the normal neighbourhood} is given by:
\begin{align}
\label{eqn:parallelinnormalnhbd}
P_{x,y} w = \sum_{i=1}^n w^i e_i(y) \in T_y M \quad .
\end{align}

The above discussion is illustrated in Figure \ref{fig:paralleltransport} below.
%
\begin{figure}[htbp]
    \centering
    \def\svgwidth{\columnwidth}
    \includegraphics[scale=0.5]{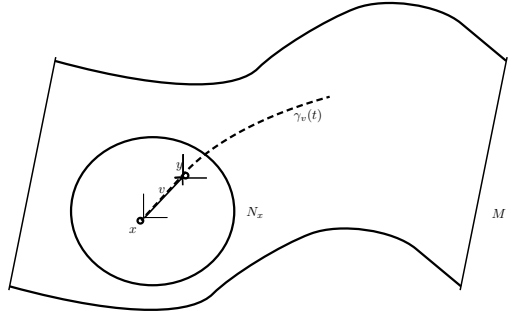}
    \caption{Illustration of parallel transport from $x$ to $y$ within normal neighbourhood $N_x$ of $x$, described in Section \ref{sec:paralleltransport}. The orthonormal basis at $x$ is parallel transported to an orthonormal basis at $y$.}
    \label{fig:paralleltransport}
\end{figure}

Finally, when restricting our attention to a metric ball within the pre-imagie of the normal neighbourhood, we retrieve a metric ball in $M$ under normal coordinates. 

\begin{definition}
Given a point $x\in M$, the \textbf{injectivity radius} at $x$ is the real number \cite{petersen2006riemannian}:
\begin{align*}
\operatorname{inj}(x) &:= \sup_{r\in \mathbb{R}} \left\{ \operatorname{exp}_x: B(\vec{0},r) \subset T_x M \rightarrow M \text{ is a diffeomorphism} \right\} \quad ,
\end{align*}
where $ B(\vec{0},r)$ is an ball of radius $r$ centered at $\vec{0}$ in $T_x M$.

For $x\in M$, the \textbf{ball of injectivity radius centered at $\vec{0}$} denoted by $B_x := B(\vec{0},\operatorname{inj}(p)) \subset U_x \subset T_x M$ is the largest metric ball in $T_x M$ such that the Riemannian exponential map $\exp_x$ is a diffeomorphism. 

For any $j_x \leq \operatorname{inj}(x)$, the set $\exp_x \left(B(\vec{0},j_x)\right) \subset M$ is a neighbourhood of $x\in M$ called a \textbf{geodesic ball}. 
\end{definition}
Geodesic balls in $M$ are also metric balls in $M$ of the same radius. For the rest of the paper, the geodesic balls are \textit{closed} unless specified otherwise.

\section{Riemannian Adaptation of Optimization Algorithms}
\label{sec:principle:manopt}

In this section we review the general principle of Riemannian adaptation of optimization algorithms from Euclidean spaces to Riemannian manifolds, otherwise known as manifold optimization or Riemannian optimization in the literature \cite{absil2019collection}. Using this Riemannian adaptation approach, we describe a generalized framework, the Riemannian Stochastic Derivative-Free Optimization (RSDFO) algorithms, for adapting Stochastic Derivative-Free Optimization (SDFO) algorithms from Euclidean spaces to Riemannian manifolds. RSDFO encompasses Riemannian adaptation of CMA-ES \cite{colutto2010cma}, and accentuates the main drawback, i.e. the local restrictions of the Riemannian adaptation approach.

Manifold optimization in the literature uses tangent spaces and Riemannian exponential maps \footnote{Gradient-based algorithms can have a slightly relaxed variations of this map called retraction \cite{absil2009optimization}, as we will discuss in the Appendix. The principle concept is the same.} to translate optimization methods from Euclidean spaces to Riemannian manifolds. The procedure of Riemannian adaptation of optimization algorithms in the literature can be summarized as follows: for each iteration, the optimization method determines a search direction on the tangent space centered at the current iterate. The next iterate is then obtained by tracing along the \emph{local} geodesic from the current iterate along the search direction via the Riemannian exponential map. The search information is subsequently parallel transported to (the tangent space of) the new search iterate. The decision space of Riemannian adapted algorithms is thus given by the tangent (Euclidean) space, whereas the search space is given by the Riemannian manifold.

The computations and estimations within a \textit{single} normal neighbourhood of Riemannian adapted optimization algorithms are therefore ``externalized" to (the pre-image of normal neighbourhood in) the tangent space, which is a vector space. This means local computations on the manifold can be done on the tangent space in the same fashion as the pre-adapted algorithms. The above discussion is summarized in Figure \ref{fig:manopt} below.

\begin{figure*}[h!]
    \centering
    \begin{subfigure}[t]{0.5\textwidth}
        \centering
        \includegraphics[width=1\textwidth]{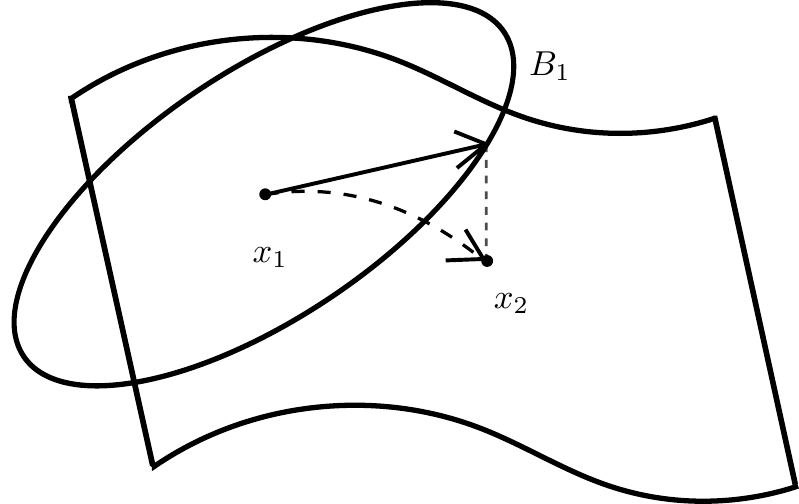} 
        \caption{Local computation in $B_1$ around $x_1$}    \label{fig:manopt:1}
    \end{subfigure}%
    ~ 
    \begin{subfigure}[t]{0.5\textwidth}
        \centering
        \includegraphics[width=1\textwidth]{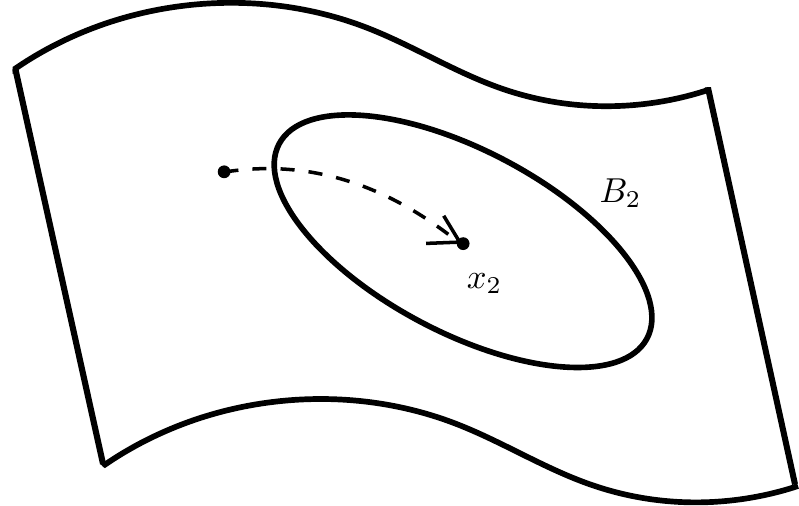} 
        \caption{Translation to new search iterate $x_2$}    \label{fig:manopt:2}
    \end{subfigure}
        \caption{Illustration of Riemannian adaptation of optimization algorithms. For each iteration, Riemannian optimization initiate with $x_1 \in M$, search direction is obtained by performing local computation on $B_1 \subset T_{x_1} M$ and a new search iterate $x_2$ is obtained by tracing along the geodesic defined by the search direction illustrated in Figure \ref{fig:manopt:1}. This process is repeated for $x_2$ on $B_2 \subset T_{x_2}M$ illustrated on Figure \ref{fig:manopt:2}}
        \label{fig:manopt}
\end{figure*}

This principle has since been adopted by authors to translate optimization methods from Euclidean space to \textbf{complete} Riemannian manifolds \footnote{The authors assume completeness of the search space manifold in order to extend the estimations in normal neighbourhood as far as possible.}. Notable examples include gradient-based methods on matrix manifolds \cite{gabay1982minimizing,absil2009optimization}, meta-heuristics (derivative-free methods) such as Particle Swarm Optimization (PSO)\cite{eberhart1995particle} on complete Riemannian manifolds \cite{borckmans2010modified} and stochastic derivative-free optimization algorithms such as Covariance Matrix Adaptation Evolutioanry Stragegies (CMA-ES) \cite{hansen1996adapting,kern2004learning, hansen2006cma} on spherical manifolds \cite{colutto2010cma}. 

Detailed descriptions of examples of Riemannian Adaptation of Optimization Algorithms in the literature will be discussed in \ref{ch:litreview:manopt}.

\subsection{Riemannian Stochastic Derivative-Free Optimization Algorithms}
\label{sec:rsdfo}
The scope of this paper is to solve multi-modal optimization problems over Riemannian manifolds. In this section, we focus on  Stochastic Derivative-Free Optimization (SDFO) algorithms, which is traditionally employed to solve multi-modal optimization problems over Euclidean spaces. Using the adaptation approach, we describe a generalized framework, the Riemannian Stochastic Derivative-Free Optimization (RSDFO) algorithms, for adapting SDFO algorithms from Euclidean spaces to Riemannian manifolds, which encompasses Riemannian adaptation of CMA-ES \cite{colutto2010cma}. We wrap up the section by describing the shortcomings of the Riemannian adaptation approach under RSDFO.


We begin by describing general Stochastic Derivative-Free Optimization (SDFO) algorithms on Euclidean spaces. Let $f: \mathbb{R}^n \rightarrow \mathbb{R}$ denote the objective function of an optimization problem over $\mathbb{R}^n$. Consider a parametrized family of probability densities over $\mathbb{R}^n$: $\left\{ p(\cdot \vert \theta) \middle| \theta \in \Xi\subset \mathbb{R}^m \right\}$. Examples of SDFO algorithms include Estimation of Distribution Algorithms (EDA) \cite{larranaga2001estimation}, Evolutionary Algorithms such as CMA-ES \cite{hansen1996adapting}, and natural gradient optimization methods \cite{infogeo13}. We summarize SDFO algorithms on Euclidean space in Algorithm \ref{alg:eda} below:

\begin{algorithm}
 \KwData{Initial point $x_0 \in \mathbb{R}^n$, initial parameter $\theta^0 \in \Xi$}
 \While{stopping criterion not satisfied}{
	Generate a sample $X^k \sim p(\cdot \vert \theta)$ \; 
	
	Evaluate fitness $f(x)$ of each $x\in X^k$ \;
	
	Update $\theta^{k+1}$ based on the fittest subset of the sample: $\hat{X}^k \subset X^k$ \; 
	
  	k = k+1\;
 } 
 \caption{General SDFO on Euclidean spaces}
 \label{alg:eda}
\end{algorithm}

Let $(M,g)$ be a Riemannian manifold, and let the function $f: M \rightarrow \mathbb{R}$ denote the objective function of an optimization problem over $M$. For each $x \in M$, let $\tilde{f}_x := f\circ \exp_x: T_x M \rightarrow \mathbb{R}$ denote the locally converted objective function on each tangent space $T_x M$ of $M$. For each $x \in M$, let  $U_x := \exp_x^{-1} N_x \subset T_x M$, centered at $\vec{0} \in T_x M$, denote the pre-image of a normal neighbourhood $N_x$ of $x$ under the Riemannian exponential map. We fix a family of parametrized distributions $S_x$, parametrized by $\theta\in \Xi \subset \mathbb{R}^\ell$, $\ell \in \mathbb{N}_+$, over \emph{all} pre-images of normal neighbourhoods in tangent spaces $U_x \subset T_x M$.

RSDFO is an iterative algorithm that generates search directions based on observations within the normal neighbourhood of the current search iterate $x \in M$. For each iteration, tangent vectors are sampled from a parametrized probability distribution in the predefined family $S_x$ over  $U_x\subset T_x M$ of the search iterate $x \in M$. The fittest sampled tangent vectors are then used to determine the search direction and to estimate the stochastic parameter $\theta\in \Xi$ for $S_x$. The new search iterate is obtained by tracing along the local geodesic starting at $x$ in the direction of the mean of the fittest tangent vectors, and the estimated statistical parameters are subsequently parallel transported to the new search iterate. This procedure is summarized in Algorithm \ref{alg:greda} and Figure \ref{fig:rsdfo}.
\newline





\begin{algorithm}[H]
 \KwData{Initial point $\mu_0 \in M$, initial statistical parameter $\theta^0  \in \Xi$, set iteration counter $k = 0$, and step size functions $\sigma_\theta(k), \sigma_v(k)$.}
 \While{stopping criterion not satisfied}{
	Generate set of sample vectors $V^k \sim P(\cdot \vert \theta^k) \in \text{Prob}(U_{\mu_k})$ \protect\footnotemark \; 
	
	Evaluate locally converted fitness $\tilde{f}_{\mu_k}(v)$ of each $v \in V^k \subset T_{\mu_k} M$ \;
	
	Estimate statistical parameter $\hat{\theta}^{k+1}$ based on the fittest tangent vectors $\hat{V}^k \subset V^k$ \label{rsdfo:parameter} \; 
	
	Translate the center of search: $\hat{v}_k:= \sigma_v(k)\cdot \mu(\hat{V}^k)$, and $\mu_{k+1} := \exp_{\mu_k} (\hat{v}_k)$.
	
	Parallel transport the statistical parameters: $\theta^{k+1} := P_k^{k+1} \left(\sigma_\theta(k)\cdot \hat{\theta}^{k+1}\right)$, where $P_k^{k+1}: T_{\mu_k} M \rightarrow T_{\mu_{k+1}} M$ denote the parallel transport from $T_{\mu_k} M$  to  $T_{\mu_{k+1}} M$ \;
	
  	k = k+1\;
 } 
 \caption{Riemannian SDFO}
 \label{alg:greda}
\end{algorithm}

\footnotetext{Note that $V^k \subset U_{\mu_k} \subset T_{\mu_k} M$ are vectors on the tangent space, which is a Euclidean space.}
It is worth noting that, adjustments of statistical parameters in line \ref{rsdfo:parameter} of RSDFO depend on the pre-adapted SDFO algorithm.
%


\begin{figure}[hbt!]
    \centering
    \def\svgwidth{\columnwidth}
    \includegraphics[width=0.6\textwidth]{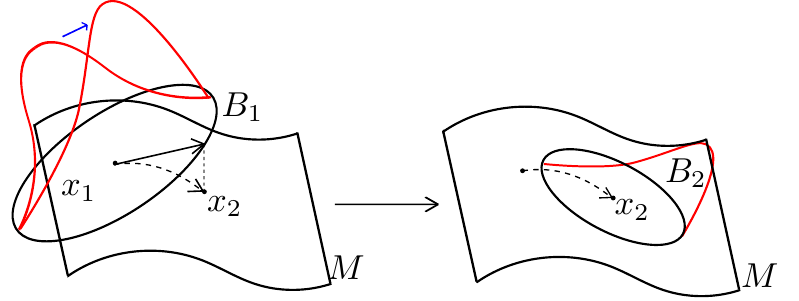}
    \caption{Illustration of an iteration of RSDFO. For each iteration, the algorithm initiates with $x_1 \in M$. A set of sample tangent vectors is drawn from a sampling distribution on $B_1 \subset T_{x_1} M$, and a new sampling distribution is estimated on the overarching statistical model $S_{x_1}$. The new sampling distribution is then parallel transported to the new search iterate $x_2 \in M$. Notice the estimations and computations are \textit{strictly local} within the tangent spaces of the search iterates.}
    \label{fig:rsdfo}
\end{figure}

\subsection{Discussion, Shortcoming and Outlook}
\label{sec:shortcoming:rsdfo}

In RSDFO, parametrized densities are inherited from tangent spaces to manifolds locally via Riemannian exponential map \footnote{Similar to the approach of \cite{Pennec2004ProbabilitiesAS}. This will be further generalized via the notion of orientation-preserving diffeomorphism as we will discuss in the next sections (Section \ref{app:localdistn}).}, this preserves both the local/step-wise computations and the overall structure of the original SDFO algorithm on Euclidean spaces. This is due to the fact that tangent spaces are vector spaces, meaning computations can be perform locally on the tangent spaces as in the pre-adapted algorithm. The only additional tool needed to translate the result back to the Riemannian manifold was the Riemannian exponential map and to some extent the knowledge of the injectivity radius of the Riemannian manifold. 

However, Riemannian SDFO algorithms constructed with the Riemannian adaptation approach also lead to two drawbacks due to the local-ness of the normal neighbourhood and inherited parametrized densities on the manifold:
\begin{itemize}
\item 
First of all, the local-ness of the inherited parametrized densities implies that the quality of the Riemannian version of SDFOs depends strongly on the structure of the manifold. For each iteration, the search region centered at $x \in M$ is bounded by the normal neighbourhood $N_x$ of $x$. Manifolds with high scalar curvature will have small injectivity radius (thus small normal neighbourhoods) \cite{abresch1997injectivity,klingenberg1961riemannsche, cheeger1975comparison}, thus the observations made in the search region bounded by normal neighbourhoods would be restrictive, especially in manifolds with high scalar curvature.

Therefore, whilst SDFO such as CMA-ES have an advantage in tackling multi-model optimization problems \cite{hansen2004evaluating}, this advantage may not be reflected on the Riemannian version of SDFO in more complex manifolds. 

\item 
Secondly, tangent spaces of a manifold are disjoint. Since the support of the parametrized density in Algorithm \ref{alg:greda} varies from iteration to iteration, the local densities do not belong to the same family of probability densities, even if they share the same parametrization. Comparison of solution quality and convergence behaviors of such algorithms would be difficult to study in the fashion of \cite{zhang2004convergence,beyer2014convergence}, as there is no explicit relation between densities from one iteration to another.

\end{itemize}

In order to overcome the local restrictions of the RSDFO under the Riemannian adaptation approach, we require a parametrized probability densities defined beyond the confines of a single normal neighbourhoods in Riemannian manifolds. In the next two sections, we study families of \textit{intrinsic}, \textit{computable}, \textit{parametrized} distribution on manifolds. In particular, we study the information geometrical structures of finitely parametrized statistical models over Riemannian manifolds. This establishes an extension framework for Riemannian adapted model-based stochastic optimization algorithms, which we then use to propose a novel algorithms in the latter sections of the paper.

\section{Dualistic Geometry of Locally Inherited Parametrized Densities on Riemannian Manifolds}
\label{app:localdistn}

In the next two sections, we study the Information Geometrical \cite{Amari2000} structure of families of intrinsic, (finitely) parametrized probability densities supported on the manifold beyond the normal neighbourhoods, where the statistical parameters are compatible with the point estimations. This establishes the geometrical foundations of an extension framework for stochastic optimization on Riemannian manifold, linking the Riemannian geometry of the search space manifold and the Information Geometry of the decision space. The latter bridges the Information Geometrical structure, statistical parameters and point estimations on the search space manifold beyond the confines of a single normal neighbourhood.

In the literature, there are two approaches of constructing intrinsic probability distributions on Riemannian manifolds $M$:
\begin{enumerate}
\item The first approach \cite{oller1993intrinsic,Pennec2004ProbabilitiesAS,pennec2006intrinsic,bhattacharya2008statistics} is motivated by statistical analysis on Riemannian manifold, which originated from fields including directional statistics \cite{jupp1989unified,mardia1975} and shape space analysis \cite{small2012statistical,le1993riemannian}. Similar to the Riemannian adaptation principle described in the previous section, this branch focuses on a \textit{single} normal neighbourhood $N_x$ centered at $x\in M$, where computational and statistical tools on $N_x \subset M$ can be translated locally from equivalent notions on the tangent space (which is a Euclidean space). Local statistical tools and probability distributions can thus be re-established on $M$ within a \textit{single} normal neighbourhood via Riemannian exponential map.
\item The second approach \cite{khesin2013geometry,bruveris2016geometry,bauer2016uniqueness} stems from the study of probability measure as a differential $n$-form. Motivated by the theoretical information geometrical interpretation of probability densities over Euclidean spaces, it studies the metric, geodesic and other geometrical structure of the vector bundle of \textit{all} probability densities over compact manifolds as an infinite dimensional manifold. \footnote{This is formally known as a Fréchet manifold. Just as finite dimensional manifolds are locally modeled on Euclidean spaces, Fréchet manifolds are locally modeled on Fréchet spaces. \cite{michor1980manifolds}.}
\end{enumerate}

However, neither approach is suitable for our purpose of constructing a geometrical framework for stochastic optimization on Riemannian manifolds:
\begin{enumerate}
\item The first approach is too restrictive: Whilst the statistical properties can be locally inherited through Riemannian exponential map, all the computations, similar to the Riemannian adaptation approach described in the previous section, are confined to a \textit{single} normal neighbourhood in $M$.
\item The second approach is too general: There is no clear way to associate the geometry and statistical parameter of a specific family of finitely parametrized probability densities on $M$.
\end{enumerate}

We begin our construction by combining the essence of the two approaches: Given a normal neighbourhood $N_x \subset M$, family of locally inherited finitely parametrized probability densities on $N_x$ is viewed as a submanifold of $\operatorname{Prob}\left(N_x\right)$ endowed with a statistical geometry \cite{Amari2000}. The statistical geometry of the locally inherited probability densities on $N_x$ is then described explicitly by inheriting the statisical geometry of probability densities on the tangent space via the Riemannian exponential map. This process is described by a bundle morphism, summarized in Figure \ref{fig:localditn}. 


\begin{figure}[hbt!]
    \centering

	\begin{tikzpicture}[thick,scale=0.9, every node/.style={transform shape}]

    \node at (-3,1.5) {$S_x \subset \text{Prob}(U_x)$};
    \node at (3,1.5) {$\tilde{S}_x \subset \text{Prob}(M)$};
    \node at (-3,-1.5) {$U_x \subset T_x M$};
    \node at (3,-1.5) {$M$};
    \draw[<-] (-2,-1.5) -- (2.5,-1.5);
    \draw[->] (-1.5,1.5) -- (1.5,1.5);
    \draw[->] (3,1.3) -- (3,-1.3);
    \draw[->] (-3,1.3) -- (-3,-1.3);
    \node at (0,1.85) {$\log_x^{^*}$};
    \node at (0,-1.15) {$\log_x$};
    
    \node[blue] at (3,-3) {Probability Distribution on $M$};
    \draw[draw=blue] (5,-2.5) rectangle (1,3);
    
    \node[black!60!green] at (-3,-3) {Information Geometry};
    \draw[draw=black!60!green] (-5,-2.5) rectangle (-1,3);
    
    \node[red] at (2,-4.5) {Locally Inherited Distributions on $M$ Sec. \ref{SectionProbOnMfold}};
    \draw[draw=red] (-6.5,-4) rectangle (6.5,4.5);
    
      \node[black] at (2.5,3.75) {Naturality of Dualistic Geometry Sec. \ref{Section:InducedDualisticGeo}};
    \draw[draw=black] (-5.5,0.5) rectangle (5.5,3.5);
    
    \end{tikzpicture}
    \caption{A summary of the induced statistical dualistic geometry of locally inherited parametrized densities on $M$ described in Section \ref{app:localdistn},}
    \label{fig:localditn}
\end{figure}

This section is organized as follows:

\begin{enumerate}
\item In Section \ref{Section:InducedDualisticGeo}, we discuss in detail the ``naturality" of dualistic geometry between two manifolds. That is, given a manifold with a predefined dualistic structure $\left(S,g,\nabla,\nabla^*\right)$, a smooth manifold $\tilde{S}$ and diffeomorphism $\varphi:\tilde{S} \rightarrow S$, we show that a dualistic structure $(\varphi^* g,\varphi^* \nabla, \varphi^* \nabla^*)$ can be induced on $\tilde{S}$ via $\varphi$.

This is an adaptation of the notion of naturality of Riemannian structure to the context of dualistic structure of Hessian-Riemannian manifolds. Furthermore, we show that the induced dualistic structure and the corresponding divergence can be computed explicitly via the pulled-back local coordinates. This describes the top horizontal part of the bundle morphism in Figure \ref{fig:localditn} boxed in black. The proofs of this section will be detailed in \ref{app:proof:induceddualisticgeo}.

\item In Section \ref{SectionProbOnMfold}, we discuss how families of probability densities on $M$ can be inherited locally from open subsets of $\mathbb{R}^n$ via an orientation-preserving bundle diffeomorphism. We discuss how the locally inherited family of probability densities on $M$ inherit the geometrical properties from the family of probability densities on $\mathbb{R}^n$, and show that it generalizes the first approach. This is summarized by the entire bundle morphism of Figure \ref{fig:localditn} boxed in red.
\end{enumerate}
For the rest of the paper we will, without loss of generality, assume $M$ is \textit{connected} and \textit{orientable} unless specified otherwise.


\subsection{Naturality of Dualistic Structure}
\label{Section:InducedDualisticGeo}

We begin the discussion by formally introducing the notion of \textbf{statistical manifolds} \cite{Amari2000,calin2014geometric} -- the Riemannian manifold structure of (finitely parametrized) statistical models. Let  $S := \left\{ p_\theta \middle| \theta \in \Xi \subset \mathbb{R}^n \right\}$  be a statistical model over a measurable space $M$ described by the (smooth) injective immersion:
\begin{align*}
\iota: \Xi  \subset \mathbb{R}^n &\hookrightarrow P\left( M\right)=  \left\{p: M \rightarrow \mathbb{R} \middle| p \geq 0,\int_M p = 1 \right\} \\
\theta = \left(\theta^1, \ldots ,\theta^n \right) &\mapsto p_\theta \quad .
\end{align*} 
The statistical model $S$ can be endowed with a Riemannian manifold structure \cite{Amari2000,calin2014geometric} described by a Riemannian metric $g$ \footnote{This is typically given by the Fisher-Rao metric \cite{Amari2000}.} and a pair of $g$-conjugate connections satisfying:
\begin{align*}
X\langle Y,Z \rangle_g = \langle \nabla_X Y,Z \rangle_g + \langle Y, \nabla^*_X Z \rangle_g\quad, \quad \forall X,Y,Z \in \mathcal{E}\left(TS\right)\quad ,
\end{align*}
where $\mathcal{E}\left(TS\right)$ denote the smooth sections of the tangent bundle $TS$. The triplet $\left(g,\nabla,\nabla^*\right)$, called a \textbf{dualistic structure} of statistical manifold $S$ fundamental to the study of the intrinsic geometry of statistical manifolds. A statistical manifold is called \textbf{dually flat} if it is flat with respect to one of $\left(\nabla, \nabla^*\right)$.

In this section we show how dualistic structure of $\left(S,g,\nabla,\nabla^*\right)$ can be inherited onto an arbitrary smooth manifold $\tilde{S}$ via a diffeomorphism.  We show explicitly the relation between induced dualistic structure, Hessian structure \cite{shima1997geometry} and local coordinate systems on finite dimensional statistical manifolds. In particular, we show how one structure can be determined from another computationally. This is represented by the top horizontal (black) part of the bundle morphism in Figure \ref{fig:localditn}.

We discuss two different ways of pulling back (dually flat) dualistic structures given a diffeomorphism from one manifold to another. We first show that general dualistic structures can be pulled back directly via diffeomorphism. We then show when the manifolds are dually flat, the induced dualistic structure can be computed implicitly via the pulled-back coordinates and metric. 

Whilst the first method arises more naturally in a theoretical setting, the second provides a more computable way to describe the induced dualistic structure which is equivalent to the first method when the manifolds are dually flat.



The result of this section is an adaptation of naturality of Levi-Civita connection in classical Riemannian geometry \cite{lee2006riemannian} to the notion of dualistic structures on statistical manifolds, which provides the foundations for the subsequent discussions of the geometry of inherited probabilty densities on manifolds. The proofs of this section will be detailed in \ref{app:proof:induceddualisticgeo}.

Suppose $S$ is  a finite dimensional manifold equipped with torsion-free dually flat dualistic structure $(g,\nabla,\nabla^*)$, then we can induce via a diffeomorphism a (dually flat) dualistic structure onto another manifold $\tilde{S}$, as described by the following proposition:
\begin{restatable}[Naturality of Dualistic Structure]{proposition}{naturality}
\label{naturality}
If $\varphi: \tilde{S} \rightarrow S$ is a diffeomorphism between smooth manifolds, and $S$ is equipped with torsion-free dualistic structure $(g,\nabla,\nabla^*)$, then $\tilde{S}$ is a Riemannian manifold with induced dualistic structure $(\varphi^* g,\varphi^* \nabla, \varphi^* \nabla^*)$.
\end{restatable}
If $\left(\nabla,g\right)$ satisfies Codazzi's equation on $S$ \cite{nomizu1994affine}:
\begin{align*}
X g\left( Y,Z \right) - g\left( \nabla_X Y,Z \right) -  g\left( Y,\nabla_X Z \right) = \left( \nabla_X g \right) \left(Y,Z \right) = \left( \nabla_Z g \right) \left(Y,X \right) \quad ,
\end{align*}
where $X,Y,Z \in \mathcal{E}\left(TS \right)$ denote vector fields on $S$. Then the pulled back connection and metric $\left( \tilde{\nabla},\tilde{g}\right) := \left(\varphi^* \nabla ,\varphi^* g \right)$ also satisfies Codazzi's equation. Therefore by \cite{shima2007geometry}, $\tilde{g}$ is a Hessian metric with respect to $\tilde{\nabla}$. This means there exists a potential function $\tilde{\psi}$ on $\tilde{S}$ such that $\tilde{g} = \tilde{\nabla}\,\tilde{\nabla}\,\tilde{\psi}$ is a Hessian metric, and $(\tilde{S},\tilde{g},\tilde{\psi})$ is a Hessian manifold.

Moreover, the induced dualistic structure $(\varphi^*g,\varphi^*\nabla,\varphi^*\nabla^*)$ on $\tilde{S}$ is invariant under sufficient statistics if $\left(g,\nabla,\nabla^*\right)$ is invariant under sufficient statistics on $S$ \cite{chenstov1982statistical}.

We can also determine the pull-back curvature on $\left(\tilde{S}, \varphi^*g,\varphi^*\nabla, \varphi^*\nabla^* \right)$ with the following immediate corollary:

\begin{restatable}{corollary}{curvaturenaturality}
\label{df}Let $\varphi: \tilde{S} \rightarrow S$ be a local  diffeomorphism. Suppose $S$ has $g$-conjugate connections $(\nabla,\nabla^*)$, and let $(\tilde{\nabla},\tilde{\nabla}^*) := (\varphi^*\nabla, \varphi^*\nabla^*)$ be the induced $\varphi^*g$-conjugate connections on $\tilde{S}$, then $\varphi^* R = \tilde{R}$ and $\varphi^* R^* = \tilde{R}^*$. In particular if $S$ is dually flat, then so is $\tilde{S}$.
\end{restatable}
\begin{proof}
See \ref{app:proof:induceddualisticgeo}.
\end{proof}
This result has been known for Levi-Civita connection on Riemannian manifold. Here we extend it slightly to pair of $g$-conjugate dual connections.

\subsubsection{Computing inherited dualistic structure}
\label{Subsection:computingInducedDualisticGeo}

Let $(\theta_i)_{i=1}^n$ denote local coordinates on $\left(S,g,\nabla,\nabla^*\right)$, and $\varphi:\tilde{S} \rightarrow S$ denote a diffeomorphism between smooth manifolds described in the previous section. In this section, we illustrate how  the pulled-back dually flat dualistic structure  $(\varphi^*g,\varphi^*\nabla, \varphi^*\nabla^*)$ on $\tilde{S}$ can be determined explicitly in the pulled-back local coordinate system $\left(\overline{\theta_i}:= \theta_i \circ \varphi\right)_{i=1}^n$ via the pulled-back metric and the corresponding induced potential function. 

\begin{definition}
\label{defn:divergence}
Given a smooth manifold $S$, a \textbf{divergence} \cite{Amari2000} $D$ or \textbf{contrast function}\cite{calin2014geometric} on $S$ is a smooth function  $D:S \times S \rightarrow \mathbb{R}_+$  satisfying the following:
\begin{enumerate}
\item $D(p;q) \geq 0$, and
\item $D(p,q) = 0$ iff $p = q$ . 
\end{enumerate}
\end{definition}

A dualistic structure $\left(g^D, \nabla^D, \nabla^{D^*} \right)$ on $S$ can be determined by divergence function $D$ via the following equations \cite{eguchi1992geometry,Amari2000,calin2014geometric} for each point $x\in S$:
\begin{align*}
\left. g^D_{ij}\right|_p =g^D_p\left( \partial_i,\partial_j\right) &:= - \partial^1_i \partial^2_j \left. D[p;q] \right|_{q=p} \\
\left. \Gamma_{ijk}^D \right|_p = \left. \langle\nabla^D_{\partial_i} \partial_j,\partial_k \rangle \right|_p &:= - \partial^1_i \partial^1_j \partial^2_k  \left. D[p;q] \right|_{q=p} \quad ,
\end{align*}
where $(\theta_i)_{i=1}^n$ denote local coordinates on $S$ with corresponding local coordinate frame $\left( \partial_i \right)_{i=1}^n$ about $x$. Let $\partial^\ell_i$ denote the $i^{th}$ partial derivative on the $\ell^{th}$ argument of $D$. By an abuse of notation, we may write \cite{Amari2000} as:
\begin{align}
\label{Dabusenotation}
D[\partial_i ; \partial_j] &:= - \partial^1_i \partial^2_j \left. D[p;q] \right|_{q=p}, \quad \text{and} \nonumber \\
- D[ \partial_i \partial_j ; \partial_k ] &:= - \partial^1_i \partial^1_j \partial^2_k  \left. D[p;q] \right|_{q=p} \quad .
\end{align}

\begin{remark}
\label{Rmk:InducedDivergence}
Conversely, given a torsion-free dualistic structure and a local coordinate system, there exists a divergence that generates the dualistic structure \cite{matumoto1993}. We will refer to the divergence $\tilde{D}$ on $\tilde{S}$ corresponding to the pulled-back dualistic structure $(\varphi^*g,\varphi^*\nabla, \varphi^*\nabla^*)$ (not necessarily dually flat) as the \textbf{induced divergence} on $\tilde{S}$.
\end{remark}

Consider smooth pulled back coordinates $\left(\overline{\theta_i}:= \theta_i \circ \varphi\right)_{i=1}^n$ on $\tilde{S}$ and the corresponding local coordinate frame $\left(\overline{\partial_i}\right)_{i=1}^n :=\left(\frac{\partial}{\partial \overline{\theta}_i} \right)_{i=1}^n$ on $T\tilde{S}$. Since $\tilde{g}$ is a Hessian metric with respect to $\tilde{\nabla}$, there exists a potential function $\tilde{\psi}$ such that $\tilde{g} = \tilde{\nabla}\tilde{\nabla}\tilde{\psi}$ and $\tilde{g}_{ij} = \overline{\partial}_i \overline{\partial}_j \tilde{\psi}$. The corresponding $\tilde{g}$-dual local coordinate system of $\tilde{S}$ with respect to $\left( \overline{\theta}_i \right)_{i=1}^n$, denoted by $\left( \overline{\eta}_i \right)_{i=1}^n$ can be defined by $\left(\overline{\eta}_i\right)_{i=1}^n = \left(\overline{\partial_i}\tilde{\psi} \right)_{i=1}^n$ with correspond local coordinate frame $\left( \overline{\partial}^i\right)_{i=1}^n = \left( \frac{\partial}{\partial \overline{\eta_i}} \right)_{i=1}^n$  of $T\tilde{S}$ \cite{Amari2000}. Consider the following divergence function on $\tilde{S}$ in the canonical form \cite{Amari2000}:
\begin{align}
\label{eqn:div}
\overline{D}: \tilde{S} \times \tilde{S} &\rightarrow \mathbb{R}_+ \nonumber \\
(\tilde{p},\tilde{q}) &\mapsto \overline{D}(\tilde{p},\tilde{q}) =  \tilde{\psi}(\tilde{p}) + \tilde{\psi}^\dagger(\tilde{q}) - \langle \overline{\theta}(\tilde{p}), \overline{\eta}(\tilde{q}) \rangle \quad ,
\end{align}
where $\tilde{\psi}^\dagger$ is a smooth function on $\tilde{S}$ representing the Legendre-Fr\'echet transformation of $\bar{\psi}$ with respect to the pair of $\tilde{g}$-dual local coordinates $\left(\overline{\theta}_i\right)_{i=1}^n,\left(\overline{\eta}_i\right)_{i=1}^n$ on $\tilde{S}$. Using the divergence function $\overline{D}$, we can compute the inherited dualistic structure  $\left(\varphi^* g,\varphi^* \nabla, \varphi^* \nabla^*\right)$ on $\tilde{S}$:

\begin{restatable}{proposition}{computenaturality}
\label{computenaturality}
Let $\left(S,g,\nabla,\nabla^*\right)$ be a dually flat manifold, and let $\varphi: \tilde{S} \rightarrow S$ be a local  diffeomorphism. On dually flat manifold $\left( \tilde{S}, \varphi^* g,\varphi^* \nabla, \varphi^* \nabla^*\right)$ the induced dualistic structure $\left(\varphi^* g,\varphi^* \nabla, \varphi^* \nabla^*\right)$ coincides with the dualistic structure $\left(\overline{g}, \overline{\nabla},\overline{\nabla}^* \right)$ generated by $\overline{D}$ on the local pulled-back coordinates $\left(\overline{\theta_i}:= \theta_i \circ \varphi\right)_{i=1}^n$.
\end{restatable}

\begin{proof}
See \ref{app:proof:induceddualisticgeo}.
\end{proof}


\subsection{Locally Inherited Probability Densities on Manifolds}
\label{SectionProbOnMfold}
%
%

In this section, we construct a family of locally inherited probability densities over smooth manifolds via orientation-preserving bundle morphism. This extends and generalizes the construction of probability distributions on geodesically complete Riemannian manifolds via Riemanian exponential map  described in \cite{pennec2006intrinsic}. The discussion in this section is represented by the entire bundle morphism of Figure \ref{fig:localditn} boxed in red. 

%
%

In the approach of \cite{Pennec2004ProbabilitiesAS} (similarly on the Riemannian adaptation principle described in Section \ref{sec:principle:manopt}), probability densities and the corresponding statistical properties on geodesically complete Riemannian manifolds are constructed by inheriting probability densities on the tangent space via the Riemannian exponential map.  \footnote{It is also worth noting that a similar direction has been pursued by related work in \cite{jermyn2005invariant}, where the author described the inheritance of metric through a \textit{single} local coordinate charts. However, further inheritance of the whole geometrical structure such as dualistic structure was not discussed.} 

Let $U_x \subset T_x M$ denote the region where the exponential map is a diffeomorphism. For each $y\in N_x = \exp_x(U_x)$ there is a unique $v\in U_x$ such that $\exp_x(v) = y$. Given a probability density function $x$ supported on $U_x \subset T_x M$, a probability density function $\tilde{p}$ on $N_x \subset M$ whose function value on $y = \exp_x(v) \in N_x =  \exp_x(U_x)$ can be constructed by:
\begin{align*}
p(v) = \tilde{p}(\exp_x(v)) \quad ,
\end{align*}
or equivalently, since $\exp_x$ is a local diffeomorphism on $U_x$ we may write:
\begin{align}
\label{pennecdistn2}
\tilde{p}(y) = p(\log_x(y)) \quad ,
\end{align}
where $\log_x$ denote the Riemannian logarithm map, which is the inverse of the Riemannian exponential map on $U_x$.


However, determining the explicit expression for Riemannian exponential map for general Riemannian manifolds could be computationally expensive in practice \cite{absil2009optimization}, as it involves solving the geodesic equation, which is a second order differential equation.  Therefore this section aims to find the explicit expression of parametrized probability distributions on manifolds with a more general map. We extend and generalize the above construction in two ways: 
\begin{enumerate}
\item Finitely parametrized probability densities can be locally inherited via\textit{ orientation-preserving diffeomorphisms} from subsets of Euclidean spaces instead of Riemannian exponential map on tangent spaces. The orientation-preserving diffeomorphism serves as a sufficient condition to inherit dualistic geometry of family of probability distributions form Euclidean spaces to Riemannian manifolds. This generalizes the use of exponential map, whose closed form expression might be difficult to determine.
\item Moreover, we can inherit the entire statistical geometry from densities on $\mathbb{R}^n$ (equivalently tangent spaces). This allow us to draw correspondence between the statistical estimations of manifold data sets $y_i \in N_x$ and their corresponding set of tangent vectors $v_i := \log_x \left( y_i \right) \in U_x = \log_x \left( N_x \right)$.
\end{enumerate}

\subsubsection{Locally Inherited Probability Densities via Bundle Morphism}
\label{Subsection:localDensity}

In this section, we describe locally inherited family of finitely parametrized probability densities on smooth manifolds induced via bundle morphism. 

Let $M$ be a smooth topological manifold (not necessarily Riemannian), and let $\text{Prob}(M) \subset \text{Vol}(M)$ denote the vector bundle of compactly supported positive differential $n$-forms over $M$ that integrates to $1$. Compactly supported positive differential $n$-forms over $M$ can be natually associated to smooth positive functions over $M$: Let $U \subset M$ be compact, denote the positive differential $n$-forms on $U$ by $\Omega^n_+\left( U \right)$ and the set of positive smooth real-valued functions $f:U \rightarrow \mathbb{R}_+$ by  $C^\infty_+ \left(M \right)$. There exists a diffeomorphism: \cite{bruveris2016geometry,bauer2016uniqueness}:
\begin{align}
\label{eqn:geoapproach:diffeo}
R: \hat{S} \subset C^\infty_+ \left(M \right) &\rightarrow S \subset \mathcal{E} \left( \Omega^n_+\left( M \right)\right) \nonumber \\
f &\mapsto f^2 \mu_0 \quad .
\end{align}
$S$ can therefore inherit the dualistic structure of $\hat{S}$ via the diffeomorphism $R: C^\infty_+ \left(U \right)\rightarrow \mathcal{E} \left( \Omega^n_+\left( U \right)\right)$. $\text{Prob}(M) \subset \Omega^n_+\left( M \right) \subset \text{Vol}(M)$ can therefore be naturally associated to density functions over $M$ \cite{lee2001introduction} given a reference measure $\mu_0$ on $M$. \footnote{The reference measure on Riemannian manifold $M$ is typically given by the Riemannian volume form: $\mu_0 := dV_g$.} Analogously, we let $\text{Prob}(\mathbb{R}^n)$ denote the volume form over $\mathbb{R}^n$ that integrates to $1$.

Let $U$ be a compact subset of $\mathbb{R}^n$, consider a family of finitely parametrized probability density functions on $U \subset \mathbb{R}^n$: $\hat{S} := \left\{ p_\theta: U \rightarrow \mathbb{R} \middle| \theta \in \Xi \subset \mathbb{R}^\ell, \int_U p_\theta = 1, \operatorname{supp}\left(p_\theta\right) \subset U  \right\}$. Suppose without loss of generality that elements of $\hat{S}$ are mutually absolutely continuous, then $\hat{S}$ has the structure of a statistical manifold \cite{Amari2000}. 

Let $\mu$ be an arbitrary reference measure on $U \subset \mathbb{R}^n$, and let $S := \left\{ \nu_\theta = p_\theta d\mu \middle| \theta \in \Xi \right\} \subset \text{Prob}(U)$ denote the set of differential $n$-forms over $U$ naturally associated to $\hat{S}$ (the space of functions) with respect to $\mu$. Since $U$ is compact, $S$ inherits the dualistic structure of $\hat{S}$ via diffeomorphism $R: C^\infty_+ \left(U \right)\rightarrow \mathcal{E} \left( \Omega^n_+\left( U \right)\right)$ in Equation \eqref{eqn:geoapproach:diffeo}.

Consider a map $\rho: U \subset \mathbb{R}^n \rightarrow M$ and the corresponding pullback $\rho^{-1^*} : S \subset \text{Prob}\left( U \right) \rightarrow \operatorname{Prob} \left( M\right)$. In order to inherit the dualistic geometry of $S$ to $\rho ^{-1^*} S =: \tilde{S} \subset \operatorname{Prob} \left( M\right)$, a sufficient condition for $\rho$ is that it must be a diffeomorphism. Moreover, to ensure the pulled-back $n$-forms in $\tilde{S}$ integrates to $1$, $\rho$ must also be orientation-preserving.

Therefore it suffices to consider orientation-preserving diffeomorphism $\rho: U \subset \mathbb{R}^n \rightarrow M$ and the locally inherited family of probability densities on $M$ induced by  $\rho$. In particular, locally inherited family of probability densities over $\rho(U) \subset M$ is constructed via the pullback bundle morphism defined by orientation-preserving diffeomorphism\footnote{Since $\rho$ is an orientation-preserving diffeomorphism, then so is $\rho^{-1}$}: $\rho^{-1} : M \rightarrow U$:
%
\begin{center}
\begin{equation}
\label{tikz:OPbundle}
\begin{tikzcd}[ampersand replacement=\&]
\hat{S} \subset C^\infty_+\left( U\right) \arrow{r}{R} \& S \subset \text{Prob}(U) \arrow{r}{\rho^{-1^*}} \arrow{d} \& \tilde{S}:= \rho ^{-1^*} S  \subset \text{Prob}(M)\arrow{d} \\ 
\& U\subset \mathbb{R}^n \& \rho (U ) \subset M \arrow[swap]{l}{\rho^{-1}}
\end{tikzcd} 
\end{equation}
\end{center} 

where $\tilde{S}  \subset \text{Prob}(M)$  is a family of probability densities over $M$  given by:
\begin{gather*}
\tilde{S}  := \rho ^{-1^*} S = \left\{ \tilde{\nu}_\theta = \rho^{-1^*} \nu_\theta \right\} \quad .
\end{gather*}
More precisely, let $x \in V := \rho(U) \subseteq M$, and let $X_1,\ldots, X_n \in T_x M$  be arbitrary vectors. Given a probability density $\nu_\theta \in S \subset \text{Prob}(U)$, the pulled-back density $\tilde{\nu}_\theta$ on $V$ is given by:
\begin{gather}
\label{logstar}
\tilde{\nu}_\theta:=\left. \rho_\alpha^{-1^*} \nu_\theta  (X_1,\ldots, X_n)\right|_y = \left. \nu_\theta (\rho_{\alpha_*}^{-1} X_1,\ldots ,\rho_{\alpha_*}^{-1} X_n)\right|_{\rho_{\alpha}^{-1}(y)} \quad , \quad \forall y \in V_\alpha \subseteq M.
\end{gather}
The above diagram commutes: since $\rho$ is a local diffeomorphism, for each $v\in U $, there exists a unique $y \in \rho (U ) \subset M$ such that $y= \rho  (v)$, and $(v,p_\theta)$ is a section in the line bundle $\pi_{U }: \operatorname{Prob}(U ) \rightarrow U $. We have the following equalities:
\begin{align*}
\rho  \circ \pi_{U } (v,p_\theta) &= \rho  (v) = y \quad , \\ 
\pi_M \circ \rho^{-1^*} (v,p_\theta) &= \pi_M \left( \rho^{-1^{-1}}(v), \rho^{-1^*} p_\theta \right) = \rho^{-1^{-1}}(v) = \rho  (v) = y \quad .
\end{align*}

Moreover, since $\rho$ is an orientation-preserving diffeomorphism, so is $\rho^{-1}$. Therefore we have the following equality on the compact subset $V \subset M$: 
\begin{gather*}
1 = \int_{U} \nu_\theta = \int_{V:=\rho (U)} \rho^{-1^*} \nu_\theta  \quad .
\end{gather*}
Suppose $\nu_\theta \in S$ has probability density function $p_\theta$ with respect to the reference measure $\mu$ on $U$, i.e. $\nu_\theta = p_\theta d\mu$ on $U\subset \mathbb{R}^n$, then in local coordinates $\left(x^1,\ldots,x^n \right)$ of $M$, the above integral has the following form:
\begin{align*}
1 = \int_U p_\theta d\mu_0 = \int_V \left(p_\theta \circ \rho^{-1} \right) \left(\operatorname{det} D \rho^{-1} \right)dx^1 \wedge \cdots \wedge dx^n  \quad ,
\end{align*}
where $dx^1 \wedge \cdots \wedge dx^n$ denote the reference measure on $M$ corresponding to the local coordinates $\left(x^1,\ldots,x^n\right)$.
Finally, suppose $S $ has dualistic structure given by $\left(g,\nabla,\nabla^* \right)$. Since $\rho^{-1}$ is a diffeomorphism, so is $\rho^{-1^*}$. Therefore by the discussion in Section \ref{Section:InducedDualisticGeo}, $\tilde{S} $ has inherited dualistic structure $(\varphi^* g,\varphi^* \nabla, \varphi^* \nabla^*)$ with $\varphi = \rho^{-1^*}$. In particular, the locally inherited family of probability distributions $ \tilde{S}= \rho ^{-1^*} S$ inherits the dualistic geometrical structure of $S$ via the bundle morphism.

\begin{remark}
\label{lazyparameter}
Note that both the local coordinate map $p_\theta \in S \mapsto \theta \in \mathbb{R}^\ell$ and $\varphi := \rho^{-1^*}$ are diffeomorphisms. For the rest of the paper we will, without loss of generality, assume $\tilde{S}$ to be parametrized by $\left(\theta_i\right)_{i=1}^n$ instead of the pulled-back local coordinates $\left( \theta_i \circ \rho^{-1^*} \right)_{i=1}^n$  unless specified otherwise (described in Section \ref{Subsection:computingInducedDualisticGeo} and Remark \ref{Rmk:pullbackcoord}) if the context is clear.
\end{remark}

\begin{example}
\label{inheritalphafamily}
Suppose $\left(S,g,\nabla^{(\alpha)}, \nabla^{(-\alpha)}\right) \subset \text{Prob}(U)$ is an $\alpha$-affine statistical manifold for some $\alpha \in \mathbb{R}$, with Fisher metric $g$, the associated $g$-dual $\alpha$-connections $\left(\nabla^{(\alpha)}, \nabla^{(-\alpha)}\right)$, and the corresponding $\alpha$-divergence $D_\alpha$ on $S$ \cite{Amari2000}. Since $\rho:U\rightarrow M$ is a (local) diffeomorphism, it is injective, hence a sufficient statistic for $S$ \cite{Amari2000}. 

By the invariance of Fisher metric and $\alpha$-connection under sufficient statistic, the induced family $\tilde{S}$ is also an $\alpha$-affine statistical manifold. Furthermore, due to the monotonicity of $\alpha$-divergence (as a special case of $f$-divergence), the induced divergence $\tilde{D}_\alpha$ (see Remark \ref{Rmk:InducedDivergence}) on $\tilde{S}$ can be computed by:
\begin{align*}
 \tilde{D}_\alpha(\rho^{-1^*}p,\rho^{-1^*}q)=  D_\alpha(p,q) \quad , \quad \text{for } p,q\in S \quad .
\end{align*}
\end{example}

\subsubsection{Special case: Riemannian exponential map}

We conclude the section by illustrating how the above framework encapsulates the Riemannian adaptation approach (Section \ref{sec:principle:manopt}) and approach of \cite{oller1993intrinsic,Pennec2004ProbabilitiesAS,bhattacharya2008statistics}. In particular, we show that the approach using Riemannian exponential map on complete Riemannian manifolds is a special case of the framework described above.

\begin{example}
\label{EX:expmap}

Let $M$ be a (complete) Riemannian manifold. For each $x\in M$, let $U_x \subset T_x M$ denote the region where the Riemannian exponential map $\exp_x: T_x M \rightarrow M$ is a local diffeomorphism.

Since each tangent $T_x M$ is a topological vector space, it can be considered naturally as a metric space with the metric topology induced by the Riemannian metric. Since finite dimensional topological vector spaces of the same dimension $n := \text{dim}(M)$ are unique up to isomorphism,  $T_x M$ is isomorphic to $\mathbb{R}^n$. Moreover, since the Euclidean metric and Riemannian metric are equivalent on finite dimensional topological vector spaces, the respective induced metric topologies are also equivalent. This means probability density functions over $T_x M$ can be considered naturally as density functions over $\mathbb{R}^n$ \cite{lee2001introduction,petersen2006riemannian}.

Let $S_x$ denote a finitely parametrized family of probability densities (equivalently positive differential $n$-forms that integrates to $1$) over $U_x$. Since $\exp_x$ is a diffeomorphism in $U_x$, we can construct a parametrized family of probability distributions on $N_x = \exp_x (U_x)$ by:


\begin{center}
\begin{tikzcd}[ampersand replacement=\&]
S_x \subset \text{Prob}(U_x) \arrow{r}{\log_x^{^*}} \arrow{d} \& \tilde{S}_x \subset \text{Prob}(N_x)\arrow{d} \\ 
U_x \subset T_x M \& N_x \subset M 
\arrow[swap]{l}{\log_x}
\end{tikzcd}
\end{center}
where $\log_x^{^*}$ denote the pullback of the Riemannian logarithm function $\log_x = \exp_x^{-1}$ on $N_x$. For $p(\cdot \vert \theta) \in S_x$, the inherited probability density $\tilde{p}(\cdot \vert \theta) \in \tilde{S}_x$ over $M$ is given by the following, coinciding with Equation \eqref{pennecdistn2}:
\begin{align*}
    \tilde{p}(y \vert \theta) = \log_x^* p(y \vert\theta) = p(\exp_x^{-1}(y)\vert\theta) = p(\log_x(y)\vert\theta) \quad .
\end{align*}
Since $\exp_x$ is an orientation preserving diffeomorphism on $U_x$, the Riemannian adaptation approach and the approach  of \cite{oller1993intrinsic,Pennec2004ProbabilitiesAS,bhattacharya2008statistics} in the literature is therefore a special case of the construction of locally inherited densities via orientation-preserving bundle morphism discussed in Section \ref{Subsection:localDensity} above. 

It is important to note that for general Riemannian manifolds, this approach maybe quite limiting since $U_x$ maybe a small region in the tangent space (bounded by the injectivity radius) (see the discussion of Section \ref{sec:shortcoming:rsdfo}).  

Throughout the rest of the paper, we will be using the Riemannian exponential map as an example to illustrate our approach. It is however worth noting that our generalized construction applies to all orientation-preserving diffeomorphisms, not just the Riemannian exponential map.
\end{example}

\section{Mixture Densities on Totally Bounded Subsets of Riemannian Manifolds}
\label{Subsection:MixtureDensity}

In the previous section, we discussed how parametrized family of probability densities over Riemannian manifolds $M$ can be locally inherited from $\mathbb{R}^n$ via an orientation-preserving diffeomorphism. This provides us with a notion of computable, parametrized probability densities over manifolds, generalizing the Riemannian adaptation approach and the approach of \cite{oller1993intrinsic,Pennec2004ProbabilitiesAS,bhattacharya2008statistics}. However, the local restriction of the Riemannian adaptation still persisted.
In particular, the locally inherited densities are only defined within $\rho\left( U \right)$, where $\rho: U\subset \mathbb{R}^n \rightarrow M$ is the aforementioned orientation-preserving diffeomorphism. In the case when $\rho: \mathbb{R}^n \rightarrow M$ is the Riemannian exponential map, the construction is still confined within a single normal neighhbourhood. 

In this section we extend the notion of parametrized probability densities over manifolds beyond the confines of normal neighbourhoods. In particular, we describe the information geometry of mixture densities over totally bounded subsets $V$ of Riemannian manifolds.

Given a totally bounded subset $V$ of $M$, we can cover it by a finite number of metric balls (geodesic balls) where each metric ball is mapped to a subset of $\mathbb{R}^n$ via an orientation-preserving diffeomorphism. This set of metric balls (together with the mapping to $\mathbb{R}^n$) forms an orientation-preserving open cover of $V$.

Suppose each element of the open cover of $V$ is equipped with a local family of inherited probability distributions, a family of parametrized mixture densities $\mathcal{L}_V$ on the entire $V$ can be constructed by ``gluing" together the locally inherited densities in the form of a mixture. The statistical geometry of the family of mixture densities on $V$ can thus be described by the product of the mixture component families and the mixture coefficients.

The section is outlined as follows:
\begin{enumerate}
\item In section \ref{sec:orientationpreservingOC}, we describe the notion of orientation-preserving open cover on Riemannian manifolds $M$. We show that such open covers admits a refinement of geodesic balls. Thus totally bounded subsets $V$ of $M$ admits a finite orientation-preserving open cover of (compact) geodesic balls. 

\item In section \ref{sec:mixturedefn}, we describe finitely parametrized mixture distributions $\mathcal{L}_V$ over totally bounded subsets $V$ on Riemannian manifolds $M$. These parametrized probability densities extend the notion of locally inherited probability densities beyond a single normal neighbourhood/geodesic ball, while preserving the locally inherited dualistic statistical geometry of the mixture components.

\item From the section \ref{Subsection:GeoLv} onwards, we describe and derive the geometrical structure of the family of mixture densities $\mathcal{L}_V$. By viewing the closure of the simplex of mixture coeffieients and family of mixture component densities as statistical manifolds, we show that $\mathcal{L}_V$ is a smooth manifold under two condition. Furthermore, we show that a torsion-free dualistic structure can be defined on $\mathcal{L}_V$ by a ``mixture" divergence $D$ from the mixture coefficient and mixture components.

In section \ref{Subsection:DualGeoLv}, we show that $\mathcal{L}_V$ is in fact a product Riemannian manifold of the simplex of mixture coefficients and the families of mixture component statistical manifolds. Finally, we show that when both the simplex of mixture coefficients and mixture component statistical manifolds are dually flat, then $\mathcal{L}_V$ is also dually flat with canonical ``mixture" divergence $D$.
\end{enumerate}

\subsection{Refinement of Orientation-Preserving Open Cover}
\label{sec:orientationpreservingOC}
We begin by formally introducing the notion of orientation-preserving open cover. It is important to note that the notion of orientation-preserving open cover is different from orientable atlas in the differential geometry literature, as we do not require the transition maps between elements of the open cover to be compatible.

\begin{definition}
Let $(M,g)$ be a smooth $n$-dimensional Riemannian manifold, an \textbf{orientation-preserving open cover} of $M$ is an at-most countable set of pairs $\mathcal{E}_M := \left\{\left(\rho_\alpha,U_\alpha \right) \right\}_{\alpha \in \Lambda_M}$  satisfying: 
\begin{enumerate}
\item $U_\alpha\subset \mathbb{R}^n$, $\rho_\alpha: U_\alpha \rightarrow M$ are orientation preserving diffeomorphisms for each $\alpha \in \Lambda_M$ 
\item the set $\left\{ \rho_\alpha(U_\alpha) \right\}_{\alpha \in \Lambda_M}$ is an open cover of $M$.
\end{enumerate} 
\end{definition}

One notable example of an orientation-preserving open cover is given by geodesic balls:
Using the notation of Section \ref{sec:dg:prelim}: for $x \in M$, let $B_x := B(\vec{0},\operatorname{inj}(x)) \subset T_x M$ denote the geodesic ball of injectivity radius. Riemannian exponential map $\exp_x: B_x \subset T_x M \rightarrow M$ is an orientation-preserving diffeomorphism within $B_x$, hence $\left\{ \left( \exp_x, B_x \right) \right\}_{x\in M}$ is an orientation-preserving open cover of $M$. 
In particular, for any positive $j_x \leq \operatorname{inj}(x)$, the set $\left\{ \left( \exp_x, B(\vec{0},j_x) \right) \right\}_{x\in M}$ is also an orientation-preserving open cover of $M$.

Given an orientation-preserving open cover over \textit{Riemannian} manifold $M$, there exists a refinement of open cover by metric balls (geodesic balls) in $M$. This will be used to construct a finite orientation-preserving open cover of totally bounded subsets of $M$.
\begin{restatable}{lemma}{orientlemma}
\label{orientlemma}
Let $M$ be a $n$-dimensional Riemannian manifold and an orientation-preserving open cover of $\left\{\left( \rho_\alpha, U_\alpha \right) \right\}_{\alpha \in \Lambda_M}$ of $M$. There always exist refinement $\left\{\left( \rho_\beta, W_\beta \right)\right\}_{\beta \in \Lambda'_M}$ of $\left\{\left( \rho_\alpha, U_\alpha \right) \right\}_{\alpha \in \Lambda_M}$ satisfying:
\begin{enumerate}
\item $\left\{ \rho_\beta\left(W_\beta\right)\right\}_{\beta \in \Lambda'_M}$ covers $M$, and 
\item $\rho_\beta(W_\beta)$ are metric balls in $M$ for all $\beta \in \Lambda'_M$. 
\end{enumerate}
\end{restatable}

\begin{proof}

Given (orientation-preserving) open cover $\left\{ \left( \rho_\alpha, U_\alpha \right) \right\}_{\alpha \in \Lambda_M}$ of $M$. For each $\alpha$, $\forall x^\beta_\alpha \in \rho_\alpha(U_\alpha)$ there exists normal neighbourhood $N_{x^\beta_\alpha} \subset \rho_\alpha(U_\alpha)$.

Since $N_{x^\beta_\alpha}$ is open for all $x^\beta_\alpha \in \rho_\alpha(U_\alpha)$, there exists $\epsilon_{x^\beta_\alpha} > 0$ such that the metric ball centred at $x^\beta_\alpha$ denoted by $B_{x^\beta_\alpha}$ satisfies: $B_{x^\beta_\alpha} := B(x^\beta_\alpha, \epsilon_{x^\beta_\alpha}) \subset N_{x^\beta_\alpha} \subset \rho_\alpha(U_\alpha) \subset M$.

In other words, $B_{x^\beta_\alpha}$ is the metric ball centred at $x^\beta_\alpha$ in normal coordinates under the radial distance function.

Since $\rho_\alpha$ is a diffeomorphism for all $\alpha \in \Lambda_M$, this implies $\rho_\alpha^{-1}(B_{x^\beta_\alpha}) =: W_{x^\beta_\alpha} \subset U_\alpha$ for each $x^\beta_\alpha \in \rho_\alpha(U_\alpha)$. Hence $\left\{ \left( \rho_\alpha, W_{x^\beta_\alpha} \right) \right\}_{\beta \in \Lambda'_M}$ is the desired refinement of $\left\{ \left( \rho_\alpha, U_\alpha \right) \right\}_{\alpha \in \Lambda_M}$.

\end{proof}

Observe that the proof did not use the fact that $\rho_\alpha$ is orientation-preserving, as the transition map between elements of the open cover does not have to be orientation-preserving. Therefore it suffices to consider open cover $\left\{ \left( \rho_\alpha,U_\alpha \right) \right\}_{\alpha \in \Lambda_M}$ of $M$ such that $\rho_\alpha$ are just diffeomorphisms with the following result:

\begin{restatable}{lemma}{lemmamakediffeoOP}
\label{lemma:makediffeoOP}
Let $f:M\rightarrow N$ be a local diffeomorphism between manifolds $M,N$. Then there exists  local orientation-preserving diffeomorphism $\tilde{f}:M \rightarrow N$.
\end{restatable}

\begin{proof}
See \ref{app:proof:OPcover}.
\end{proof}

For the rest of the discussion we will consider the orientation-preserving open cover by metric balls $\left\{ \left( \rho_\alpha, W_{x^\beta_\alpha} \right) \right\}_{\beta \in \Lambda'_M}$ of \textit{Riemannian} manifold $M$. 





\subsection{Mixture Densities on Totally Bounded subsets of Riemannian Manifolds}
\label{sec:mixturedefn}
In this section, we construct parametrized family of mixture probability densities on totally bounded subsets of Riemannian manifold $M$. Let $V \subset M$ be totally bounded subset of $M$, and let $\left\{(\rho_\alpha, U_\alpha) \right\}_{\alpha \in \Lambda_M}$ be an orientation-preserving open cover of $M$. Let $\left\{( \rho_{x_\alpha}, W_{x^\beta_\alpha}) \right\}_{\beta \in \Lambda'_M}$ denote a refinement of $\left\{(\rho_\alpha, U_\alpha) \right\}_{\alpha \in \Lambda_M}$ by open metric balls in $M$ discussed in Lemma \ref{orientlemma}.

Since $\left\{\rho_\alpha\left( W_{x^\beta_\alpha}\right) \right\}_{\beta \in \Lambda'_M}$  is an open cover of $M$ by metric balls, it is an open cover of $V\subset M$ as well. Moreover, since $V$ is totally bounded, there exists a finite subcover $\left\{\rho_\alpha\left( W_{x^\beta_\alpha}\right) \right\}_{\beta = 1}^\Lambda$ of $V$. For simplicity, by an abuse of notation, we remove the excessive indices and rewrite the finite subcover as $\left\{ \left( \rho_\alpha, W_\alpha  \right) \right\}_{\alpha=1}^\Lambda := \left\{ \rho_\alpha\left( W_{x^\beta_\alpha}\right)\right\}_{\beta = 1}^\Lambda$.

For each $\alpha \in \left[1,\ldots, \Lambda\right]$, let $S_\alpha := \left\{\nu_\alpha := \nu_{\theta^\alpha} \vert \theta^\alpha \in \Xi_\alpha \subset \mathbb{R}^{m_\alpha} \right\} \subset \operatorname{Prob}(W_\alpha) $ denote a finitely parametrized volume form over $W_\alpha$ parametrized by $\theta^\alpha \in \Xi_\alpha  \subset \mathbb{R}^{m_\alpha}$, $m_\alpha \in \mathbb{N}_+$. Note that in general we allow the parametric families $S_\alpha$ to have different parametrizations and dimensions $m_\alpha$.
Furthermore, consider for each $\alpha$ the induced family of local probability densities $\tilde{S}_\alpha := \rho_\alpha^{-1^*}S_\alpha = \left\{ \tilde{v}_\alpha = \rho^{-1^*}_\alpha \nu_\alpha \right\} \subset \text{Prob}(M) \subset \operatorname{Vol}(M)$ on $\rho_\alpha(W_\alpha) \subset M$ (see Section \ref{Subsection:localDensity}), we can then define  parametrized mixture densities over $V \subset M$ by patching together the locally induced ones:
\begin{align}
\label{globaldistn}
\tilde{\nu} := \sum_{\alpha=1}^\Lambda \varphi_\alpha \cdot \tilde{\nu}_\alpha \quad ,
\end{align}
where $\varphi_\alpha \in \left[0,1\right]$ for all $\alpha \in \left[1,\ldots,\Lambda \right]$, and $\sum_{\alpha = 1}^\Lambda \varphi_\alpha = 1$.  Let $\overline{S}_0$ denote the closure of the simplex of mixture coefficients:
\begin{align*}
\overline{S}_0 &:= \left\{ \left\{\varphi_\alpha\right\}_{\alpha=1}^\Lambda \middle| \varphi_\alpha \in \left[0,1\right], \sum_{\alpha = 1}^\Lambda \varphi_\alpha = 1 \right\} \subset \left[0,1\right]^{\Lambda-1} 
\end{align*}
\begin{remark}
\label{rmk:s0struct}
The closure of simplex of mixture coefficients $\overline{S}_0$ is has the structure of a dually flat statistical manifold using the modified Fisher metric and Bregman  divergence pair discussed in Section \ref{sec:theo}.  We denote the set of parameters of $\overline{S}_0$ by $\Xi_0 := \left\{ \theta^0 \right\}$.
\end{remark}
We denote the set of mixture densities by $\mathcal{L}_V$ by:
\begin{align*}
\mathcal{L}_V := \left\{ \tilde{\nu} = \sum_{\alpha=1}^\Lambda \varphi_\alpha \cdot \tilde{\nu}_\alpha \, \vert \, \left\{\varphi_\alpha\right\}_{\alpha=1}^\Lambda \in \overline{S}_0, \, \tilde{\nu}_\alpha \in \tilde{S}_\alpha \right\} \quad .
\end{align*}

In local coordinates $\left(x^1,\ldots,x^n \right)$ of $M$, the mixture volume form $\tilde{\nu}_\alpha$ can be expressed as:
\begin{gather}
\label{densitypotential}
\left.\tilde{\nu}\right|_x = \tilde{p}(x,\xi)dx^1 \wedge \cdots \wedge dx^n = \sum_{\alpha=1}^\Lambda \varphi_\alpha \cdot  \tilde{p}_\alpha (x,\theta^\alpha)dx^1 \wedge \cdots \wedge dx^n \quad ,
\end{gather}
where $\tilde{p}_\alpha(x,\theta^\alpha)$ are parametrized probability density functions naturally associated to $\tilde{\nu}_\alpha$ that are parametrized by sets of parameters $\left(\theta^\alpha \right) \in \Xi_\alpha$ (see remark \ref{lazyparameter}). In local coordinates $\left(x^1,\ldots,x^n \right)$ of $M$, $\tilde{p}_\alpha(x,\theta^\alpha)$ is defined implicitly by $\tilde{p}_\alpha(x,\theta^\alpha)dx^1 \wedge \cdots \wedge dx^n := \tilde{\nu}_\alpha$ \footnote{When $M$ is a Riemannian manifold, we may choose to use the Riemannian volume form $dV_g$ instead of $dx^1 \wedge \cdots \wedge dx^n$. In which case we obtain $\tilde{p}_\alpha(x,\theta^\alpha)$ by: $\tilde{p}_\alpha(x,\theta^\alpha) dV_g := \tilde{\nu}_\alpha$ }. The parameters of the mixture distribution $\tilde{p}(x,\xi)$ are collected in $\left( \xi_i \right)_{i=1}^d := \left(\theta^0,\theta^1,\ldots,\theta^\Lambda\right) \in \Xi_0 \times \Xi_1 \times \cdots \times \Xi_\Lambda$, where $d:=\text{dim}\left(\Xi_0 \times \Xi_1 \times \cdots \times \Xi_\Lambda\right)$ denote the dimension of $\mathcal{L}_V$.

\subsubsection{Exhaustion by Compact Set and the Extent of Mixture Densities}
\label{rmk:exhaustion}
Since every smooth topological manifold is $\sigma$-locally compact, $M$ admits a compact exhaustion: 
\begin{definition} 
An \textbf{exhaustion by compact sets} is an increasing sequence of compact subsets $\overline{W}^k$ of $M$ such that $\phi \neq \overline{W}^1 \subset  \operatorname{int}(\overline{W}^2) \subset \overline{W}^2 \subseteq \cdots \subseteq M$ and:
\begin{gather*}
\lim_{k \rightarrow \infty} \overline{W}^k = \bigcup_{k=1}^\infty \overline{W}^k = M \quad .
\end{gather*}
\end{definition} 




In particular, for any totally bounded subset $V\subsetneq M$, there exists $K \in \mathbb{N}$ such that for all $k > K$:
\begin{gather*}
\overline{W^k} \,(\, \supset W^k \,) \, \supset V \quad.
\end{gather*}
Therefore even though at first glance totally bound-ness of $V$ might be quite restrictive, this shows that it allows us to approximate the manifold sufficiently well.


\subsection{Geometrical Structure of \texorpdfstring{$\mathcal{L}_V$}{Mixture densities}}
\label{Subsection:GeoLv}
In this section we describe the geometrical structure of family of mixture densities $\mathcal{L}_V$ defined over totally bounded subsets $V$ of $M$. We first show that $\mathcal{L}_V$ is a smooth manifold under two conditions. We then show that $\mathcal{L}_V$ is a product Riemannian manifold of the closure of simplex of mixture cofficients and the locally inherited families of component densities. The product dualistic Riemannian structure on $\mathcal{L}_V$ is given by a ``mixture'' divergence function on $\mathcal{L}_V$ constructed by the divergence functions on the mixture coefficients and family of component densities. Furthermore, if the families of mixture component densities are all dually flat, we show that $\mathcal{L}_V$ is also dually flat where the canonical divergence is given by the ``mixture" divergence 

We begin by considering a totally bounded subset $V \subset M$ and an orientation-preserving finite open cover $\left\{ \left( \rho_\alpha, W_\alpha \right) \right\}_{\alpha=1}^{\Lambda}$ of $V$ by metric balls.
For each $\alpha = 1,\ldots, \Lambda$, let $\left( S_\alpha, g_\alpha, \nabla^\alpha, \nabla^{\alpha^*}\right) \subset \text{Prob}(W_\alpha)$ be a family of parametrized probability densities over $W_\alpha$ with the corresponding statistical manifold structure. Let $\tilde{S}_\alpha := \rho_\alpha^{-1^*} S_\alpha$ denote the pulled-back family of probability densities on $\rho_\alpha(W_\alpha)\subset M$ with corresponding pulled-back dualistic structures $\left( \tilde{g}_\alpha, \tilde{\nabla}^\alpha, \tilde{\nabla}^{\alpha^*} \right)$ (see Section \ref{Section:InducedDualisticGeo} and \ref{Subsection:localDensity}). Let $\overline{S}_0$ denote the closure of simplex of mixture coefficients with corresponding \textit{dually flat} dualistic structure $\left(g_0,\nabla^0, \nabla^{0^*} \right)$ (see for example the dually flat dualistic structure defined by the Bregman divergence defined in Equation \eqref{eqn:b-div-clS0}). 

\subsubsection{ \texorpdfstring{$\mathcal{L}_V$}{Mixture densities} as a smooth manifold} 
\label{sec:lvsmoothmfold}

In this section we show that the family of mixture densities  $\mathcal{L}_V$ over totally bounded subsets $V$ of Riemannian manifolds $M$ is itself a smooth manifold under two natural conditions. In particular, we show that the parametrization map of $\mathcal{L}_V$ is an  injective immersion, which makes $\mathcal{L}_V$ an immersion submanifold of $\operatorname{Prob}(M)$.

For simplicity, let $\tilde{p}_\alpha(x) := \tilde{p}(x,\theta^\alpha)$ denote the probability density function corresponding to $\tilde{\nu}_\alpha \in \tilde{S}_\alpha$ for all $\alpha$. We show that $\mathcal{L}_V$ is a smooth manifold under the following two natural conditions:
\begin{condenum}
\item \textbf{Family of mixture component distributions have different proper support}:  Let
$V_\alpha := \left\{ x\in M \, \vert \, \tilde{p}_\alpha (x) > 0, \, \forall \tilde{p}_\alpha \in \tilde{S}_\alpha \right\}$ denote the proper support of probability densities $\tilde{p}_\alpha \in \tilde{S}_\alpha$ for each $\alpha \in \left\{1,\ldots, \Lambda \right\}$. We assume $V_\alpha \setminus V_\beta \neq \emptyset$ for $\beta \neq \alpha$. \label{cond:prod1}
\item \textbf{No functional dependency between mixture component densities} : We construct mixture densities in $\mathcal{L}_V$ as unconstrained mixtures, meaning there is no functional dependency between mixture component.
In other words, changing parameters $\theta^\beta \in \Xi_\beta$ of mixture component $\tilde{p}_\beta \in \tilde{S}_\beta$ has no influence on $\tilde{p}_\alpha\in \tilde{S}_\alpha$ for $\beta \neq \alpha$ and vice versa. We write this condition as follows:
For each $\tilde{p}_\alpha \in \tilde{S}_\alpha$, $\frac{\partial \tilde{p}_\alpha}{\partial \theta^\beta_k} = 0 , \quad \forall \theta^\beta_k \in \Xi^\beta \subset \mathbb{R}^{m_\beta}$, $\forall \beta \neq \alpha$.
\label{cond:prod2}
\end{condenum}

\begin{remark}
\begin{enumerate}
\item The first condition \ref{cond:prod1} can always be satisfied simply by choosing a suitable open cover of $V$.
\item The second condition \ref{cond:prod2} is automatically fulfilled for unconstrained mixture models. One can imagine introducing functional dependencies among mixture component distributions, but this is not the case considered here. We make the assumption that: if we alter one distribution $\tilde{p}_\alpha\in \tilde{S}_\alpha$, it does not affect distributions in $\tilde{S}_\beta$ for $\beta \neq \alpha$. 
\end{enumerate} 
\end{remark}

We now discuss the implications of conditions \ref{cond:prod1} and \ref{cond:prod2} in further detail:

The first condition \ref{cond:prod1} implies: the component distributions $\tilde{p}_\alpha \in \tilde{S}_\alpha$ are linearly independent functions, and the map $\theta^0 \in \Xi_0 \mapsto \left\{ \varphi_\alpha\right\}_{\alpha=1}^\Lambda \mapsto \sum_{\alpha=1}^{\Lambda} \varphi_{\alpha} \tilde{p}_{\alpha} (x,\theta^{\alpha})$ is injective \footnote{Note that the mixture component densities $\tilde{p}_{\alpha} (x,\theta^{\alpha})$ are fixed.}.

The second condition \ref{cond:prod2} implies the following:

Consider two distributions $\tilde{p}, \tilde{q} \in \mathcal{L}_V$ sharing the same mixture coefficients $\left\{\varphi_\alpha \right\}_{\alpha=1}^\Lambda$, i.e. $\tilde{p}(x) = \sum_{\alpha = 1}^\Lambda \varphi_\alpha \tilde{p}_\alpha(x)$ and $\tilde{q} = \sum_{\alpha  = 1}^\Lambda \varphi_\alpha \tilde{q}_\alpha(x)$. If $\tilde{p}(x) = \tilde{q}(x)$ for all $x\in M$:
\begin{align*}
\frac{\partial \tilde{p}(x)}{\partial \theta_i^\alpha} &= \frac{\partial \tilde{q}(x)}{\partial \theta_i^\alpha} \quad , \quad \forall \alpha, \forall i \\
\frac{\partial^\ell \tilde{p}(x)}{\left(\partial \theta_i^\alpha \right)^\ell} &= \frac{\partial^\ell \tilde{q}(x)}{\left(\partial \theta_i^\alpha \right)^\ell} \quad , \quad \forall \ell \in \mathbb{N}_+ \quad ,
\end{align*}

hence by condition \ref{cond:prod2}, for each $\alpha \in \left[1,\ldots, \Lambda \right]$, there exists a constant $c_\alpha$ such that:
\begin{align*}
\varphi_\alpha \tilde{p}_\alpha (x) = \varphi_\alpha \tilde{q}_\alpha (x) + c_\alpha ,\quad \forall x \in M \quad .
\end{align*}

Since $\tilde{p}_\alpha$ and $\tilde{q}_\alpha$ are probability densities, we have the following:
\begin{align*}
\varphi_\alpha \cdot \underbrace{\int_M \tilde{p}_\alpha}_{= 1} =  \varphi_\alpha \cdot \underbrace{\int_M \tilde{q}_\alpha}_{= 1} + \int_M c_\alpha \quad .
\end{align*}

This means $\int_M c_\alpha = c_\alpha\cdot \int_M 1 = 0$. Since $M$ is orientable we have $\int_M 1  \neq 0$, which implies $c_\alpha = 0$ and $\tilde{p}_\alpha = \tilde{q}_\alpha$ for all $\alpha$. 

Hence by injectivity of the parametrization mapping of the local component families $\theta^\alpha \mapsto \tilde{p}_\alpha$ of $\tilde{S}_\alpha$ for $\alpha \in \left[1,\ldots,\Lambda \right]$, the parametrization of mixture component parameters $(\theta^1,\ldots,\theta^\Lambda) \mapsto (\tilde{p}_1,\ldots,\tilde{p}_\Lambda)$ is injective as well. This implies the local parametrization map $\xi: \left(\theta^0,\theta^1,\ldots,\theta^\Lambda\right)\in \Xi_0\times \Xi_1 \times \cdots \times \Xi_\Lambda \mapsto \tilde{p}(x) = \sum_{\alpha=1}^\Lambda \varphi_\alpha \tilde{p}_\alpha(x)$ onto $\mathcal{L}_V$ is also injective. 


We now show that the local parametrization map $\xi: \Xi_0\times \Xi_1 \times \cdots \times \Xi_\Lambda \mapsto \mathcal{L}_V$ is an immersion. By condition \ref{cond:prod2}, the parameters are also independent in the sense that for $\beta \neq \alpha$:
\begin{align*}
0 = \frac{\partial \tilde{p}_\alpha}{\partial \theta_k^\beta} = \frac{\partial \tilde{p}_\alpha}{\partial \theta_i^\alpha} \cdot \frac{\partial \theta^\alpha_i}{\partial \theta_k^\beta} \quad \forall \tilde{p}_\alpha \in \tilde{S}_\alpha \Leftrightarrow \frac{\partial \theta^\alpha_i}{\partial \theta_k^\beta}  = 0\quad .
\end{align*}

Moreover, let $\left(\overline{\theta}_i^\alpha := \theta_i^\alpha \circ \left(\rho^{-1^*}\right)^{-1} = \theta_i^\alpha \circ \rho^* \right)$ on $\tilde{S}_\alpha$ denote the pulled-back local coordinates maps (see Section \ref{Subsection:localDensity}), then:
\begin{align*}
\frac{\partial \theta^\alpha_i\circ \rho^*}{\partial \theta^\beta_k \circ \rho^*} = \left(\rho^{-1^*}\right)_*\frac{\partial}{\partial \theta_k^\beta} \left(\theta_i^\alpha \circ \rho^* \right) = \frac{\partial}{\partial \theta_k^\beta} \left( \theta_i^\alpha \circ \rho^* \circ \rho^{-1^*} \right) = 0 \quad .
\end{align*}

In other words, pullback by $\rho$ does not introduce additional functional dependencies among parameters. 

Let $\xi_0: \theta^0 \in \Xi_0 \mapsto \left\{ \varphi_\alpha\right\}_{\alpha=1}^\Lambda \mapsto \sum_{\alpha=1}^{\Lambda} \varphi_{\alpha} \tilde{p}_{\alpha}$ denote the parametrization of $\overline{S}_0$ and let $\xi_\alpha: \tilde{\theta}^\alpha \in \Xi_\alpha \mapsto \tilde{S}_\alpha$  denote the parametrization map of $\tilde{S}_\alpha$ for $\alpha = 1,\ldots \Lambda$. The Jacobian matrix of the parametrization map $\xi: (\theta^0,\theta^1,\ldots,\theta^\Lambda) \in \Xi_0\times \Xi_1\times \cdots \Xi_\Lambda \mapsto \sum_{\alpha=1}^{\Lambda} \varphi_{\alpha} \tilde{p}_{\alpha} (x,\theta^{\alpha}) =: \tilde{p}$ is thus given by:

\[
\operatorname{Jac}_{\xi}\left(\tilde{p} \right) = \left[
\begin{array}{c|c}
 \operatorname{Jac}_{\xi_0}\left(\tilde{p} \right) & 0 \\ \hline
0 & 
 \begin {array}{ccccc} \operatorname{Jac}_{\xi_1}\left(\tilde{p} \right) & 0 & \cdots & \cdots &0 \\ 
0 & \operatorname{Jac}_{\xi_2}\left(\tilde{p} \right) & 0 & \cdots & 0 \\
\vdots & \ddots &  \ddots & \ddots &0\\ 
0 & \cdots & 0 & \operatorname{Jac}_{\xi_{\Lambda-1}}\left(\tilde{p} \right) &0 \\
\noalign{\medskip}0 & 0 &\cdots &0& \operatorname{Jac}_{\xi_\Lambda}\left(\tilde{p} \right)
\end {array}
  \\
\end{array}
\right] \quad .
\]

$\overline{S}_0$ can be endowed with a statistical manifold structure (Remark \ref{rmk:s0struct}) \footnote{This is further described in Section \ref{sec:theo}.}, and by the discussion of Section \ref{SectionProbOnMfold}, $\tilde{S}_\alpha$ are statistical manifolds for $\alpha = 1,\ldots,\Lambda$. Since parametrization maps of statistical manifolds are immersions \cite{calin2014geometric}, the submatrices $\operatorname{Jac}_{\xi_0}\left(\tilde{p} \right), \operatorname{Jac}_{\xi_1}\left(\tilde{p} \right),\ldots, \operatorname{Jac}_{\xi_\Lambda}\left(\tilde{p} \right)$ are injective, and the Jacobian matrix $\operatorname{Jac}_{\xi}\left(\tilde{p} \right)$ described above is injective as well. This means the local parametrization map $\xi: \Xi_0\times \Xi_1 \times \cdots \times \Xi_\Lambda \mapsto \mathcal{L}_V$ is an immersion.


The set of mixture densities $\mathcal{L}_V$ over totally bounded subsets $V$ of $M$ is therefore an immersion submanifold of $\operatorname{Prob}(M)$, where the mixture densities $\tilde{p} \in  \mathcal{L}_V$ are be identified by the local coordinates $\mathcal{L}_V \ni \tilde{p}:= \sum_{\alpha=1}^{\Lambda} \varphi_{\alpha} \tilde{p}_{\alpha} (x,\theta^{\alpha}) \mapsto (\theta^0,\theta^1,\ldots,\theta^\Lambda) \in \Xi_0\times \Xi_1\times \cdots \Xi_\Lambda$.	


\subsubsection{Torsion-free dualistic structure on  \texorpdfstring{$\mathcal{L}_V$}{Mixture densities}}
In this section we construct a dualistic structure on $\mathcal{L}_V$. We begin by considering the following ``mixture divergence" function on $\mathcal{L}_V \times \mathcal{L}_V$:
\begin{align}
\label{productdivergence}
D: \mathcal{L}_V \times \mathcal{L}_V &\rightarrow \mathbb{R} \nonumber \\
D\left( \tilde{p},\tilde{q} \right) = D\left( \sum_{\alpha=1}^\Lambda \varphi_\alpha p_\alpha, \sum_{\alpha=1}^\Lambda \varphi'_\alpha q_\alpha \right) &:= D_{0}\left(\left\{\varphi_\alpha\right\}_{\alpha=1}^\Lambda,\left\{\varphi'_\alpha\right\}_{\alpha=1}^\Lambda\right) + \sum_{\alpha=1}^\Lambda \overline{D}_\alpha (\tilde{p}_\alpha, \tilde{q}_\alpha)  \quad ,
\end{align}

where $D_{0}$ is a divergence  $\overline{S}_0$ (see for example the Bregman  divergence defined in Equation \eqref{eqn:b-div-clS0}), and $\overline{D}_\alpha$ is the induced  divergence on smooth manifolds $\tilde{S}_\alpha$ described by Remark \ref{Rmk:InducedDivergence} in Section \ref{Section:InducedDualisticGeo}. It is immediate by construction that $D$ satisfies the conditions of a divergence function (Definition \ref{defn:divergence}):
\begin{enumerate}
\item {\textbf{Non-negativity}}:
Since $\overline{D}_\alpha$'s and $D_{0}$ are both non-negative, so is $D$:
\begin{gather*}
D = \underbrace{D_{0}}_{\geq 0} + \sum_{\alpha=1}^\Lambda \underbrace{\overline{D}_\alpha}_{\geq 0}
\end{gather*}
\item {\textbf{Identity}}: Since $\overline{D}_\alpha$'s and $D_{0}$ are divergences, the following is satisfied:
\begin{align*}
D = 0 &\Leftrightarrow 
\begin{Bmatrix}
D_{0} (\varphi) = 0\\ 
\overline{D}_\alpha(\tilde{p}_\alpha,\tilde{q}_\alpha) = 0 \quad \forall \alpha
\end{Bmatrix} 
&\Leftrightarrow \begin{Bmatrix}
\varphi_\alpha = \varphi'_\alpha \\ 
\tilde{p}_\alpha = \tilde{q}_\alpha 
\end{Bmatrix} \, \forall \alpha \Leftrightarrow \tilde{p} = \tilde{q} \quad .
\end{align*}
\end{enumerate}
Given a divergence function on a manifold, we can construct the corresponding (torsion-free) dualistic Riemannian structure on the manifold \cite{Amari2000}. In particular, the torsion-free dualistic structure on $\mathcal{L}_V$ induced by the mixture divergence $D$ (Equation \eqref{productdivergence}) is derived in the following proposition.
\begin{theorem}
\label{prop:lvdstruct}
The torsion-free dualistic structure $\left\{ g^D, \nabla^D, \nabla^{D^*} \right\}$  of $\mathcal{L}_V$ generated by $D$ (Equation \eqref{productdivergence}), is given by:
\begin{gather*}
\left\{g^{D_{0}} \oplus \bigoplus_{\alpha =1}^\Lambda g^{\overline{D}_\alpha}, \nabla^{D_{0}} \oplus \bigoplus_{\alpha=1}^\Lambda \nabla^{\overline{D}_\alpha}, \nabla^{D_{0}^*} \oplus \bigoplus_{\alpha=1}^\Lambda  \nabla^{\overline{D}_\alpha^*} \right\} \quad .
\end{gather*}
\end{theorem}
\begin{proof}
Let $\theta^\alpha \in \Xi^\alpha \subset \mathbb{R}^{m_\alpha}$ be coordinates of $S_\alpha$, and let $\left(\xi_i\right)_{i=1}^d := (\theta^0,\tilde{\theta}^1,\ldots,\tilde{\theta}^\Lambda)$ denote the coordinates of $\mathcal{L}_V$, where $d:=\text{dim}\left(\Xi_0 \times \Xi_1 \times \cdots \times \Xi_\Lambda\right)$ and $\left( \tilde{\theta}^\alpha_j := \theta^\alpha_j \circ \rho_\alpha^{*} \right)_{j=1}^{m_\alpha}$ denote the local pulled-back coordinates of $\tilde{S}_\alpha$ discussed in section \ref{Subsection:computingInducedDualisticGeo} (see also Remark \ref{Rmk:pullbackcoord} and \ref{lazyparameter}). Let the local coordinate frame correspond to $\left(\xi_i\right)_{i=1}^d$ be denoted by $\left(\overline{\partial}_i := \frac{\partial}{\partial \xi_i}\right)$. We can then construct a dualistic structure $\left(g^D, \nabla^D, \nabla^{D^*} \right)$ on $\mathcal{L}_V$ \cite{eguchi1992geometry,Amari2000} corresponding to the  divergence function $D$ (defined in Equation \eqref{productdivergence}). The matrix $\left[g_{ij}^D\right]$ corresponding to the induced Riemannian metric $g^D$ is given by:
\begin{align}
\label{gd}
g^D_{ij}(\tilde{p}) := \left. g^D\left( \overline{\partial}_i, \overline{\partial}_j \right)\right|_{\tilde{p}} &= \left.-D[\overline{\partial}_i ; \overline{\partial}_j] \right|_{\tilde{p}} \nonumber \\
&= \left. -D_{0}[\overline{\partial}_i ; \overline{\partial}_j] \right|_{\tilde{p}} + \left. \sum_{\alpha=1}^\Lambda -\overline{D}_\alpha [\overline{\partial}_i ; \overline{\partial}_j]\right|_{\tilde{p}} \nonumber \\
&= \left. g^{D_{0}}\left(\overline{\partial}_i,\overline{\partial}_j\right)\right|_{\tilde{p}} + \sum_{\alpha=1}^\Lambda \left. g^{\overline{D}_\alpha}\left(\overline{\partial}_i,\overline{\partial}_j\right) \right|_{\tilde{p}}\quad .
\end{align}

We consider the following cases:
\begin{case}
If both the input of Equation \eqref{gd} $\left(\overline{\partial}_i, \overline{\partial}_i\right)$ belong to $\mathcal{E}\left(T\overline{S}_0\right)$: i.e. if $\overline{\partial}_i = \frac{\partial}{\partial \theta^0_i}$ and $\overline{\partial}_j = \frac{\partial}{\partial \theta^0_j} $ then:
\begin{align}
\label{gdkl}
\left. g^D\left( \overline{\partial}_i, \overline{\partial}_j \right)\right|_{\tilde{p}} &= \left. g^{D_{0}}\left(\overline{\partial}_i,\overline{\partial}_j\right)\right|_{\tilde{p}} + 0  \nonumber \\
& =  \left. g^{D_{0}}\left(\overline{\partial}_i,\overline{\partial}_j\right) \right|_{\left\{\varphi_\alpha \right\}_{\alpha=1}^\Lambda} \quad ,
\end{align}
where $\left\{\varphi_\alpha\right\}_{\alpha=1}^\Lambda \in \overline{S}_0$. 
\end{case}
\begin{case}
If both the input of Equation \eqref{gd} $\left(\overline{\partial}_i, \overline{\partial}_i\right)$ belong to the \textit{same} $\mathcal{E}\left(T\tilde{S}_\beta\right)$: i.e.  If $\overline{\partial}_i = \tilde{\partial}_i^\beta = \frac{\partial}{\partial \tilde{\theta}^\beta_i}$ and $\overline{\partial}_j = \tilde{\partial}_j^\beta = \frac{\partial}{\partial \tilde{\theta}^\beta_j}$, where $\left(\tilde{\theta}^\beta_j\right)_{j=1}^{m_\beta}$ denote the local parametrization of $\tilde{S}_\beta$ for some $\beta \in \left\{1,\ldots,\Lambda\right\}$. Then by condition \ref{cond:prod2}:
\begin{align}
\label{gdalpha}
\left. g^D \left( \overline{\partial}_i, \overline{\partial}_j \right)\right|_{\tilde{p}} &= 0 + \sum_{\alpha=1}^\Lambda \left. g^{\overline{D}_\alpha}\left(\tilde{\partial}^\beta_i,\tilde{\partial}^\beta_j\right) \right|_{\tilde{p}_\alpha} \nonumber \\
&= 0 + \sum_{\alpha=\beta} \left. g^{\overline{D}_\alpha} \left(\tilde{\partial}^\beta_i,\tilde{\partial}^\beta_j\right) \right|_{\tilde{p}_\alpha} \nonumber \\
&= \left. g^{\overline{D}_\beta} \left(\tilde{\partial}^\beta_i,\tilde{\partial}^\beta_j\right)\right|_{\tilde{p}_\beta} = \left. g^\beta (\partial^\beta_i,\partial^\beta_j) \right|_{p_\beta} \quad ,
\end{align}
where $g^\beta$ is the metric on $S_\beta$, and $p_\beta := \left(\rho^{-1^*} \right)^{-1} \tilde{p}_\beta = \rho^* \tilde{p}_\beta$. The last equality is due to the analysis towards the end of Section \ref{Section:InducedDualisticGeo}. 
\end{case}
\begin{case}If the first input of Equation \eqref{gd} belongs to $\mathcal{E}\left(T\overline{S}_0\right)$ and the second input belongs to $\mathcal{E}\left(T\tilde{S}_\beta\right)$ (or vice versa by symmetry): i.e. If $\overline{\partial}_i = \frac{\partial}{\partial \theta^0_i}$ and $\overline{\partial}_i = \tilde{\partial}_j^\beta = \frac{\partial}{\partial \tilde{\theta}^\beta_j}$ where $\left(\tilde{\theta}^\beta_j\right)_{j=1}^{m_\beta}$ denote the local parametrization of $\tilde{S}_\beta$ for some $\beta \in \left\{1,\ldots,\Lambda\right\}$, then
\begin{align*}
\left. g^D \left( \overline{\partial}_i, \overline{\partial}_j \right)\right|_{\tilde{p}} &= \left. g^{D_{0}}\left(\overline{\partial}_i,0 \right) \right|_{\left\{\varphi_\alpha \right\}_{\alpha=1}^\Lambda} + \left. g^{\overline{D}_\beta} \left(0,\tilde{\partial}^\beta_j\right)\right|_{\tilde{p}_\beta} \\
&= 0+ 0 = 0 \quad .
\end{align*}

\begin{case} Otherwise if the first and second input of Equation \eqref{gd} belongs to \textit{different}  $\mathcal{E}\left(T\tilde{S}_\alpha\right)$ and $\mathcal{E}\left(T\tilde{S}_\beta\right)$: i.e. If $\overline{\partial}_i = \tilde{\partial}_i^\alpha = \frac{\partial}{\partial \tilde{\theta}^\alpha_i}$ and $\overline{\partial}_j = \tilde{\partial}_j^\beta = \frac{\partial}{\partial \tilde{\theta}^\beta_j}$, where  $\left(\tilde{\theta}^\alpha_j\right)_{j=1}^{m_\alpha}$ and $\left(\tilde{\theta}^\beta_j\right)_{j=1}^{m_\beta}$ denote the local parametrization of $\tilde{S}_\alpha$ and $\tilde{S}_\beta$ respectively for some $\alpha \neq\beta \in \left\{1,\ldots,\Lambda\right\}$. Equation \eqref{gd} then becomes:
\begin{align*}
\left. g^D \left( \overline{\partial}_i, \overline{\partial}_j \right)\right|_{\tilde{p}} &= 0 +
\left. g^{\overline{D}_\alpha} \left(\tilde{\partial}^\alpha_i,0 \right)\right|_{\tilde{p}_\alpha} + \left. g^{\overline{D}_\beta} \left(0,\tilde{\partial}^\beta_j\right)\right|_{\tilde{p}_\beta} \\
&= 0+ 0 + 0 = 0 \quad .
\end{align*}
\end{case}
\end{case}

On the other hand, the Christoffel symbols of the connection $\nabla^D$ induced by $D$ are given by:
\begin{align}
\label{nablad}
\left. \Gamma^D_{ijk} \right|_{\tilde{p}} := \left. \langle \nabla^D_{\overline{\partial}_i} \overline{\partial}_j,\overline{\partial}_k \rangle_{g^D} \right|_{\tilde{p}} &= \left. - D \left[ \overline{\partial}_i \overline{\partial}_j ; \overline{\partial}_k \right]\right|_{\tilde{p}} \nonumber \\ 
&= \left.- D_{0}[ \overline{\partial}_i \overline{\partial}_j ; \overline{\partial}_k ] \right|_{\tilde{p}} + \sum_{\alpha=1}^\Lambda \left.- \overline{D}_\alpha[ \overline{\partial}_i \overline{\partial}_j ; \overline{\partial}_k ]\right|_{\tilde{p}} \nonumber \\
&= \left. \langle \nabla^{D_{0}}_{\overline{\partial}_i} \overline{\partial}_j,\overline{\partial}_k \rangle_{g^D}\right|_{\tilde{p}} + \sum_{\alpha=1}^\Lambda \left. \langle \nabla^{\overline{D}_\alpha}_{\overline{\partial}_i} \overline{\partial}_j,\overline{\partial}_k \rangle_{g^D} \right|_{\tilde{p}} \quad  ,
\end{align}
where $- D[ \partial_i \partial_j ; \partial_k ] := - \partial^1_i \partial^1_j \partial^2_k  \left. D[p;q] \right|_{q=p}$ (by Equation (\ref{Dabusenotation}) in Section \ref{Section:InducedDualisticGeo}).
By the similar  four case argument in the induced metric $g^D$ derivation above, we obtain the following \footnote{The third line of Equation \eqref{nabladsplit} encapsulates the last two cases of $g^D$.}:
\begin{equation}
\label{nabladsplit}
\left. \langle \nabla^D_{\overline{\partial}_i} \overline{\partial}_j,\overline{\partial}_k \rangle_{g^D}\right|_{\tilde{p}} = 
\left\{
\begin{array}{ll}
\left. \langle \nabla^{D_{0}}_{\overline{\partial}_i} \overline{\partial}_j,\overline{\partial}_k \rangle_{g^{D_{0}}} \right|_{\left\{\varphi_\alpha\right\}_{\alpha=1}^\Lambda} &\quad \text{if} \quad \overline{\partial}_\ell = \frac{\partial}{\partial \theta^0_\ell} \quad , \text{ for } \ell = i,j,k \quad, \\
\left. \langle \nabla^{\overline{D}_\beta}_{\overline{\partial}_i} \overline{\partial}_j,\overline{\partial}_k \rangle_{g^{D_\beta}}\right|_{p_\beta} &\quad \text{if} \quad \overline{\partial}_\ell = \tilde{\partial}_\ell^\beta = \frac{\partial}{\partial \tilde{\theta}^\beta_\ell} \quad , \text{ for } \ell = i,j,k \quad,\\
0 & \quad \text{Otherwise.}
\end{array}
\right.
\end{equation}
The above result implies that for vector fields $X_0,Y_0 \in \mathcal{E}\left(T\overline{S}_0\right)$, $X_\alpha,Y_\alpha \in \mathcal{E}\left(T\tilde{S}_\alpha\right)$, where $\alpha \in \left[1,\ldots,\Lambda \right]$, we have the following:
\begin{align}
\label{nabladproductform}
\nabla^D_{\left(Y_0 + \sum_{\alpha = 1}^\Lambda Y_\alpha\right)} \left(X_0 + \sum_{\alpha = 1}^\Lambda X_\alpha \right)&= \nabla^D \left( X_0 + \sum_{\alpha = 1}^\Lambda X_\alpha, Y_0 + \sum_{\alpha = 1}^\Lambda Y_\alpha \right) \nonumber \\
&=\nabla^{D_{0}}_{Y_0} X_0 + \sum_{\alpha=1}^\Lambda \nabla^{\overline{D}_\alpha}_{Y_\alpha} X_\alpha \quad .
\end{align}

By symmetry and the fact that $g^{D} = g^{D^*}$ \cite{Amari2000}, we have the following for the induced dual connection $\nabla^{D^*}$:
\begin{align}
\label{nabla*d}
\left. \langle \nabla^{D^*}_{\overline{\partial}_i} \overline{\partial}_j,\overline{\partial}_k \rangle_{g^D} \right|_{\tilde{p}} = \left.\langle \nabla^{D_{0}}_{\overline{\partial}_i} \overline{\partial}_j,\overline{\partial}_k \rangle_{g^D} \right|_{\tilde{p}} + \sum_{\alpha=1}^\Lambda \left.\langle \nabla^{\overline{D}_\alpha}_{\overline{\partial}_i} \overline{\partial}_j,\overline{\partial}_k \rangle_{g^D} \right|_{\tilde{p}} \quad .
\end{align}
Furthermore, we obtain the following result analogous to equation (\ref{nabladsplit}):
\begin{equation}
\label{nabla*dsplit}
\left. \langle \nabla^{D^*}_{\overline{\partial}_i} \overline{\partial}_j,\overline{\partial}_k \rangle_{g^D}\right|_{\tilde{p}} = 
\left\{
\begin{array}{ll}
\left. \langle \nabla^{D^*_{0}}_{\overline{\partial}_i} \overline{\partial}_j,\overline{\partial}_k \rangle_{g^{D_{0}}} \right|_{\left\{\varphi_\alpha\right\}_{\alpha=1}^\Lambda} &\quad \text{if} \quad \overline{\partial}_\ell = \frac{\partial}{\partial \theta^0_\ell} \quad , \text{ for } \ell = i,j,k \quad \\
\left. \langle \nabla^{\overline{D}^*_\beta}_{\overline{\partial}_i} \overline{\partial}_j,\overline{\partial}_k \rangle_{g^{\overline{D}_\beta}} \right|_{p_\beta} &\quad \text{if} \quad \overline{\partial}_\ell = \tilde{\partial}_\ell^\beta = \frac{\partial}{\partial \tilde{\theta}^\beta_\ell} \quad , \text{ for } \ell = i,j,k \quad, \\
0 & \quad \text{Otherwise.}
\end{array}
\right.
\end{equation}
Finally, for $X_0,Y_0 \in \mathcal{E}\left(T\overline{S}_0\right)$, $X_\alpha,Y_\alpha \in \mathcal{E}\left(T\tilde{S}_\alpha\right)$ for $\alpha \in \left[1,\ldots,\Lambda \right]$, we obtain:
\begin{align}
\label{nabla*dproductform}
\nabla^{D^*}_{\left(Y_0 + \sum_{\alpha = 1}^\Lambda Y_\alpha\right)} \left(X_0 + \sum_{\alpha = 1}^\Lambda X_\alpha\right) &= \nabla^{D^*} \left( X_0 + \sum_{\alpha = 1}^\Lambda X_\alpha, Y_0 + \sum_{\alpha = 1}^\Lambda Y_\alpha \right) \nonumber \\
&=\nabla^{D^*_{0}}_{Y_0} X_0 + \sum_{\alpha=1}^\Lambda \nabla^{\overline{D}^*_\alpha}_{Y_\alpha} X_\alpha \quad .
\end{align}

Therefore by Equations \eqref{gd}, \eqref{nabladproductform}, \eqref{nabla*dproductform}, the dualistic structure $\left\{ g^D, \nabla^D, \nabla^{D^*} \right\}$ induced by $D$ naturally decomposes into the parts corresponding to the mixture coefficients and mixture components. We abbreviate Equations \eqref{gd}, \eqref{nabladproductform}, \eqref{nabla*dproductform} by the following compact form:
\begin{gather}
\label{Eqn:productdualisticdivergenceform}
\left\{g^{D_{0}} \oplus \bigoplus_{\alpha =1}^\Lambda g^{\overline{D}_\alpha}, \nabla^{D_{0}} \oplus \bigoplus_{\alpha=1}^\Lambda \nabla^{\overline{D}_\alpha}, \nabla^{D_{0}^*} \oplus \bigoplus_{\alpha=1}^\Lambda  \nabla^{\overline{D}_\alpha^*} \right\} \quad .
\end{gather}
\end{proof}

\subsection{ \texorpdfstring{$\mathcal{L}_V$}{Mixture densities} as a Product Statistical Manifold}
\label{Subsection:DualGeoLv}
In this section we show that  $\mathcal{L}_V$ is a product Riemannian manifold of the statistical manifolds $\overline{S}_0, \tilde{S}_1,\ldots,\tilde{S}_\Lambda$. We first recall some properties of product Riemannian manifolds from \cite{sakai1996riemannian,do1992riemannian,lee2006riemannian}. 

\begin{remark}
Note that since $\mathcal{L}_V$ consists of mixture densities with finite number of mixtures, to show that $\mathcal{L}_V = \overline{S}_0 \times \tilde{S}_1 \times \cdots \tilde{S}_\Lambda$, it suffices to consider the dualistic structure of the product of two manifolds. 
\end{remark}

Given two Riemannian manifolds $(M,g_1), (N,g_2)$, and points $x_1\in M$ and $x_2\in N$. The tangent spaces of $M \times N$ can be expressed as: $T_{\left(x_1,x_2\right)} = T_{x_1} M \oplus T_{x_2} N$. The product Riemannian metric on $M\times N$ is given by $g:= g_1\oplus g_2$ \cite{lee2006riemannian}:
\begin{gather*}
g \left(\left.X_1\right|_{x_1}+\left.X_2\right|_{x_2},\left.Y_1\right|_{x_1} +\left.Y_2\right|_{x_2}\right)|_{(x_1,x_2} = g_1\left(\left.X_1\right|_{x_1},\left.Y_1\right|_{x_1}\right)|_{x_1} + g_2\left(\left.X_2\right|_{x_2},\left.Y_2\right|_{x_2}\right)|_{x_1} \quad,
\end{gather*}
where $\left.X_1\right|_{x_1},\left.Y_1\right|_{x_1} \in T_{x_1} M$ and $\left.X_2\right|_{x_2},\left.Y_2\right|_{x_2} \in T_{x_2} N$. For the rest of the discussion, when the context is clear, we abbreviate the above equation by:
\begin{align}
\label{eqn:productg}
    g\left(X_1 + X_2, Y_1 + Y_2 \right) = g_1\left(X_1,Y_1 \right) + g_2\left(X_2, Y_2 \right) \quad ,
\end{align}
where $X_1,Y_1 \in \mathcal{E}\left(TM\right)$ and $X_2,Y_2 \in \mathcal{E}\left(T N\right)$. Furthermore, suppose $\nabla^1$ and $\nabla^2$ are connections of $M,N$ respectively, then the product connection is given by \cite{do1992riemannian}:
\begin{gather}
\label{productconnection}
\nabla_{Y_1+ Y_2} X_1 + X_2 = \nabla^1_{Y_1} X_1 + \nabla^2_{Y_2}X_2 \quad,
\end{gather}
where $X_1,Y_1 \in \mathcal{E}\left(TM\right)$ and $X_2,Y_2 \in \mathcal{E}\left(T N\right)$. We abbreviate the product connection to a more compact notation for simplicity: $ \nabla = \nabla^1 \oplus \nabla^2$. Since the Lie bracket of $M\times N$ is given by:
\begin{gather*}
\left[X_1 + X_2, Y_1+ Y_2 \right]_{M\times N} = \left[X_1 , Y_1 \right]_M + \left[X_2,Y_2 \right]_N \quad ,
\end{gather*}
and the curvature tensor on manifolds given by:
\begin{gather*}
R(X,Y)Z = \nabla_X \nabla_Y Z - \nabla_Y \nabla_X Z - \nabla_{\left[X,Y \right]} Z \quad ,
\end{gather*}
we therefore obtain the curvature tensor on the $M\times N$ as follows:
\begin{align}
\label{productcurvature}
R(X_1+X_2,Y_1+Y_2,Z_1+Z_2,W_1+W_2) = R_1(X_1,Y_1,Z_1,W_1) + R_2(X_2,Y_2,Z_2,W_2) \quad ,
\end{align}
where $X_1,Y_1,Z_1,W_1 \in \mathcal{E}\left(TM\right)$ and $X_2,Y_2,Z_2,W_2 \in \mathcal{E}\left(T N\right)$, and $R_1, R_2$ denote the curvature tensor of $M,N$ respectively. Hence if $M$ and $N$ are flat, so is $M\times N$.

Therefore to show that the Riemannian structure derived from the divergence $D$ in equation (\ref{productdivergence}) (given by equation (\ref{Eqn:productdualisticdivergenceform})) coincides with the product Riemannian structure discussed above, it suffices to show the following result.

\begin{theorem}
\label{thm:productdual}
Let $(M,g_1,\nabla^1,\nabla^{1^*})$, $(N,g_2,\nabla^2,\nabla^{2^*})$ be two smooth manifolds with their corresponding dualistic structure, and consider the product manifold $M\times N$ with product metric $g = g_1 \oplus g_2$ and product connection $\nabla := \nabla^1 \oplus \nabla^2$. Then the connection $\nabla^{1^*} \oplus \nabla^{2^*}$ is $g$-dual to $\nabla$. In particular $\left(\nabla^1 \oplus \nabla^2\right)^* = \nabla^{1^*} \oplus \nabla^{2^*} $.
Furthermore, if $M$ and $N$ are dually flat, then so is $M \times N$.
\end{theorem}

\begin{proof}
Let $(M,g_1,\nabla^1,\nabla^{1^*})$, $(N,g_2,\nabla^2,\nabla^{2^*})$ be two smooth manifolds with their corresponding dualistic structure. Let $\nabla = \nabla^1 \oplus \nabla^2$ denote the product connection on $M \times N$ given by equation (\ref{productconnection}) in compact notation. \footnote{Please refer to comments following equation (\ref{productconnection})} For simplicity, let $\langle \cdot,\cdot\rangle := \langle \cdot ,\cdot\rangle_g$, where $g = g_1 \oplus g_2$ denote the product Riemannian metric on $M \times N$.

Let $X_1,Y_1,Z_1 \in \mathcal{E}\left(TM\right)$ and $X_2,Y_2,Z_2 \in \mathcal{E}\left(T N\right)$. Given points $x_1 \in M$ and $x_2 \in N$, using the natural identification $T_{\left(x_1,x_2\right)} = T_{x_1} M \oplus T_{x_2} N$, we let: 
\begin{align*}
\left\{X,Y,Z\right\} := \left\{ X_1 + X_2, Y_1 + Y_2, Z_1 + Z_2 \right\} \in \mathcal{E}\left( T\left(M \times N\right) \right)
\end{align*}
Then by Equation \eqref{eqn:productg} we have:
\begin{align}
\label{lhs1}
\nabla_X \langle Y,Z \rangle = X\langle Y,Z \rangle &= (X_1 + X_2) \langle Y_1 + Y_2, Z_1 + Z_2 \rangle \\
&= X_1 \left( \langle Y_1,Z_1 \rangle + \langle Y_1, Z_2 \rangle + \langle Y_2,Z_1\rangle + \langle Y_2,Z_2\rangle \right) \nonumber \\ 
\, & \quad+ X_2 \left( \langle Y_1,Z_1 \rangle + \langle Y_1, Z_2 \rangle + \langle Y_2,Z_1\rangle + \langle Y_2,Z_2 \rangle \right) \quad ,
\end{align}
where $X_1,Y_1,Z_1 \in \mathcal{E}\left( TM \right)$, $X_2,Y_2,Z_2 \in \mathcal{E}\left(T N \right)$. 

Let $\pi_{*_1}: T(M\times N) \rightarrow TM$ and $\pi_{*_2}: T(M\times N) \rightarrow TN$ denote the natural projections, then we have:
\begin{align}
\label{zeroprojectinproduct}
\langle Y_1, Z_2 \rangle &= \langle Y_1, \pi_{*_1} Z_2 \rangle_{g_1} + \langle \pi_{*_2} Y_1 , Z_2 \rangle_{g_2} \nonumber \\
&= \langle Y_1, 0 \rangle_{g_1} + \langle 0 , Z_2 \rangle_{g_2} = 0 \quad , 
\end{align}
and by symmetry we obtain $\langle Y_2, Z_1 \rangle = 0 $ as well. 

Since $\langle Y_1,Z_1 \rangle$ is a function on $M$, we have $ \nabla^2_{X_2} \langle Y_1,Z_1 \rangle = 0$. By symmetry we obtain $\nabla^1_{X_1} \langle Y_2,Z_2 \rangle = 0$ as well. Equation \eqref{lhs1} therefore becomes:
\begin{align}
\label{lhs2}
&X_1\left( \langle Y_1,Z_1 \rangle + \langle Y_2,Z_2 \rangle\right) + X_2 \left( \langle Y_1,Z_1 \rangle + \langle Y_2,Z_2 \rangle \right) \nonumber \\
&= \nabla^1_{X_1} \langle Y_1,Z_1 \rangle + \nabla^1_{X_1} \langle Y_2,Z_2 \rangle  + \nabla^2_{X_2} \langle Y_1,Z_1 \rangle + \nabla^2_{X_2} \langle Y_2,Z_2 \rangle \nonumber \\
&= \nabla^1_{X_1} \langle Y_1,Z_1 \rangle + 0 + 0 + \nabla^2_{X_2} \langle Y_2,Z_2 \rangle \nonumber \\
&= \nabla^1_{X_1} \langle Y_1,Z_1 \rangle + \nabla^2_{X_2} \langle Y_2,Z_2 \rangle \quad .
\end{align}
By a similar argument, we obtain from Equation \eqref{productconnection} and Equation \eqref{eqn:productg}:
\begin{align}
\label{rhs1}
\langle \nabla_X Y,Z \rangle &= \langle \nabla^1_{X_1} Y_1,Z_1 + Z_2 \rangle + \langle \nabla^2_{X_2} Y_2,Z_ 1+ Z_2 \rangle \nonumber \\
&= \langle \nabla^1_{X_1} Y_1,Z_1 \rangle + \langle \nabla^1_{X_1} Y_1,Z_2 \rangle  + \langle \nabla^2_{X_2} Y_2,Z_1 \rangle + \langle \nabla^2_{X_2} Y_2,Z_2 \rangle \nonumber \\
& = \langle \nabla^1_{X_1} Y_1,Z_1 \rangle + \langle \nabla^2_{X_2} Y_2,Z_2 \rangle \quad .
\end{align}

The first equality is due to Equation \eqref{productconnection}, and the fact that $\nabla$ is a product connection. The last equality is due to the same argument as in  Equation \eqref{zeroprojectinproduct}: since $\pi_{*_i} Z_j = 0$, and $\pi_{*_i}\nabla^j_{X_j} Y_j = 0$ for $i\neq j$, we have $\langle \nabla^1_{X_1} Y_1,Z_2 \rangle = \langle \nabla^2_{X_2} Y_2,Z_1 \rangle = 0$.

Finally, subtracting equation (\ref{rhs1}) from equation (\ref{lhs2}), we have the following equality:
\begin{align*}
X\langle Y,Z \rangle - \langle \nabla_X Y,Z \rangle &= \nabla^1_{X_1} \langle Y_1,Z_1 \rangle + \nabla^2_{X_2} \langle Y_2,Z_2 \rangle - \left(\langle \nabla^1_{X_1} Y_1,Z_1 \rangle + \langle \nabla^2_{X_2} Y_2,Z_2 \rangle \right) \\
& = \left(\nabla^1_{X_1} \langle Y_1,Z_1 \rangle -  \langle \nabla^1_{X_1} Y_1,Z_1 \rangle\right) + \left( \nabla^2_{X_2} \langle Y_2,Z_2 \rangle - \langle \nabla^2_{X_2} Y_2,Z_2 \rangle \right)\\
&= \langle Y_1, \nabla^{1^*}_{X_1} Z_1 \rangle  + \langle Y_2, \nabla^{2^*}_{X_2} Z_2 \rangle \quad ,
\end{align*}

where $\nabla^{1^*},\nabla^{2^*}$ denote the $g_1,g_2$-dual connection to $\nabla^1,\nabla^2$ on $M,N$ respectively. The unique \cite{nomizu1994affine} $g$-dual connection to $\nabla = \nabla^1 \oplus \nabla^2$ of $M \times N$, denoted by  $\nabla^*:= \nabla{1^*} \oplus \nabla^{2^*}$, is thus given by the following:
\begin{align*}
\nabla^*_{Y_1 + Y_2} X_1 + X_2 =  \nabla^{1^*}_{Y_1} X_1 + \nabla^{2^*}_{Y_2} X_2 \quad ,
\end{align*}
where $X_1,Y_1 \in \mathcal{E}\left(TM\right)$ and $X_2,Y_2 \in \mathcal{E}\left(T N\right)$.

Furthermore, since the curvature of $M \times N$ satisfies the product curvature tensor described in equation (\ref{productcurvature}), if $(M,g_1,\nabla^1,\nabla^{1^*}), (N,g_2,\nabla^2,\nabla^{2^*})$ are both dually flat, then so is their product $(M\times N,g_1+g_2,\nabla^1 \oplus \nabla^2, \nabla^{1^*} \oplus \nabla^{2^*})$.
\end{proof}

\begin{remark}
\label{rmk:lvisprod}
By Proposition \ref{prop:lvdstruct} and Theorem \ref{thm:productdual}, the family of parametrized mixture densities $\mathcal{L}_V = \overline{S}_0 \times \tilde{S}_1 \times \cdots \tilde{S}_\Lambda = \overline{S}_0 \times \bigoplus_{\alpha =1} ^ \Lambda \tilde{S}_\alpha$ is therefore a product manifold with product dualistic structure:
\begin{align*}
&\left\{g_0 \oplus \bigoplus_{\alpha =1}^\Lambda \tilde{g}_\alpha, \nabla^0 \oplus \bigoplus_{\alpha=1}^\Lambda \tilde{\nabla}^\alpha, \nabla^{0^*} \oplus \bigoplus_{\alpha=1}^\Lambda  \tilde{\nabla}^{\alpha^*} \right\} \\
& \quad =  \left\{g^{D_{0}} \oplus \bigoplus_{\alpha =1}^\Lambda g^{\overline{D}_\alpha}, \nabla^{D_{0}} \oplus \bigoplus_{\alpha=1}^\Lambda \nabla^{\overline{D}_\alpha}, \nabla^{D_{0}^*} \oplus \bigoplus_{\alpha=1}^\Lambda  \nabla^{\overline{D}_\alpha^*} \right\} \quad .
\end{align*}

The closure of simplex of mixture coefficients $\overline{S}_0$ can be endowed with a dually flat dualistic structure (Remark \ref{rmk:s0struct} and Section \ref{sec:theo}). If the mixture component families $S_1,\ldots, S_\Lambda$ on orientation-preserving open cover $\left\{\left(\rho_\alpha,W_\alpha\right) \right\}_{\alpha=1}^\Lambda$ over totally bounded $V\subset M$ are all dually flat, then so is $\tilde{S}_1,\ldots, \tilde{S}_\Lambda$ (See Corollary \ref{df}). In particular, by Theorem \ref{thm:productdual}, this implies $\mathcal{L}_V = \overline{S}_0 \times \tilde{S}_1 \times \cdots \tilde{S}_\Lambda$ is also dually flat.
\end{remark}

Finally, consider $\left( \overline{S}_0, g_0, \nabla^0,\nabla^{0^*} \right),\left( \tilde{S}_1, \tilde{g}_1, \tilde{\nabla}^1,\tilde{\nabla}^{1^*} \right),\ldots, \left(\tilde{S}_\Lambda, g_\Lambda, \tilde{\nabla}^\Lambda,\tilde{\nabla}^{\Lambda^*} \right)$ as dually flat manifolds with their corresponding dualistic structures. \footnote{Please refer to the beginning of this section (\ref{Subsection:GeoLv}) for detailed discussion of the induced dualistic structure on $\tilde{S}_\alpha$ for $\alpha \in \left[1,\ldots,\Lambda\right]$.} We show that the ``mixture" divergence defined in Equation \eqref{productdivergence} is in fact the canonical divergence \cite{Amari2000} of dually flat manifold $\mathcal{L}_V$.

\begin{theorem}
If $\overline{S}_0,S_1,\ldots,S_\Lambda$ are all dually flat, the divergence function $D$ defined in Equation \eqref{productdivergence} is the canonical divergence of $\mathcal{L}_V = \overline{S}_0 \times \tilde{S}_1 \times \cdots \tilde{S}_\Lambda$ with respect to product dually flat dualistic structure:
\begin{gather*}
\left\{g_0 \oplus \bigoplus_{\alpha =1}^\Lambda \tilde{g}_\alpha, \nabla^0 \oplus \bigoplus_{\alpha=1}^\Lambda \tilde{\nabla}^\alpha, \nabla^{0^*} \oplus \bigoplus_{\alpha=1}^\Lambda  \tilde{\nabla}^{\alpha^*} \right\} \quad .
\end{gather*}
\end{theorem}
\begin{proof}
Let $\left( \theta^0,\eta_0 \right)$ denote the local $g_0$-dual coordinates of $\overline{S}_0$, and let $\left(\tilde{\theta}^\alpha\right)$ denote the local $\tilde{\nabla}^\alpha$-affine coordinates of $\tilde{S}_\alpha$ for $\alpha \in \left[1,\ldots, \Lambda\right]$. 

By the discussion in Section \ref{Section:InducedDualisticGeo}, let $\tilde{\psi}_\alpha$ denote the pulled-back potential function on $\tilde{S}_\alpha$ defined by local coordinates $\left(\tilde{\theta}^\alpha\right)$  and pulled-back metric $\tilde{g}_\alpha$ on $\tilde{S}_\alpha$. The $\tilde{g}_\alpha$-dual local coordinates to $\left(\tilde{\theta}^\alpha\right)$ can be defined via induced potential function $\tilde{\psi}_\alpha$ by: $\left( \tilde{\eta}_\alpha^i := \frac{\partial}{\partial \theta^\alpha_i} \tilde{\psi}_\alpha \right)$.

Since $\overline{S}_0$ and $\tilde{S}_\alpha$ are all dually flat for $\alpha \in \left[1,\ldots, \Lambda \right]$, we can write the divergences $D_{0}$ and $\overline{D}_\alpha$ of $\overline{S}_0$ and $\tilde{S}_\alpha$ in the canonical form \cite{Amari2000} respectively as follows:
\begin{align*}
D_{0}\left(\left\{\varphi_\alpha\right\}_{\alpha=1}^\Lambda,\left\{\varphi'_\alpha\right\}_{\alpha=1}^\Lambda\right) &:= \psi_0\left(\left\{\varphi_\alpha\right\}_{\alpha=1}^\Lambda \right) + \psi_0^\dagger\left(\left\{\varphi'_\alpha\right\}_{\alpha=1}^\Lambda \right) \\
& \quad - \left\langle \theta^0\left(\left\{\varphi_\alpha\right\}_{\alpha=1}^\Lambda \right), \eta_0\left(\left\{\varphi'_\alpha\right\}_{\alpha=1}^\Lambda\right) \right\rangle \quad , \\
\overline{D}_\alpha(\tilde{p}_\alpha, \tilde{q}_\alpha) &:= \tilde{\psi}_\alpha (\tilde{p}_\alpha) + \tilde{\psi}_\alpha^\dagger(\tilde{q}_\alpha) - \left\langle \tilde{\theta}^\alpha(\tilde{p}_\alpha) , \tilde{\eta}_\alpha(\tilde{q}_\alpha) \right\rangle \quad ,
\end{align*}
where $\psi_0$ denotes the canonical divergence on dually flat manifold $\overline{S}_0$ (see for example Bregman divergence defined in Equation \eqref{eqn:b-div-clS0}).

For $\alpha = 1,\ldots, \Lambda$, the functions $\psi^\dagger_0$ and $\tilde{\psi}_\alpha^\dagger$ denote the Legendre-Fenchel transformation of $\psi_0$ and $\tilde{\psi}_\alpha$ given by the following equations, respectively:
\begin{align*}
\psi_0^\dagger\left(\left\{\varphi'_\alpha\right\}_{\alpha=1}^\Lambda \right)  &:= \sup_{\varphi\in \overline{S}_0} \left\{ \left\langle \theta^0\left(\left\{\varphi_\alpha\right\}_{\alpha=1}^\Lambda \right), \eta_0\left(\left\{\varphi'_\alpha\right\}_{\alpha=1}^\Lambda\right) \right\rangle - \psi_0\left(\left\{\varphi_\alpha\right\}_{\alpha=1}^\Lambda \right) \right\} \quad , \\
\tilde{\psi}_\alpha^\dagger (\tilde{q}_\alpha) &:= \sup_{\tilde{p}_\alpha \in \tilde{S}_\alpha}\left\{\langle \tilde{\theta}^\alpha(\tilde{p}_\alpha) , \tilde{\eta}_\alpha(\tilde{q}_\alpha) \rangle - \tilde{\psi}_\alpha (\tilde{p}_\alpha) \right\} \quad .
\end{align*}
The divergence $D$ on $\mathcal{L}_V$ from Equation \eqref{productdivergence} can therefore be written as:
\begin{align}
\label{eqn:productdivergence:expand}
D\left( \tilde{p},\tilde{q} \right) &= D\left( \sum_{\alpha=1}^\Lambda \varphi_\alpha p_\alpha, \sum_{\alpha=1}^\Lambda \varphi'_\alpha q_\alpha \right) = D_{0}\left(\left\{\varphi_\alpha\right\}_{\alpha=1}^\Lambda,\left\{\varphi'_\alpha\right\}_{\alpha=1}^\Lambda\right) + \sum_{\alpha=1}^\Lambda \overline{D}_\alpha (\tilde{p}_\alpha, \tilde{q}_\alpha) \nonumber  \\ 
&= 
\psi_0\left(\left\{\varphi_\alpha\right\}_{\alpha=1}^\Lambda \right)  + \psi_0^\dagger\left(\left\{\varphi'_\alpha\right\}_{\alpha=1}^\Lambda \right) 
- \left\langle \theta^0\left(\left\{\varphi_\alpha\right\}_{\alpha=1}^\Lambda \right), \eta_0\left(\left\{\varphi'_\alpha\right\}_{\alpha=1}^\Lambda\right) \right\rangle \nonumber \\
&\quad + \sum_{\alpha=1}^\Lambda \left(\tilde{\psi}_\alpha (\tilde{p}_\alpha) + \tilde{\psi}_\alpha^\dagger(\tilde{q}_\alpha) - \langle \tilde{\theta}^\alpha(\tilde{p}_\alpha) , \tilde{\eta}_\alpha(\tilde{q}_\alpha) \rangle \right) \nonumber \\ 
& = \left( \psi_0\left(\left\{\varphi_\alpha\right\}_{\alpha=1}^\Lambda \right)  + \sum_{\alpha=1}^\Lambda  \tilde{\psi}_\alpha (\tilde{p}_\alpha) \right) + \left( \psi_0^\dagger\left(\left\{\varphi'_\alpha\right\}_{\alpha=1}^\Lambda \right) + \sum_{\alpha=1}^\Lambda  \tilde{\psi}_\alpha^\dagger(\tilde{q}_\alpha) \right)  \nonumber \\
& \quad -  \left( \left\langle \theta^0\left(\left\{\varphi_\alpha\right\}_{\alpha=1}^\Lambda \right), \eta_0\left(\left\{\varphi'_\alpha\right\}_{\alpha=1}^\Lambda\right) \right\rangle + \sum_{\alpha=1}^\Lambda \langle \tilde{\theta}^\alpha(\tilde{p}_\alpha) , \tilde{\eta}_\alpha(\tilde{q}_\alpha) \rangle \right) \quad .
\end{align}
The first and second part of $D$ in Equation \eqref{eqn:productdivergence:expand} are convex due to linearity of derivative, the independence of parameters given by condition \ref{cond:prod2}, and the fact that the Hessians of potential functions $\psi_0,\tilde{\psi}_1,\ldots,\tilde{\psi}_\Lambda$ are positive semi-definite. The third part of $D$ in Equation \eqref{eqn:productdivergence:expand} is a sum of inner products, which is again an inner product on the product parameter space in $\Xi_0 \times \Xi_1\times \cdots \Xi_\Lambda$. It remains to show that the Legendre–Fenchel transformation of the first component of $D$ is the second component in Equation \eqref{eqn:productdivergence:expand}.

Recall that  the parameters of the mixture distribution $\tilde{p}(x,\xi) \in \mathcal{L}_V$ are collected in $\left( \xi_i \right)_{i=1}^d := \left(\theta^0,\tilde{\theta}^1,\ldots,\tilde{\theta}^\Lambda\right) \in \Xi_0 \times \Xi_1 \times \cdots \times \Xi_\Lambda$, where $d:=\text{dim}\left(\Xi_0 \times \Xi_1 \times \cdots \times \Xi_\Lambda\right)$ \footnote{Recall that coordinates $\theta^0$ of mixture coefficients are not pulled-back.}. By linearity of derivative and condition \ref{cond:prod2}, the $g^D$-coordinate dual to $(\xi_i)_{i=1}^d$ is just the dual coordinates of $\overline{S}_0,\tilde{S}_1,\ldots,\tilde{S}_\Lambda$ expressed in product form: $\left(\eta_0,\tilde{\eta}_1,\ldots,\tilde{\eta}_\Lambda \right) = \left( \frac{\partial}{\partial \xi_i}\left( \psi_0\left(\left\{\varphi_\alpha\right\}_{\alpha=1}^\Lambda \right)  + \sum_{\alpha=1}^\Lambda  \tilde{\psi}_\alpha (\tilde{p}_\alpha) \right) \right)$. 

For simplicity, we let  $\left\langle \theta^0(\varphi), \eta_0(\varphi') \right\rangle := \left\langle \theta^0\left(\left\{\varphi_\alpha\right\}_{\alpha=1}^\Lambda \right), \eta_0\left(\left\{\varphi'_\alpha\right\}_{\alpha=1}^\Lambda\right) \right\rangle$. The dual potential of $\left( \psi_0\left(\left\{\varphi_\alpha\right\}_{\alpha=1}^\Lambda \right)  + \sum_{\alpha=1}^\Lambda  \tilde{\psi}_\alpha (\tilde{p}_\alpha) \right)$ ( the first part of $D$ in Equation \eqref{eqn:productdivergence:expand}) is given by the following Legendre-Fenchel transformation:

\begin{align*}
&\left( \psi_0\left(\left\{\varphi_\alpha\right\}_{\alpha=1}^\Lambda \right) + \sum_\alpha \tilde{\psi}_\alpha (\tilde{p}_\alpha) \right)^\dagger = \\  
&\sup_{\varphi\in \overline{S}_0,\tilde{p}_\alpha \in \tilde{S}_\alpha} 
 \left\{\left(\left\langle \theta^0(\varphi), \eta_0(\varphi') \right\rangle + \sum_{\alpha=1}^\Lambda\left\langle \tilde{\theta}^\alpha(\tilde{p}_\alpha) , \tilde{\eta}_\alpha(\tilde{q}_\alpha) \right\rangle \right)  - \left( \psi_0\left(\left\{\varphi_\alpha\right\}_{\alpha=1}^\Lambda \right) + \sum_\alpha  \tilde{\psi}_\alpha (\tilde{p}_\alpha) \right) \right\}\\
&= \sup_{\varphi\in \overline{S}_0,\tilde{p}_\alpha \in \tilde{S}_\alpha} \left\{ 
\left(\left\langle \theta^0(\varphi), \eta_0(\varphi') \right\rangle - \psi_0\left(\left\{\varphi_\alpha\right\}_{\alpha=1}^\Lambda \right) \right)
+  \sum_{\alpha=1}^\Lambda \left(\left\langle \tilde{\theta}^\alpha(\tilde{p}_\alpha) , \tilde{\eta}_\alpha(\tilde{q}_\alpha) \right\rangle - \tilde{\psi}_\alpha (\tilde{p}_\alpha) \right) \right\} \\
&= \sup_{\varphi\in \overline{S}_0} \left\{\left\langle \theta^0(\varphi), \eta_0(\varphi') \right\rangle - \psi_0\left(\left\{\varphi_\alpha\right\}_{\alpha=1}^\Lambda \right) \right\} + \sum_{\alpha=1}^\Lambda \sup_{\tilde{p}_\alpha \in \tilde{S}_\alpha}\left\{\left\langle \tilde{\theta}^\alpha(\tilde{p}_\alpha) , \tilde{\eta}_\alpha(\tilde{q}_\alpha) \right\rangle - \tilde{\psi}_\alpha (\tilde{p}_\alpha) \right\} \\
&= \psi_0^\dagger\left(\left\{\varphi'_\alpha\right\}_{\alpha=1}^\Lambda \right) + \sum_{\alpha=1}^\Lambda \tilde{\psi}_\alpha^\dagger (\tilde{q}_\alpha) \quad .
\end{align*}

The third equality follows from the functional independence of $\left\{\varphi_\alpha\right\}_{\alpha=1}^\Lambda \in \overline{S}_0$ and  $\tilde{p}_\alpha$'s in $\tilde{S}_\alpha$.  There the Legendre-Fenchel transform of the first component of $D$ is exactly the second component of $D$, and $D$ is the canonical divergence of the family of mixture densities:
\begin{align*}
\left( \mathcal{L}_V, g_0 \oplus \bigoplus_{\alpha =1}^\Lambda \tilde{g}_\alpha, \nabla^0 \oplus \bigoplus_{\alpha=1}^\Lambda \tilde{\nabla}^\alpha, \nabla^{0^*} \oplus \bigoplus_{\alpha=1}^\Lambda  \tilde{\nabla}^{\alpha^*} \right) \quad .
\end{align*}
\end{proof}

\subsection{Towards a Stochastic Optimization method on Riemannian Manifolds}

%

In this section, we extended the notion of locally inherited probability densities (Section \ref{app:localdistn}) beyond the normal neighbourhood, and constructed a family of parametrized mixture densities $\mathcal{L}_V$ on totally bounded subsets $V$ of $M$. We derived the geometrical structure of the family of mixture densities $\mathcal{L}_V$ and showed that $\mathcal{L}_V$ is a product statistical manifold of the simplex of mixture coefficients and the families of mixture component densities.

The significance of the product structure of $\mathcal{L}_V$ is twofold: 
it provides a computable geometrical structure for finitely parametrized statistical model on $M$ that extends beyond the confines of a single normal neighbourhood, and it allows us to handle statistical parameter estimations and computations of mixture coefficients and mixture components independently.

The product Riemannian structure of mixture densities $\mathcal{L}_V$ over Riemannian manifolds provides us with a geometrical framework to tackle the optimization problem described in Section \ref{sec:intro}. In the remainder of the paper, we apply this framework to construct a novel stochastic derivative-free optimization algorithm on Riemannian manifolds using the statistical geometry of the decision space and the Riemannian geometry of the search space, overcoming the local restrictions of the RSDFO under the Riemannian adaptation approach. 


%

\section{Extended Riemannian Stochastic Derivative-Free Optimization Algorithms}
\label{ch:paper2}

Towards the end of the Section \ref{sec:principle:manopt}, we discussed the local restrictions of Riemannian adaptations of Euclidean optimization algorithms. On Riemannian adaptation of Stochastic Derivative-Free Optimization (SDFO) algorithms such as Riemannian CMA-ES \cite{colutto2010cma}, this issue is particularly prevalent, as the local restrictions are imposed by \textit{both} the Riemannian adaptation framework of optimization algorithms (Section \ref{sec:principle:manopt}) and the notion of locally inherited probability densities (\cite{pennec2006intrinsic} and Section \ref{app:localdistn}).

Tangent spaces on general manifolds are disjoint spaces, families of locally inherited probability densities over them are \textit{different} statistical manifolds. Whereas on Euclidean spaces $\mathbb{R}^n$ the tangent space can be canonically identified with one another by translation, a family of finitely parametrized probability densities on $\mathbb{R}^n$ can be regarded as a \textit{single} statistical manifold. This makes it difficult to study Riemannian adaptation of SDFO from the Information Geometric perspective as its Euclidean counterparts.

In order to overcome the local restrictions of RSDFO and to extend it to a more global scale, we require families of finitely parametrized probability densities defined beyond the normal neighbourhood on Riemannian manifolds. This is accomplished by the notion of mixture densities over totally bounded subsets on Riemannian manifolds described in Section \ref{Subsection:MixtureDensity}.
%


Using the product Riemannian structure of mixture densities, in this section we propose a novel algorithm -- Extended RSDFO, which extends and augments an RSDFO core and addresses the local restriction of RSDFO using both the intrinsic Riemannian geometry of the manifold $M$ (search space) and the product statistical Riemannian geometry of families of mixture densities over $M$ (decision spaces). Moreover, the product statistical geometry of mixture densities also allows us to study the evolutionary step and convergence of Extended RSDFO from a geometrical perspective.

The remainder of the section is outlined as follows:

\begin{enumerate}
\item In Section \ref{sec:ereda}, we present an overview of the proposed novel algorithm -- Extended RSDFO, which addresses the local restrictions of RSDFO. 
\item In Section \ref{sec:global:distnmfold}, we recall the notion of parametrized mixture densities over totally bounded subsets of the search space manifold $M$ and describe it in the context of Extended RSDFO.
\item In Section \ref{sec:extrsdfo}, we describe Extended RSDFO in detail. We first describe Extended RSDFO in algorithm form, then we discuss the components of Extended RSDFO in detail. We show that the expected fitness of solutions in Extended RSDFO is non-decreasing.
\item In Section \ref{sec:theo}, we discuss the geometry of the evolutionary step of Extended RSDFO using a novel metric on the simplex of mixture coefficients, which allows us to study Extended RSDFO from the geometrical perspective with Riemannian geometry of the base space and Information Geometry of the decision space.
\item Finally in Section \ref{sec:conv}, we discuss the convergence behaviors of Extended RSDFO. In particular, we show that Extended RSDFO converges in \textit{finitely many} steps in compact connected Riemannian manifolds.
\end{enumerate}

\subsection{Overview of Extended RSDFO on Riemannian manifolds}
\label{sec:ereda}
In this section we propose a novel algorithm,  Extended RSDFO to address the local restrictions of RSDFO discussed in Section \ref{sec:rsdfo}.

For each iteration of RSDFO algorithm (Algorithm \ref{alg:greda}), search information is generated locally within the normal neighbourhood of each search iterate. This means at each iteration, the search region of the algorithm is confined by the normal neighbourhood of the search iterate.
%
In order to overcome this local restriction, the proposed algorithm considers multiple search centroids simultaneously, instead of a single search iterate. We optimize over a set of overlapping geodesic balls covering the manifold centering around each of the centroids, instead of the centroids themselves. This process is guided by families of parametric mixture densities over the manifold, allowing us to relate probability distributions across different tangent spaces.

We fix a core RSDFO algorithm throughout the rest of the discussion, and we summarize briefly a single step of the proposed algorithm as follows: at each iteration, Extended RSDFO initiates with a finite set of centroids on the Riemannian manifold. The search region of the iteration is thus formed by the union of the set of geodesic balls of injectivity radius centered at the centroids. We then perform one iteration of the chosen RSDFO algorithm for \textit{each} centroid, which then generates a new set of centroids (and the corresponding geodesic balls of injectivity radius) in the manifold. In addition, another set of ``exploration" centroids are also generated from the boundary of explored region. The sets of centroids are then compared based on the expected fitness over the geodesic balls centered at each centroid, and the fittest centroids are carried over to the next iteration. Each ``stream" of local RSDFO search shares a parametrized family of locally inherited densities, whereas a boundary exploration point initiates a new ``stream". An illustration of the above discussion is summarized in Figure \ref{fig:full2}.

\begin{figure*}[h!]
    \centering
    \begin{subfigure}[t]{0.32\textwidth}
        \centering
        \includegraphics[width=1\textwidth]{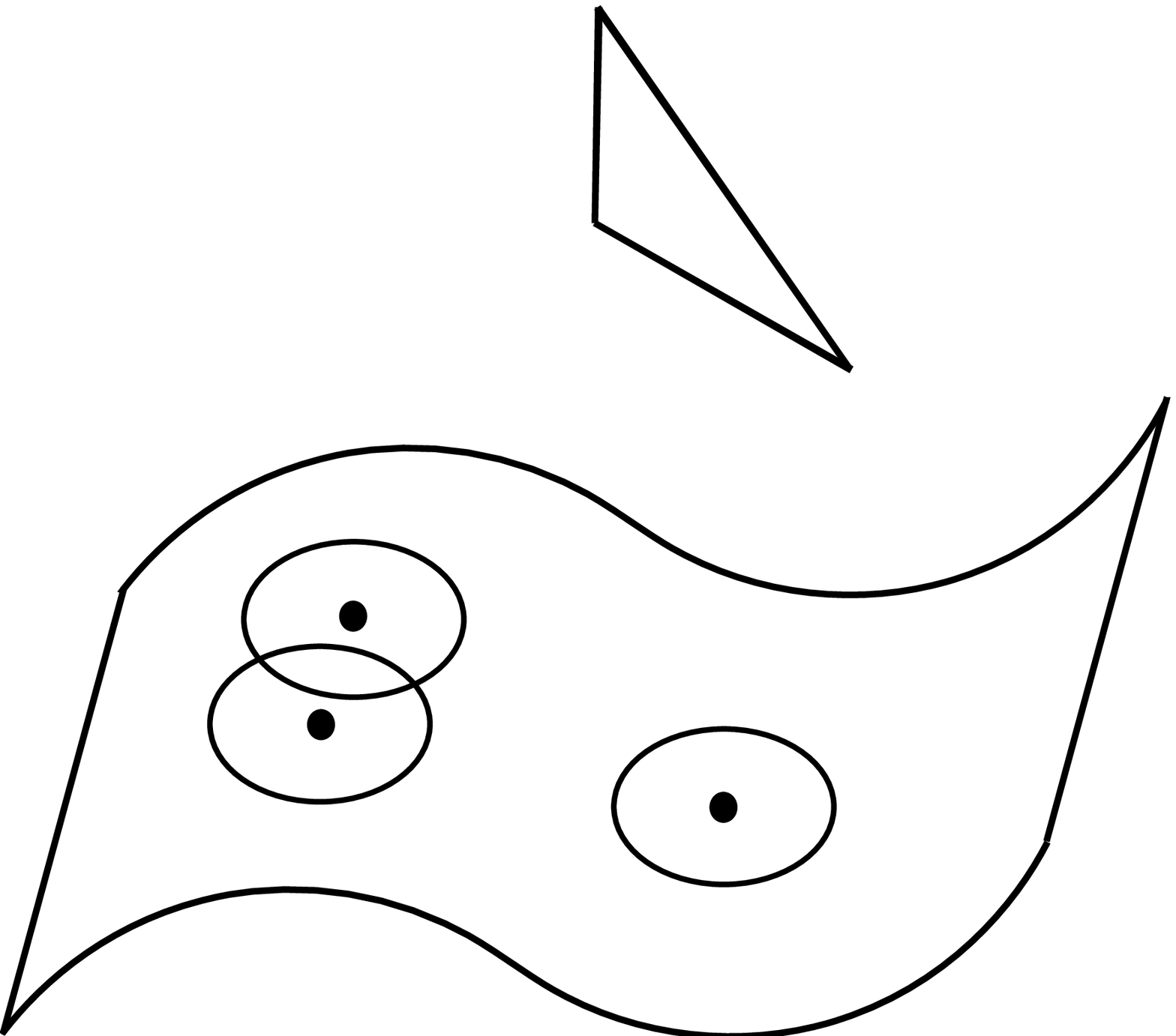} 
        \caption{Initiation.}    \label{fig:full2:a}
    \end{subfigure}%
    ~ 
    \begin{subfigure}[t]{0.32\textwidth}
        \centering
        \includegraphics[width=1\textwidth]{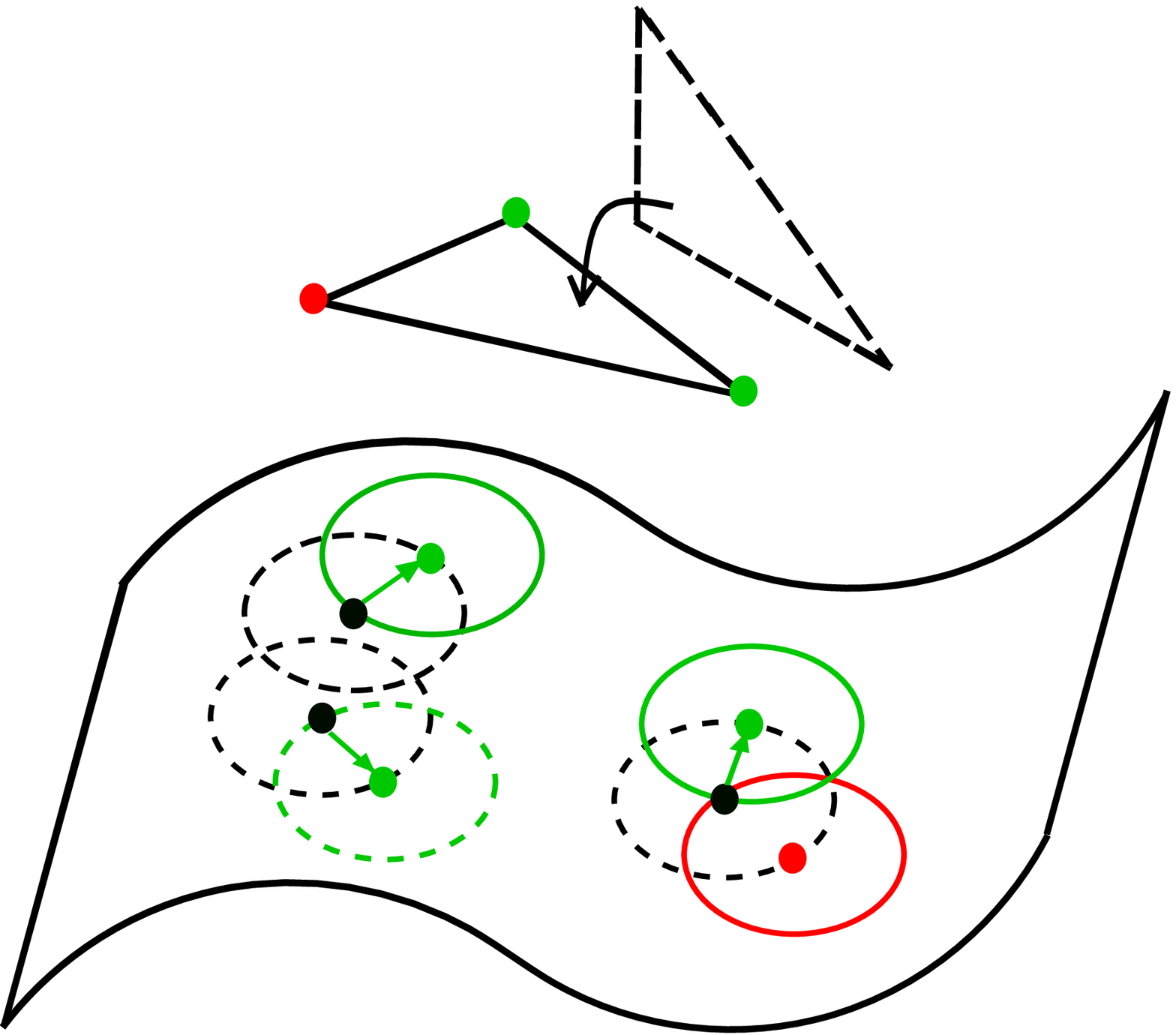} 
        \caption{Evolution of mixture coefficients and mixture components.}    \label{fig:full2:b}
    \end{subfigure}
        ~ 
    \begin{subfigure}[t]{0.32\textwidth}
    \centering
       \includegraphics[width=1\textwidth]{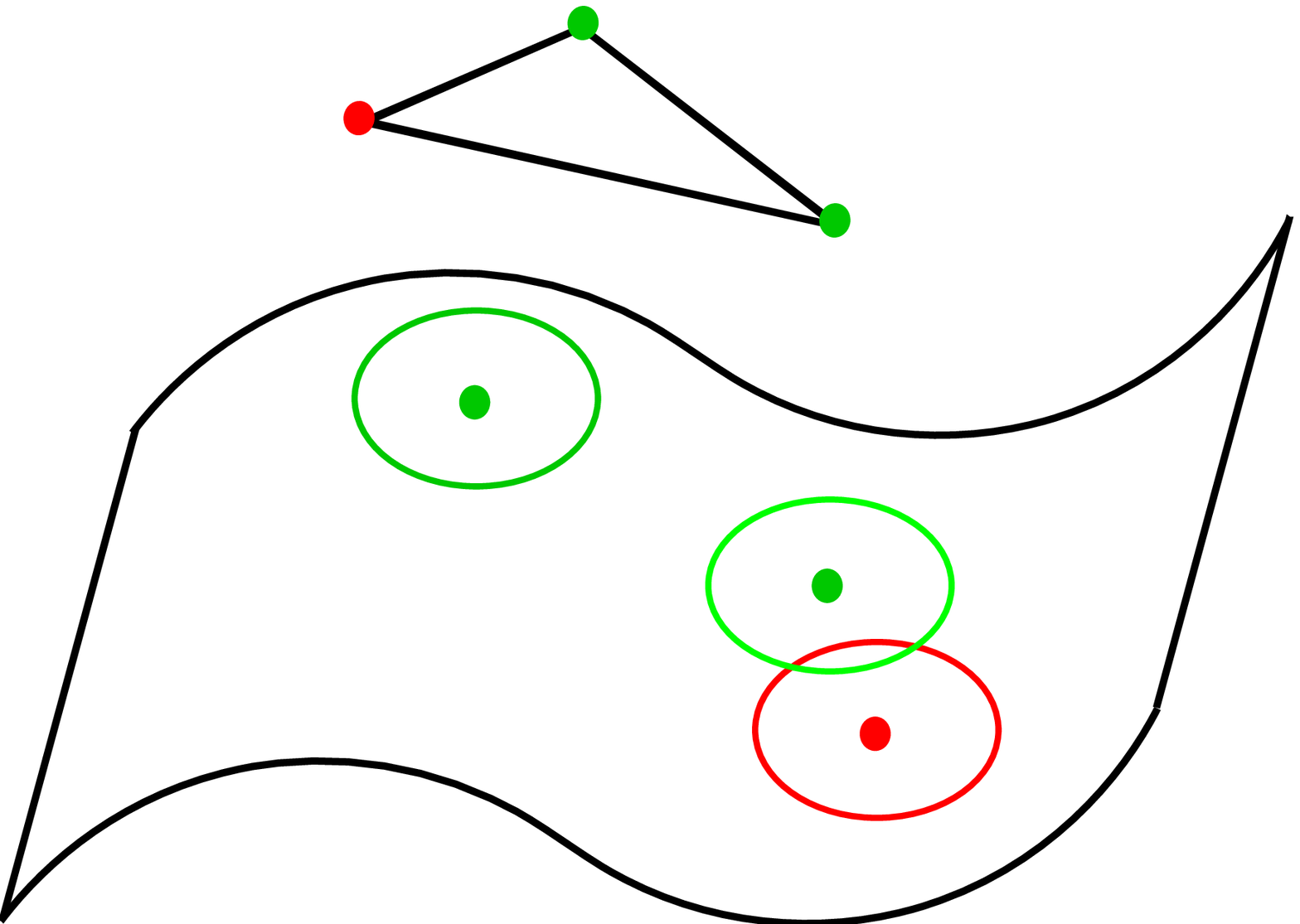} 
        \caption{Culling of centroids.}\label{fig:full2:c}
    \end{subfigure}
    \caption{Illustration (from left to right) of the ``full" evolutionary step of Extended RSDFO. At each iteration, Extended RSDFO begins with a set of centroids (black) on the manifold (Figure \ref{fig:full2:a} ). In Figure \ref{fig:full2:b}, each centroid generates a new ``offspring" with a step of RSDFO ``core" (green). An additional centroid (red), is generated on the boundary of the explored region. The centroids are then culled according their relative expected fitness. The culled centroids are marked by the dotted lines.  The remaining centroids move onto the next iteration, as illustrated on the right image (Figure \ref{fig:full2:c} ). \protect\footnotemark}
    \label{fig:full2}
\end{figure*}
\interfootnotelinepenalty=10000

\footnotetext{It is worth noting that for this illustration (Figure \ref{fig:full2}), the ``full" interim simplex in the middle image consists of $7$ vertices, where each vertex corresponds to the mixture coefficient of each of the $7$ interim centroids (initial, generated, boundary). The ``original" simplex (left picture) and the ``target" simplex share at most $(3+4)-4=3$ number of vertices (in this case none). For each step of the algorithm we transverse through a different interim simplex in this fashion. One can picture the evolutionary step of the algorithm in the context of simplex of mixture coefficients as jumping from one subsimplex to another.}

%

The objective of the extended algorithm is to find a \textit{set} of fittest centroids (and their corresponding geodesic balls), where the weight of each centroid is determined by the mixture coefficients of the overarching mixture densities, as we will show in the remainder of the section.

Furthermore, while the local families of parametrized densities on each geodesic ball may be different, they are related by an overarching family of \emph{mixture distributions} over the union of the geodesic balls and an accumulative open cover of $M$. This allows us to study the evolutionary step from an Information Geometric perspective and analyze the convergence behaviour.

It is also worth noting that, whilst parallel transport is used in RSDFO algorithms to inherit the statistical parameter from one tangent space to another, it is not completely necessarily in Extended RSDFO. The special case of Extended RSDFO algorithm (discussed in Section \ref{sec:theo}) shows that with sufficiently small search radius $j_x \ll \operatorname{inj}(x)$ and a sufficiently large number of geodesic balls covering the manifold, the fitness landscape can be described by varying the mixture coefficients, while keeping the mixture components fixed.

For each iteration, a typical search region of Extended RSDFO is given by a union of geodesic balls, forming a totally bounded subset in the searh space manifold $M$. In the next section, we recall the notion of parametrized mixture distributions on totally bounded subsets of Riemannian manifolds described in Section \ref{Subsection:MixtureDensity} in the context of Extended RSDFO \footnote{The detailed construction and dualistic geometry \cite{Amari2000} of the mixture densities on Riemannian manifolds is detailed in Section \ref{Subsection:MixtureDensity}.}. 

\subsection{Parametrized Mixture Distribution over search regions of Extended RSDFO}
\label{sec:global:distnmfold}


%
We recall from Section \ref{Subsection:MixtureDensity} the notion of finitely parametrized mixture densities on totally bounded subsets of $M$: Let $(M,g)$ be an smooth $n$-dimensional Riemannian manifold, an orientation-preserving open cover of $M$ is an at-most countable set of pairs $\mathcal{E}_M := \left\{\left(\rho_\alpha,U_\alpha \right) \right\}_{\alpha \in \Lambda_M}$  satisfying: 
\begin{enumerate}
\item $U_\alpha\subset \mathbb{R}^n$, $\rho_\alpha: U_\alpha \rightarrow M$ are orientation preserving diffeomorphisms for each $\alpha \in \Lambda_M$ 
\item the set $\left\{ \rho_\alpha(U_\alpha) \right\}_{\alpha \in \Lambda_M}$ is an open cover of $M$.
\end{enumerate} 



%
%
%
%

For each $x_\alpha \in M$, let $j_{x_\alpha} \leq \operatorname{inj}(x_\alpha)$ and let $\mathcal{E}_M := \left\{ \left( \exp_{x_\alpha}, B(\vec{0},j_{x_\alpha}) \right) \right\}_{\alpha \in \Lambda_M}$ denote an orientation-preserving open cover of $M$ indexed by at most countable  $\Lambda_M$.\footnote{This is due to the fact that manifolds are second-countable, thus Lindelof, meaning all open cover has an at-most countable sub-cover.} Given a finite set of centroids $\left\{ x_\alpha \right\}_{\alpha \in \Lambda_V}$, indexed by a finite subset $\Lambda_V \subset \Lambda_M$, a typical search region of Extended RSDFO can be represented by:
\begin{align}
\label{eqn:typicalsearchregion}
V := \bigcup_{\alpha \in \Lambda_V}  \exp_{x_\alpha}\left( B(\vec{0},j_{x_\alpha}) \right) \subset M \quad .
\end{align}
Let $\mathcal{E}_V := \left\{ \left( \exp_{x_\alpha}, B(\vec{0},j_{x_\alpha}) \right) \right\}_{\alpha \in \Lambda_V} \subset \mathcal{E}_M$, denote the finite sub-cover of $\mathcal{E}_M$ over $V$. 
For each $\alpha \in \Lambda_V$, let $m_\alpha \in \mathbb{N}$, and let $S_\alpha := \left\{p(\cdot \vert\theta^\alpha) \vert \theta^\alpha \in \Xi_\alpha \subset \mathbb{R}^{m_\alpha} \right\} \subset \operatorname{Prob}(B(\vec{0},j_{x_\alpha})) $ denote a family of finitely parametrized probability densities over $B(\vec{0},j_{x_\alpha}) \subset T_{x_\alpha} M$. 

By the discussion in Section \ref{app:localdistn}, we can define for each $\alpha \in \Lambda_V$ a family of locally inherited finitely parametrized densities over \textit{each} $\exp_{x_\alpha}\left(B(\vec{0},j_{x_\alpha}) \right) \subset M$: 
\begin{align*}
\tilde{S}_\alpha := \exp_\alpha^{-1^*}S_\alpha = \left\{ \tilde{p}(\cdot \vert \theta^\alpha) = \exp^{-1^*}_\alpha p(\cdot \vert\theta^\alpha) \right\}  \quad .
\end{align*} 


Consider the closure of the simplex of mixture coefficients:  
\begin{align}
\label{eqn:s:closure}
\overline{S}_0 := \left\{ \left\{\varphi_\alpha\right\}_{\alpha \in \Lambda_V} \, \middle| \, \varphi_\alpha \in \left[0,1\right], \sum_{\alpha \in \Lambda_V} \varphi_\alpha = 1 \right\} \subset \left[0,1\right]^{\vert\Lambda_V\vert-1} \quad .
\end{align}

A family of parametrized mixture densities on $V$ can thus be defined by: 
\begin{align}
\label{eqn:mixture-distn-mfold}
\mathcal{L}_V := \left\{\sum_{\alpha \in \Lambda_V} \varphi_\alpha \cdot \tilde{p}(\cdot \vert \theta^\alpha) \, \middle| \, \left\{\varphi_\alpha\right\}_{\alpha \in \Lambda_V} \in \overline{S}_0, \, \tilde{p}(\cdot \vert \theta^\alpha) \in \tilde{S}_\alpha \right\} \quad .
\end{align}

By the discussion in Section \ref{Subsection:DualGeoLv}, the set of mixture densities $\mathcal{L}_V$ is a product Riemannian manifold given by: $\mathcal{L}_V = \overline{S}_0 \times \bigoplus_{\alpha \in \Lambda_V} \tilde{S}_\alpha$. \footnote{The conditions required  in Section \ref{sec:lvsmoothmfold} are naturally satisfied by the construction of $V$ and the families of inherited densities described in this paper.}

\subsection{Extended RSDFO} 
\label{sec:extrsdfo}
%



Let $(M,g)$ be a Riemannian manifold, and let the function $f: M \rightarrow \mathbb{R}$ denote the objective function of a \emph{maximization} problem over $M$. We assume, without loss of generality, that $f$ is a strictly positive function, i.e. $f(x) >0 $ for all $x\in M$. The minimization case is discussed in Remark \ref{rmk:evo_min}. 

Given a point $x_\alpha \in M$, the \textbf{expected fitness} \cite{wierstra2008fitness}(or stochastic relaxation in \cite{malago2015gradient}) of $f$ over probability density $p(x \vert \theta^\alpha)$ locally supported in geodesic ball  $\exp_{x_\alpha}\left(B(\vec{0},j_{x_\alpha}) \right)$ is the real number \footnote{Note that this is an integral on the Riemannian manifold, for further detail see \cite{lee2001introduction}.}:
\begin{align}
\label{eqn:expectedfitness:ball}
    E_\alpha := \int_{\exp_{x_\alpha}\left(B(0,j_{x_\alpha}) \right)} f(x) \tilde{p}(x \vert \theta^\alpha) dx \quad .
\end{align}

In practice, this integral can be approximated using $\tilde{p}(x \vert \theta^\alpha)-$Monte Carlo samples for $x \in M$.

    
Choose and fix an RSDFO (Algorithm \ref{alg:greda}) method as the ``core", Extended RSDFO on Riemanian manifolds is summarized in Algorithm \ref{alg:ereda}. 

%

%

%
%
%


\begin{algorithm}[hbt!]
 \KwData{Initial set of centroids in $M$: $X^0 := \left\{x_\alpha \right\}_{\alpha \in \Lambda^0}$, $\Lambda^0\subset \Lambda_M$, initial set of parameters for the mixture distribution: $ \left\{ \varphi^0_\alpha, \theta_0^\alpha \right\}_{\alpha \in \Lambda^0}$. Set iteration count $k = 0$, choose non-increasing non-zero function $ 0 < \tau(k) \leq 1$. Integers $N_{cull}$, $N_{random} > 0$.}
 \While{stopping criterion not satisfied}{
	Generate $N_{random}$ non-repeating centroids $\overline{X}^{k+1}$ from the mixture distribution:
	\begin{align*}
	P_k(\cdot) = \left(1-\tau(k) \right) \cdot \sum_{\alpha \in \Lambda^k} \varphi^k_\alpha \delta_{x_\alpha} + \tau(k) \cdot U \quad ,
	\end{align*}
	where $\delta_{x_\alpha}$ is the delta function on $x_\alpha \in X^k$, and $U$ denote the exploration distribution described in Section \ref{sec:prac:expdistn} below  \label{alg:prac:generate} \;

	\textbf{Evolution of mixture components:} For each $x_\alpha \in \overline{X}^{k+1}$ perform one step of the chosen RSDFO, generate new set of centroids $\hat{X}^{k+1}$, and the corresponding statistical parameter $\theta^\alpha_{k+1}$ for each $x_\alpha \in \hat{X}^{k+1}$. \label{alg:prac:add} \; 
%
 
 	We now have interim centroids $X^k \cup \hat{X}^{k+1}$, indexed by  $\hat{\Lambda}^{k+1}$, where $\vert \hat{\Lambda}^{k+1} \vert =\vert X^k \cup \hat{X}^{k+1} \vert < \infty$. \label{alg:prac:interim} \; 
 	
	\textbf{Evolution of mixture coefficients:} Evaluate the expected fitness $E_\alpha$ for each $x_\alpha \in X^k \cup \hat{X}^{k+1}$.  \label{alg:prac:sr} \; 
	Preserve the fittest $N_{cull}$ centroids  $X^{k+1} := \left\{ x_\alpha\right\}_{\alpha \in \Lambda^{k+1}}$, indexed by $\Lambda^{k+1} \subset \hat{\Lambda}^{k+1}$. This defines the search region of the next iteration. \label{alg:prac:cull} \;
	
  Recalibration of mixture coefficients $\left\{\varphi_\alpha^{k+1} \right\}_{\alpha \in \Lambda^{k+1}}$ described in Equation \eqref{eqn:mixture-evo-reduced} in Section \ref{sec:prac:mixturecoeff}   \label{alg:prac:mixture}\;
  k = k+1\;
 } 
 \caption{Extended RSDFO}
 \label{alg:ereda}
\end{algorithm}
Note that in Line \ref{alg:prac:add} of Extended RSDFO (Algorithm \ref{alg:ereda}): each search centroid has its own parametrized family of locally inherited probability densities. The statistical parameters on the new centroids are parallel translated from their previous iterate, as described in Section \ref{sec:paralleltransport} and Algorithm \ref{alg:greda} in Section \ref{sec:principle:manopt}.






\subsubsection{Additional Parameters}
\label{sec:additionalparameter}
In addition to the native parameters inherited from the chosen RSDFO algorithm, Extended RSDFO requires three additional parameters: two positive integers $N_{random},N_{cull} \in \mathbb{N}_+$ to control the number of search centroids, and a non-increasing, non-zero function $\tau(k)$. The function $\tau(k)$ decays to an arbitrary small positive constant $1 \gg \epsilon_c > 0 $ as $k\rightarrow \infty$. That is, the sampling will eventually be mostly from $\sum_{\alpha \in \Lambda^k} \varphi^k_\alpha \delta_{x_\alpha}$.

\subsubsection{Evolutionary Step}
\label{sec:prac:mixturecoeff}

\begin{table*}[h]
\begin{tabular}{|p{0.09\textwidth}|l|l|p{0.09\textwidth}|l|p{0.35\textwidth}|}
\hline
Mixture family & Simplex                & Index             & Search region   & Centroids                & Description                                                                                                                                             \\ \hline
$\mathcal{L}_{V^k} $ & $\overline{S}^{k}_0$   & $\Lambda^k$       & $V^k$           & $X^k$                    & Generated by centroids from the $\left(k-1\right)^{th}$ iteration                                                                                                \\ \hline
$\mathcal{L}_{V^{k+1}} $ &$\overline{S}^{k+1}_0$ & $\Lambda^{k+1}$   & $V^{k+1}$       & $X^{k+1}$                & Generated by new centroids for the $\left( k+1\right)^{th}$ iteration from the $k^{th}$ iteration                                                                                 \\ \hline
$\mathcal{L}_{\hat{V}^{k+1}} $ & $\hat{S}_0^{k+1}$      & $\hat{\Lambda}^{k+1}$ & $\hat{V}^{k+1}$ & $X^k \cup \hat{X}^{k+1}$ & Generated by the interim centroids within the $k^{th}$ iteration.  \newline
$\hat{S}_0^{k+1} \supset \overline{S}^{k+1}_0 , \overline{S}^{k}_0$ \\ \hline
\end{tabular}
  \caption{Summary of notations related to closure of simplex of mixture coefficients in the evolutionary step of the $k^{th}$ iteration of Extended RSDFO.}
\label{table:evonotation}
\end{table*}

In this section, we discuss the evolutionary steps of Algorithm \ref{alg:ereda} in the context of family of mixture distributions described in Section \ref{sec:global:distnmfold} (illustrated in Figure \ref{fig:full2}). The geometrical interpretations and justifications will be fleshed out in Section \ref{sec:theo}.

In the $k^{th}$ iteration for $k>0$, a finite set of centroids $X^k := \left\{ x_\alpha \right\}_{\alpha \in \Lambda^k}$ is generated from the previous iteration on the Riemannian manifold, and a new set of centroids $\hat{X}^{k+1}$ is generated through lines \ref{alg:prac:generate} and \ref{alg:prac:add} of Extended RSDFO (Algorithm \ref{alg:ereda}). The union of the two sets of centroids: $X^k \cup \hat{X}^{k+1}$, indexed by $\hat{\Lambda}^{k+1}$, forms the centroids of the interim search region of the $k^{th}$ iteration.

The set of interim centroids is subsequently culled according to their expected fitness, and the fittest centroids $X^{k+1} := \left\{ x_\alpha\right\}_{\alpha \in \Lambda^{k+1}}$, indexed by $\Lambda^{k+1} \subset \hat{\Lambda}^{k+1}$, are preserved for the next iteration. This process is repeated until a termination criteria is satisfied (See Section \ref{sec:termin}).

%

More formally, consider $\left\{ \exp_{x_\alpha} \left(B(\vec{0},j_{x_\alpha} \right) \right\}_{\alpha \in \hat{\Lambda}^{k+1}}$ the set of (closed) (See Section \ref{sec:dg:prelim}.) geodesic ball of injectivity radius centered at $\left\{ x_\alpha \right\}_{\hat{\Lambda}^{k+1}} := X^k \cup \hat{X}^{k+1}$. The region in $M$ explored by the $k^{th}$ iteration of Algorithm \ref{alg:ereda} can be separated into three parts:
\begin{align}
V^k := \bigcup_{\alpha \in \Lambda^{k}} \tilde{B}_\alpha \quad &, \quad V^{k+1} := \bigcup_{\alpha \in \Lambda^{k+1}} \tilde{B}_\alpha , \quad \text{and} \label{eqn:searchvks} \\
\hat{V}^{k+1} &:= \bigcup_{\alpha \in \hat{\Lambda}^{k+1}} \tilde{B}_\alpha \supseteq  V^k \cup V^{k+1} \quad , \nonumber
\end{align}
where $\tilde{B}_\alpha := \exp_{x_\alpha} \left(B(\vec{0},j_{x_\alpha} ) \right) \subset M$ for simplicity. We begin the $k^{th}$ iteration with the search region $V^k$, which expands to the interim search region $\hat{V}^{k+1}$. The interim search region is then reduced to $V^{k+1} \subset \hat{V}^{k+1}$, the search region for the $\left( k +1\right)^{th}$ iteration, after the culling of centroids.



As a result, for each iteration of Algorithm \ref{alg:ereda}, we would transverse through \emph{three} families of parametrized mixture densities supported on the subsets $V^k, V^{k+1}, \hat{V}^{k+1} \subset M$ respectively. We describe the families of parametrized mixture densities as follows, a summary of the notation can be found in Table \ref{table:evonotation}.

Consider the closure of simplices of mixture coefficients given by:
\begin{align}
\overline{S}_0^{k} &:= \left\{ \left\{\varphi_\alpha\right\}_{\alpha \in \Lambda^{k}} \, \middle| \, \varphi_\alpha \in \left[0,1\right], \sum_{\alpha \in \Lambda^{k}} \varphi_\alpha = 1 \right\} \label{eqn:s:start}
\end{align}
\begin{align}
\hat{S}_0^{k+1} &:= \left\{ \left\{\varphi_\alpha\right\}_{\alpha \in \hat{\Lambda}^{k+1}} \, \middle| \, \varphi_\alpha \in \left[0,1\right], \sum_{\alpha \in \hat{\Lambda}^{k+1}} \varphi_\alpha = 1 \right\}  \label{eqn:s:midd}
\end{align}
\begin{align}
\overline{S}_0^{k+1} &:= \left\{ \left\{\varphi_\alpha\right\}_{\alpha \in \Lambda^{k+1}} \, \middle| \, \varphi_\alpha \in \left[0,1\right], \sum_{\alpha \in \Lambda^{k+1}} \varphi_\alpha = 1 \right\} \quad . \label{eqn:s:end}
\end{align} 
Note that the simplices $\overline{S}_0^k, \, \overline{S}_0^{k+1}$ are in fact two faces of $\hat{S}_0^{k}$ sharing \textit{at most} $\left( N_{cull} + N_{random}\right) - N_{random}  = N_{cull}$ vertices \footnote{This can be zero, as illustrated in Figure \ref{fig:full2}}. An example of the simplices is illustrated in Figure \ref{fig:simplexprac}.

\begin{figure}[hbt!]
    \centering
    \def\svgwidth{\columnwidth}
    \includegraphics[scale=0.7]{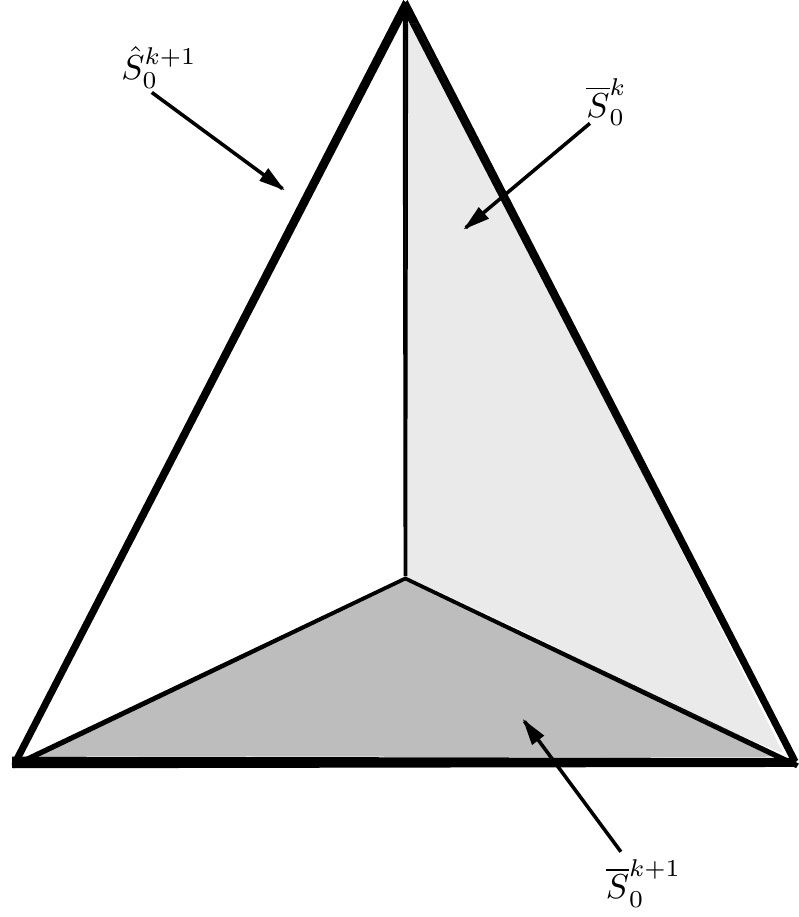}
    \caption{Simplicial illustration of the simplices in evolutionary step of Extended RSDFO (Algorithm \ref{alg:ereda}). The interim simplex $\hat{S}_0^k$ is given by the entire simplex with 4 vertices. The subsimplices $\overline{S}_0^k$ and $\overline{S}_0^{k+1}$ is given by the left and bottom simplices (shaded) with 3 vertices respectively.}
    \label{fig:simplexprac}
\end{figure}

The three mixture families in the $k^{th}$ iteration, summarized in Table \ref{table:evonotation}, are thus given by:
\begin{align}
\label{eqn:prac:mixture-distn-mfold}
\mathcal{L}_V = \left\{\sum_{\alpha \in \Lambda_V} \varphi_\alpha \cdot \tilde{p}(\cdot \vert \theta^\alpha) \middle| \left\{\varphi_\alpha\right\}_{\alpha \in \Lambda_V} \in \overline{S}_0, \, \tilde{p}(\cdot \vert \theta^\alpha) \in \tilde{S}_\alpha \right\} \quad ,
\end{align}
where $V = V^k, \hat{V}^{k+1}, V^{k+1} \subset M$ denoting the searching region generated set of centroids $X = X^k, X^k \cup \hat{X}^{k+1}, X^{k+1} \subset M$ respectively described in Equation \eqref{eqn:searchvks}. The corresponding sets of indices are given by $\Lambda_V = \Lambda^k, \hat{\Lambda}^{k+1},\Lambda^{k+1}$, and the corresponding closure of simplices of mixture coefficients is denoted by $\overline{S}_0 = \overline{S}_0^{k}, \hat{S}_0^{k+1}, \overline{S}_0^{k+1}$ respectively. For all $\alpha \in \Lambda_V$, where $\Lambda_V = \Lambda^k, \hat{\Lambda}^{k+1},\Lambda^{k+1}$, the set $\tilde{S}_\alpha := \exp_\alpha^{-1^*}S_\alpha$ are parametrized densities on $\tilde{B}_\alpha \subset M$ described in Section \ref{app:localdistn}. A summary of the notations can be found in Table \ref{table:evonotation}. For descriptions of the geometry of $\mathcal{L}_V$, please refer to Section \ref{Subsection:MixtureDensity}. \footnote{Detailed descriptions of the geometry of $\mathcal{L}_V$ is in Section \ref{Subsection:MixtureDensity}. } 


Due to the product manifold structure of the family of mixture densities $\mathcal{L}_V = \overline{S}_0 \times \bigoplus_{\alpha \in \Lambda_V} \tilde{S}_\alpha$, the evolution of statistical parameters of the mixture component densities $\tilde{S}_\alpha$ for $\alpha \in \Lambda_V$ can be handled \textit{separately} and \textit{independently} of the mixture coefficients in $\overline{S}_0$ in Extended RSDFO. Since the evolution of component densities of $\tilde{S}_\alpha$ is handled by the RSDFO (Algorithm \ref{alg:greda}) separately, it remains to  describe the evolution of mixture coefficients.


Given that for each center indexed by $\alpha$ we quantify the quality of solutions around it by the expected fitness $E_\alpha$, it is natural to assign to the individual mixture coefficients $\varphi_\alpha$ a value proportional to $E_\alpha$. In particular,
in line \ref{alg:prac:mixture} of Algorithm \ref{alg:ereda} the mixture coefficients are updated via the following equation:
\begin{align}
\varphi_\alpha^{k+1} = \frac{E_\alpha }{\sum_{\alpha \in {\Lambda^{k+1}}} E_\alpha} \label{eqn:mixture-evo-reduced} \quad .
\end{align}
where $E_\alpha$ is the expected fitness of $\tilde{B}_\alpha$ described in line \ref{alg:prac:sr} in Algorithm \ref{alg:ereda} above. Note that due to the assumption that $f$ is strictly positive, the expected fitness $E_\alpha$ is also strictly positive as well for all $\alpha$. A detailed theoretical justification will be discussed in Section \ref{sec:theo}. The case where we have a minimization problem is discussed in Remark \ref{rmk:evo_min}, in particular Equation \eqref{eqn:min_mixturecoeff}.



\subsubsection{Monotone improvement on expected fitness}
Recall the discussion in Section \ref{sec:shortcoming:rsdfo}: given two search centroids (or search iterates) $x_k, x_{k+1}$ in $M$ in Riemannian adaptated optimization algorithms,, the tangent spaces $T_{x_k} M$ and $T_{x_{k+1}} M$ centered at  $x_k, x_{k+1}$  respectively are \textit{different} (disjoint) spaces. In the case of RSDFO, the locally inherited parametrized probability densities (see Section \ref{app:localdistn}) over the tangent spaces $T_{x_k} M$ and $T_{x_{k+1}} M$ have \textit{different} supports, thus they belong to different statistical manifolds / models. In ``classical" RSDFO, the search information therefore cannot be compared beyond the objective value of the search iterate.

On the other hand in Extended RSDFO, the family of mixture densities $\mathcal{L}_V$ relates the sets of parametrized densities across separate tangent spaces. This allows us to derive the following result:

\begin{proposition}
\label{rmk:monoincreasing} 
The sequence of expected fitness values across iterations of Extended RSDFO on a maximization problem is monotonic, non-decreasing. In other words, for each $k> 0$: 
\begin{align*}
E^k \leq E^{k+1} \quad ,
\end{align*}
where $E^k$ denote the expected fitness of the solutions obtained in the $k^{th}$ iteration by Extended RSDFO.
\end{proposition}
\begin{proof}
 Let $ \mathcal{L}_{V^k}$ be the family of mixture densities defined in Equation \eqref{eqn:mixture-distn-mfold}, and let $\xi^k := \left\{  \left\{\varphi^k_\alpha \right\}_{\alpha\in {\Lambda_V}}, \left\{\theta_k^\alpha \right\}_{\alpha\in {\Lambda^k}} \right\}$ denote the set of parameters of  $\mathcal{L}_{V^k} = \overline{S}^k_0 \times \bigoplus_{\alpha \in \Lambda^k} \tilde{S}_\alpha$. For each $p(x \vert \xi^k) := \sum_{\alpha \in \Lambda^k} \varphi^k_\alpha \cdot \tilde{p}(x \vert \theta_k^\alpha) \in \mathcal{L}_{V^k} \subset \operatorname{Prob}(V)$, the expected fitness \cite{wierstra2008fitness} of $f$ over $p(x \vert\xi^k)$ (supported in $V^k$) on Riemannian manifold $M$ is given by:
\begin{align}
\label{eqn:sr} 
E^k &:= \int_{x\in V^k} f(x) p(x \vert\xi^k) dx   \\
&= \sum_{\alpha \in \Lambda^k} \varphi^k_\alpha \cdot \int_{x\in \tilde{B}_\alpha} f(x)   \tilde{p}(x \vert \theta_k^\alpha) dx \quad , 
\nonumber
\end{align} 
where $\tilde{B}_\alpha :=  \exp_{x_\alpha}\left(B(\vec{0},j_{x_\alpha})\right)$ and $\tilde{p}(x\vert \theta_k^\alpha) \in \tilde{S}_\alpha := \exp_{x_\alpha}^{-1^*}S_\alpha$ are inherited parametrized probability densities described in Section \ref{app:localdistn}. 

We now show that the sequence of expected fitness values across iterations of Extended RSDFO is monotonic, non-decreasing \footnote{Recall we assume the optimization problem to be a maximization problem.}: $E^k \leq E^{k+1}$ for $k>1$.

Let $X^k \cup \hat{X}^{k+1}$ denote the interim centroids of the $k^{th}$ iteration, and let $\hat{V}^{k+1}$ denote the search region on $M$ corresponding to the interim centroids described in Equation \eqref{eqn:searchvks}. The expected fitness of the $k^{th}$ iteration is given by: 
\begin{align} 
\label{eqn:conv-expectedfitness}
E^k &=  \int_{\hat{V}^{k+1}} f(x) \sum_{\alpha \in {\Lambda^k}}  \varphi_\alpha \tilde{p}(x \vert\theta^\alpha) dx \nonumber + \int_{\hat{V}^{k+1}} f(x) \sum_{\alpha \in {\hat{\Lambda}^k}\setminus \Lambda^{k}}  \varphi_\alpha \tilde{p}(x \vert\theta^\alpha) dx  \nonumber \\
&= \sum_{\alpha \in {\Lambda^k}} \varphi_\alpha \cdot E^k_\alpha + 0 \quad ,
\end{align} 
where $E_\alpha^k := \int_{\exp_{x_\alpha}\left(B(\vec{0},j_{x_\alpha})\right)}f(x) \tilde{p}(x\vert\theta^\alpha) dx$ for $\alpha \in \Lambda^k$ for simplicity (see Equation \eqref{eqn:expectedfitness:ball}). Similarly the expected fitness of  $\left(k+1\right)^{th}$ iteration can be found in the interim search region $\hat{V}^{k+1}$, and is given by $E^{k+1} = \sum_{\beta \in {\Lambda^{k+1}}} \varphi_\beta \cdot E^{k+1}_\beta$.

By the selection step of line \ref{alg:prac:cull} of Algorithm \ref{alg:ereda}, new centroids $X^{k+1} := \left\{ x_\alpha\right\}_{\alpha \in \Lambda_{k+1}}$ are selected based on the corresponding expected fitness, therefore for any $\beta \in \Lambda^{k+1}, \alpha \in \Lambda^{k}$, we would have: $E^{k+1}_\beta \geq E^k_\alpha$.
Hence, by Equation \eqref{eqn:conv-expectedfitness}, since $\sum_{\alpha \in {\Lambda^k}}  \varphi_\alpha = 1$ and $\sum_{\beta \in {\Lambda^{k+1}}} \varphi_\beta = 1$, we obtain the following desired inequality:
\begin{align*}
E^{k+1} = \sum_{\beta \in {\Lambda^{k+1}}} \varphi_\beta \cdot E^{k+1}_\beta \geq \sum_{\alpha \in {\Lambda^k}} \varphi_\alpha \cdot E^k_\alpha = E^k \quad .
\end{align*}
\end{proof}

\subsubsection{Exploration Distribution of Extended RSDFO}
\label{sec:prac:expdistn}

In this section we describe the exploration distribution $U$ in line \ref{alg:prac:generate} of Algorithm \ref{alg:ereda}. At the $k^{th}$ iteration, the exploration distribution $U$ is defined to be the uniform distribution on the boundary of the \textit{explored region up-to the $k^{th}$ iteration}. This allows us to explore new search regions in the manifold $M$ extending beyond the current explored region using only local computations.


%

For each $k>0$, recall from Equation \eqref{eqn:searchvks}, the interim search region $\hat{V}^{k+1}$ denoting the closed search region explored in the $k^{th}$ iteration of Extended RSDFO. Consider for each $k$, the union of \textit{all} search regions of $M$ explored by the algorithm \emph{up-to} the $k^{th}$ iteration:
\begin{align}
\label{eqn:explored:wk}
W^k = \bigcup_{j = 0}^k \hat{V}^{j+1} \quad .
\end{align}
The exploration distribution $U$ in line \ref{alg:prac:generate} of Algorithm \ref{alg:ereda} is given by the uniform distribution over the boundary of the explored region up to the $k^{th}$ iteration.: 
\begin{align*}
U := \operatorname{unif}\left( \partial\left(W^k \right) \right) \quad .
\end{align*}

Points from the exploration distribution $U$ can be generated by an acceptance-rejection approach using local computations on the tangent spaces, this is illustrated in Figure \ref{fig:expdistn} below:
\begin{figure}[hbt!]
    \centering
    \def\svgwidth{\columnwidth}
    \includegraphics[width=0.45\textwidth]{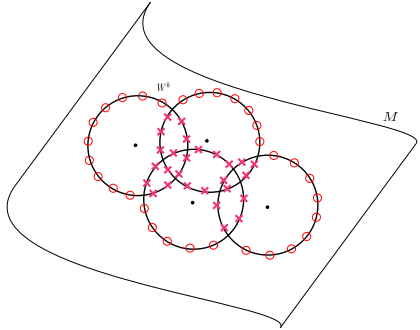} 
    \caption{An simplified illustration of the exploration distribution sampling using acceptance-rejection method. We first sample uniformly from the boundary of each of the geodesic balls, with centroids labelled by the filled circles. Points sampled from the boundary of the geodesic balls that are also in the interior of $W^k$ will be close to at least one of the centroids, hence are discarded according to Equation \eqref{eqn:expdistn:rej}. Therefore points labelled with hollow circles are accepted, whereas the crossed points are rejected.}
    \label{fig:expdistn}
\end{figure}

More formally, using the notation from Section \ref{sec:prac:mixturecoeff} summarized in Table \ref{table:evonotation}: for $k>0$, let $X^k \cup \hat{X}^{k+1}$ denote the set of centroids explored in the $k^{th}$ iteration of Extended RSDFO. The set of \textbf{accumulated centroids} explored by Extended RSDFO \emph{up-to} the $k^{th}$ iteration (this includes the explored but discarded ones) is given by:
\begin{align}
\label{eqn:accum:centroids}
X^k_A := \bigcup_{j=0}^k X^j \cup \hat{X}^{j+1} \quad .
\end{align}
The set of accumulated centroids  $X_A^k$ is indexed by the finite set $\Lambda_A^k := \cup_{j=0}^k \hat{\Lambda}^{j+1}$. 


To generate points on the boundary of the explored region, we first sample points from the union of the boundary of explored geodesic balls (the geodesic spheres): 
\begin{align}
\label{eqn:unif:bdd:all}
y \sim \operatorname{unif} \left( \bigcup_{\alpha \in \Lambda_A^k} \partial\left( \tilde{B}_\alpha \right) \right) \quad ,
\end{align}
where $\tilde{B}_\alpha = \exp_{x_\alpha} \left(B(\vec{0},j_{x_\alpha}) \right)$ for simplicity. This sampling can be done with strictly local computations on the geodesic sphere $\partial\left( \tilde{B}_\alpha \right)$ within the tangent space $T_{x_\alpha} M$ around each accumulated centroids $x_\alpha \in X_A^k$ on the manifold.

It remains to reject points that lies in the interior of $W^k$. Since every geodesic ball in $M$ is also a metric ball in $M$ under Riemannian distance with the same radius \cite{lee2006riemannian}: for any  $x_\alpha \in X_A^k$, a point $y \in \tilde{B}_\alpha$ satisfies:
\begin{align*}
d\left(x_\alpha,y \right) \leq j_{x_\alpha} \leq \operatorname{inj}(x_\alpha) \quad ,
\end{align*}
with equality if only if $y\in \partial\left(\tilde{B}_\alpha \right)$, where $d$ denote the Riemannian distance function \cite{lee2006riemannian} on $M$ \footnote{This can be computed locally within the normal neighbourhood of $x_\alpha$}. Therefore to preserve sampled points on the boundary of the explored region, we reject $y$ that lies in the interior of \textit{any} geodesic ball. In particular from Equation \eqref{eqn:unif:bdd:all} we reject $y$ if it satisfies: there exists $x_\alpha \in X_A^k$ such that 
\begin{align}
\label{eqn:expdistn:rej}
d\left(x_\alpha,y \right) < j_{x_\alpha} \quad . 
\end{align} 
\subsubsection{Termination Criterion}
\label{sec:termin}
Due to the ``multi-layer" structure of Extended RSDFO (see Figure \ref{fig:full2}), we may choose different termination criteria depending on the usage  of the algorithm. In this paper we discuss two examples of termination criteria:

For practical implementations (see experiments in Section \ref{sec:experiment}), due to limitations of computation power: we say that Extended RSDFO terminates when the local RSDFO's terminate in their own accord. In other words, as local RSDFO's inherit the termination criteria from their pre-adapted counterparts, we say that Extended RSDFO terminates when \textit{all} the local RSDFO steams terminate around the current search centroids.

For the theoretical studies (for geometry of evolutionary step in Section \ref{sec:theo-evo-step}, and convergence in Section Section \ref{sec:conv} and Section \ref{appendix:proof:conv}), when $M$ is a compact connected Riemannian manifold: Extended RSDFO terminates when the boundary of the explored region (see Section \ref{sec:prac:expdistn}) is empty. 

\section{Geometry of Evolutionary Step of Extended RSDFO}
\label{sec:theo}
In this section, we consider a special case of Extended RSDFO. The aim is to discuss and flesh out the geometry and dynamics of the evolutionary step for Extended RSDFO (Algorithm \ref{alg:ereda}).

Let $M$ be a Riemannian manifold, the special case of Extended RSDFO differs from the general case in one aspect: the local component densities around the search centroids in $M$ are fixed. In particular, the local RSDFO will determine new search centroids \textit{without} changing the statistical parameter around the search centroids. This allows us to focus on the dynamics of mixture coefficients without considering the mixture component parameters generated/altered by the RSDFO core. 
\begin{remark}
\label{remark:rad} 
We may assume the radius of geodesic ball $j_{x_\alpha} \ll \operatorname{inj}(x_\alpha)$ is sufficiently small for all $\alpha \in \Lambda_M$, such that varying the mixture coefficients while keeping the mixture components density fixed will describe the overall density in sufficient detail. 
Suppose if we cover $M$ with sufficient number of mixture components, while fixing the individual mixture components and allowing the mixture coefficients to vary. The mixture density formed by fixed mixture components and free mixture coefficients will sufficiently describe the search distributions appropriate for our optimization problem. This is illustrated in Figure \ref{fig:mixture}.

\begin{figure}[hbt!]
    \centering
    \def\svgwidth{\columnwidth}
    \includegraphics[width=0.8\textwidth]{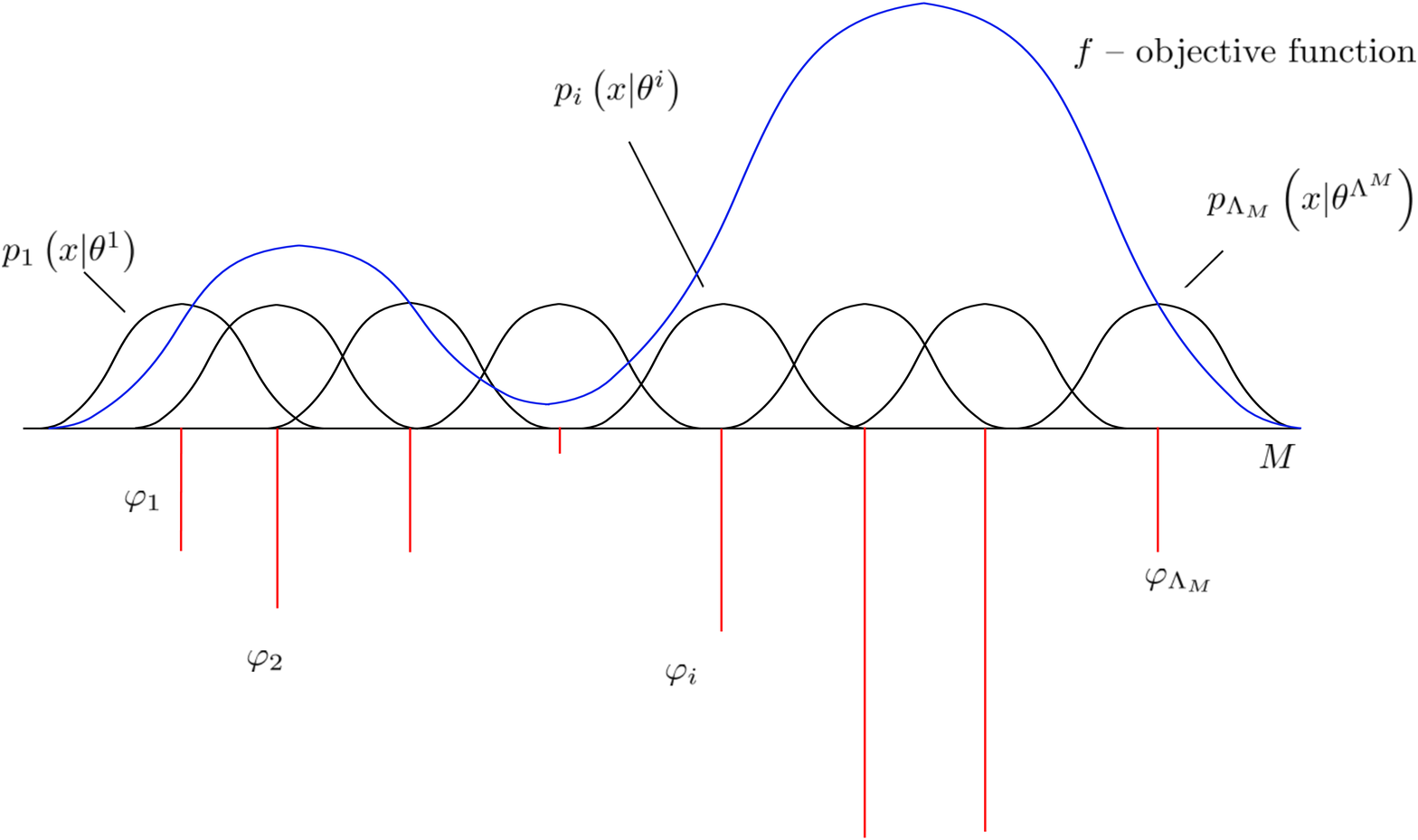} 
    \caption{A simplified illustrated of the mixture distribution in the special case if the search region is sufficiently well covered. The figure shows a $1$-dimensional manifold $M$, covered by fixed mixture componets and free mixture coefficents $\varphi_\alpha$. The promising regions in $M$ corresponding to the fitness landscape of objective function $f$ can be well expressed through the search distributions obtained solely by varying mixtures coefficients. The magnitude of mixture coefficients $\varphi$'s are showed in the figure above as downward bars.}
    \label{fig:mixture}
\end{figure}

\end{remark}

More formally, using the notation from Section \ref{sec:global:distnmfold}: Let $\mathcal{E}_M$ denote an orientation-preserving open cover of $M$ by geodesic balls indexed by $\Lambda_M$. Let $\Lambda_V \subset \Lambda_M$ be a finite subset of indices, and let $V$ denote a typical search region of Extended RSDFO define by the set of centroids $\left\{ x_\alpha \right\}_{\alpha \in \Lambda_V}$ (described as in Equation \eqref{eqn:typicalsearchregion}). Let $\mathcal{E}_V \subset \mathcal{E}_M$, denote the finite sub-cover of $\mathcal{E}_M$ over $V$. 


We begin the discussion by constructing a special case of the aforementioned (Section \ref{sec:global:distnmfold}) mixture family of densities where families of local component densities consisting of a \emph{single} fixed density: $S_\alpha = \left\{F_\alpha \right\}$. For example, for each $\alpha \in \Lambda_V$, we may let $F_\alpha := \overline{N}(0,I)$ denote the restriction of the spherical Gaussian density $N(0,I)$ to $B(0,j_{x_\alpha}) \subset T_{x_\alpha} M$ (properly renormalized). More generally, let the singleton $S_\alpha = \left\{F_\alpha \right\}$ denote the local component density over $B\left(\vec{0},j_{x_\alpha}\right) \subset T_{x_\alpha} M$. For each $\alpha \in \Lambda_V$, the inherited component density over $\exp_{x_\alpha}\left(B\left(\vec{0},j_{x_\alpha}\right)\right)$ consists of a single element: $\tilde{S}_\alpha = \left\{\tilde{F}_\alpha := \exp_{x_\alpha}^{-1^*}F_\alpha \right\}$.

The family of parametrized mixture densities on $M$ described in Equation \eqref{eqn:mixture-distn-mfold} thus reduces to the following special case: 
\begin{align}
\label{eqn:reduced-mixture-distn-mfold}
\overline{\mathcal{L}}_V := \left\{\sum_{\alpha\in \Lambda_V}  \varphi_\alpha \cdot \tilde{F}_\alpha \, \middle| \, \left\{\varphi_\alpha\right\}_{\alpha\in \Lambda_V} \in \overline{S}_0, \, \tilde{F}_\alpha \in \tilde{S}_\alpha \right\} \quad ,
\end{align}
where $\overline{S}_0 $ denote the closure of simplex described in Equation \eqref{eqn:s:closure}. The set of mixture densities $\overline{\mathcal{L}}_V$ over $V$ is therefore the product Riemannian manifold given by $\overline{\mathcal{L}}_V = \overline{S}_0 \times \bigoplus_{\alpha \in \Lambda_V} \tilde{S}_\alpha$, which in turn is diffeomorphic to $\overline{S}_0$. \footnote{This just means they can be thought of as the same space.} In this sense, this special case allows us to focus on the dynamics of mixture coefficients in $\overline{S}_0$.






%
 

%



In this paper we describe the geometry of $\overline{\mathcal{L}}_V \cong  \overline{S}_0$ with a novel modification of the Fisher information metric. This metric is used in \textbf{both} Extended RSDFO described in the previous section, and the following special case. There are two reasons to use the modified metric on the closure of the simplex $\overline{S}_0$:
\begin{enumerate}
    \item Fisher information metric is defined only on the interior of the simplex $S_0$, not on the closure of the simplex $\overline{S}_0$.
    \item Natural gradient ascent/descent \cite{calamai1987projected,amari1998natural} on the simplex $S_0$ under Fisher information metric favours the vertices of the simplex. On the other hand, Natural gradient ascent on the closure $\overline{S}_0$ under the modified metric, as we will discuss in the subsequent subsection, favours the interior point of the simplex with coordinates proportional to the relative weights of the vertices.
\end{enumerate}

We make explicit the geometrical structure of the modified metric as follows:  Let $F$ denote the Fisher information matrix on the mixture coefficient simplex $S_0$ (see for example \cite{lebanon2002learning}). Let $\epsilon_0 > 0$ be an arbitrarily small nonzero real number, consider the Riemannian metric on the \emph{closure} of the coefficient simplex $\overline{S}_0$ given by the matrix:
\begin{align}
\label{eqn:new-metric}
G := F^{-1} + \epsilon_0 \cdot I \quad .
\end{align}
Let $\xi := \left\{ \varphi_\alpha \right\}_{\alpha \in \Lambda_V}$ be the local coordinate of a point in $\overline{S}_0 \cong \overline{\mathcal{L}}_V$, and let $\left\{ \partial^\alpha \right\}_{\alpha \in \Lambda_V}$ denote the corresponding local coordinate frame on $T_\xi \overline{S}_0$. Let $Y := \sum_{\alpha \in \Lambda_V} y_\alpha \partial^\alpha$, $Z := \sum_{\alpha \in \Lambda_V} z_\alpha \partial^\alpha$ be vector fields over $\overline{S}_0$. The Riemannian metric $g$ on $\overline{S}_0$ corresponding to the modified matrix $G$ is given by:
\begin{align*}
g_{\xi}(Y,Z) := \sum_{\alpha \in \Lambda_V} y_\alpha z_\alpha \cdot \left(\varphi_\alpha + \epsilon_0 \right) \quad .
\end{align*} 

The corresponding divergence on $\overline{\mathcal{L}}_V \cong  \overline{S}_0$ is the Bregman  divergence \cite{amari2010information} given by the strictly convex function on $\overline{S}_0$:
\begin{align*}
\kappa: \overline{\mathcal{L}}_V \cong \overline{S}_0 &\rightarrow \mathbb{R} \\
\xi = \left(\varphi_\alpha \right)_{\alpha \in \Lambda_V} &\mapsto \frac{1}{6}  \sum_{\alpha \in \Lambda_V} \left( \varphi_\alpha + \epsilon_0 \right)^3 \quad ,
\end{align*}
hence the Bregman  divergence on $\overline{\mathcal{L}}_V$ is given by:
\begin{align}
\label{eqn:b-div-clS0}
D:\overline{\mathcal{L}}_V \times \overline{\mathcal{L}}_V \cong \overline{S}_0\times \overline{S}_0 &\rightarrow \mathbb{R} \nonumber \\
\left( \xi,\xi' \right) &\mapsto \kappa(\xi) - \kappa(\xi') - \left\langle \nabla\kappa(\xi'), \varphi - \xi' \right\rangle \quad .
\end{align}


The dualistic structure of the statistical manifold $\overline{\mathcal{L}}_V \cong \overline{S}_0$, induced by the Bregman  divergence described above, is in fact dually flat. Interested readers may find the general construction of the dualistic structure of $\overline{\mathcal{L}}_V \cong \overline{S}_0$ under the Bregman  divergence described above in \cite{amari2010information}.

Furthermore, let the function $f: M \rightarrow \mathbb{R}$ denote the objective function of an optimization problem over Riemannian manifold $M$. Let $\xi := \left\{ \varphi_\alpha \right\}_{\alpha\in {\Lambda_V}}$ denote the set of parameters of $\overline{\mathcal{L}}_V$, then for each $p(x,\xi) := \sum_{\alpha\in \Lambda_V} \varphi_\alpha \tilde{F}_\alpha \in \overline{\mathcal{L}}_V \subset \operatorname{Prob}(V)$, the expected fitness \cite{wierstra2008fitness} of $f$ over $p(x,\xi) \in \overline{\mathcal{L}}_V$ (supported in set $V \subset M$) in the special case is denoted by:
\begin{align}
\label{eqn:sr:theo}
J^V(\xi) &:= \int_{\operatorname{supp}(p(\xi))} f(x) p(x,\xi) dx \\
&= \int_{\operatorname{supp}(p(\xi))} f(x) \sum_{\alpha \in {\Lambda_V}}  \varphi_\alpha \tilde{F}_\alpha(x) dx \quad . \nonumber
\end{align} 
In practice, this integral can be approximated using $p(x , \xi)-$Monte Carlo samples for $x$.

\subsection{Geometry and Simplicial Illustration of Evolutionary step}
\label{sec:simplex}

By the product Riemannian structure of family of mixture densities $\overline{\mathcal{L}}_V$, the change of mixture coefficients in the evolutionary step of Extended RSDFO (Algorithm \ref{alg:ereda}) is \textit{independent} from the change of the parameter of the mixture component densities. Moreover, in the special case described above, the family of mixture component densities consists of a singleton: $\tilde{S}_\alpha = \left\{\tilde{F}_\alpha := \exp_{x_\alpha}^{-1^*}F_\alpha \right\}$, there will not be any changes in the parameter of the component densities. Therefore it remains to consider the changes in mixture coefficients in $\overline{S}_0$.

For each iteration, we transverse through three family of parametrized mixture densities on $M$, similar to the discussion of Section \ref{sec:prac:mixturecoeff}:

\begin{align}
\label{eqn:theo:mixture-distn-mfold}
\overline{\mathcal{L}}_V = \left\{\sum_{\alpha \in \Lambda_V} \varphi_\alpha \cdot \tilde{F}_\alpha \, \middle| \, \left\{\varphi_\alpha\right\}_{\alpha \in \Lambda_V}  \in \overline{S}_0, \, \tilde{F}_\alpha \in \tilde{S}_\alpha  \right\} \quad ,
\end{align}
where the search region $V$ and the corresponding centroids $X$, index sets $\Lambda_V$, and simplices $\overline{S}_0$ are summarized in Table \ref{table:evonotation} with $\overline{\mathcal{L}}_V$ replacing $\mathcal{L}_V$. Note that for all $\alpha \in \Lambda_V$, where $\Lambda_V = \Lambda^k, \hat{\Lambda}^{k+1},\Lambda^{k+1}$, the set of inherited density $\tilde{S}_\alpha = \left\{\tilde{F}_\alpha := \exp_{x_\alpha}^{-1^*}F_\alpha \right\}$ consists of a single density on $\tilde{B}_\alpha \subset M$ for each $\alpha$ (as described in the beginning of Section \ref{sec:theo}).


The evolution of  $\overline{\mathcal{L}}_V \cong \overline{S}_0$ is thus split into the following three parts, and the corresponding geometrical implications is illustrated in Figure \ref{fig:simplex} (extending Figure \ref{fig:simplexprac}). 

%

\begin{figure}[hbt!]
    \centering
    \def\svgwidth{\columnwidth}
    \includegraphics[scale=0.8]{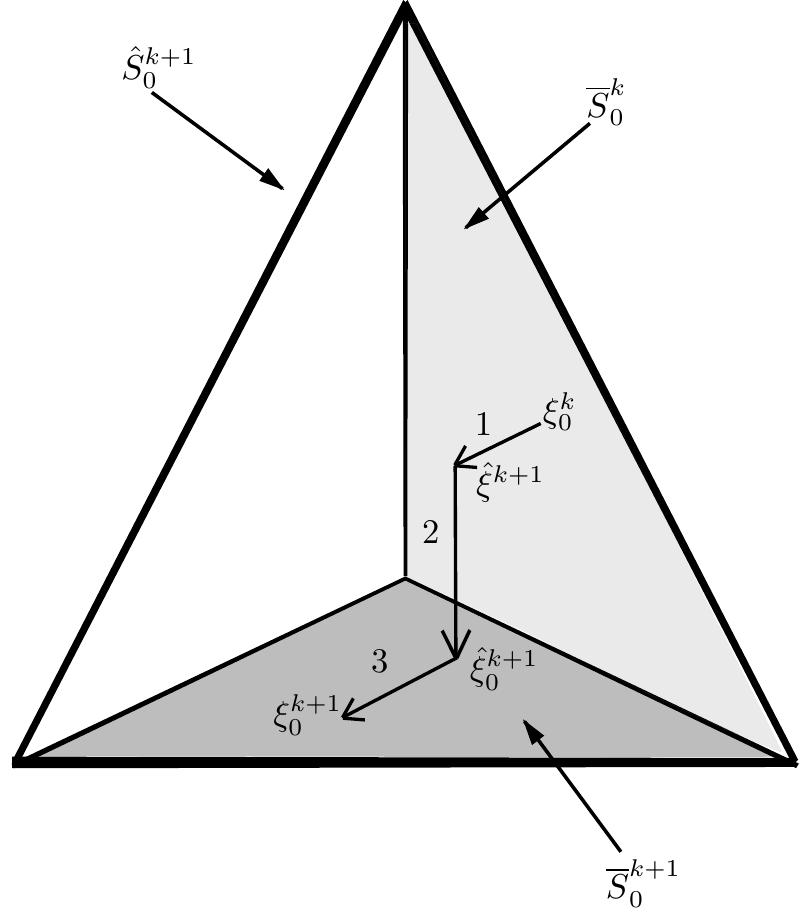}
    \caption{Simplicial illustration of the evolutionary step of Algorithm \ref{alg:ereda}. The number on the arrows corresponds to the Part number of the evolutionary step. The interim simplex $\hat{S}_0^k$ is illustrated by the entire simplex with 4 vertices. The subsimplices $\overline{S}_0^k$ and $\overline{S}_0^{k+1}$ is given by the left and bottom subsimplices (shaded) with 3 vertices respectively.}
    \label{fig:simplex}
\end{figure}

\begin{description}[align=left]
\item[Part one:] Determine the fixed point of natural gradient ascent \footnote{We perform natural gradient ascent for maximization problem and descent for minimization problem. For minimization case please refer to Remark \ref{rmk:evo_min}.} \cite{calamai1987projected,amari1998natural} in $\hat{S}_0^{k+1}$ with respect to the Riemannian structure described in Equation \eqref{eqn:new-metric} starting at $\xi_0^k \in \overline{S}_0^k \subset \hat{S}_0^{k+1}$. This is illustrated by arrow 1 in Figure \ref{fig:simplex}, and the fixed point is denoted by $\hat{\xi}^{k+1}$.
\item[Part two:] Project $\hat{\xi}^{k+1} \in \hat{S}_0^{k+1}$ onto $\overline{S}_0^{k+1}$, the $\left(N_{cull} -1 \right)$-dimensional faces of $\hat{S}_0^{k}$. This is illustrated by arrow 2 Figure \ref{fig:simplex}.
\item[Part three:] Determine the new fixed point $\xi_0^{k+1}\in \overline{S}_0^{k+1}$ of natural gradient ascent \footnote{Similar to part one: we perform natural gradient ascent for maximization problem and descent for minimization problem. For minimization case, please refer to Remark \ref{rmk:evo_min}. } \cite{calamai1987projected,amari1998natural} in $\hat{S}_0^{k+1}$ with respect to the Riemannian structure described in Equation \eqref{eqn:new-metric} on $\overline{S}_0^{k+1}$. This is labelled by arrow 3 on Figure \ref{fig:simplex}. 
\end{description}

$\overline{S}_0^k, \, \overline{S}_0^{k+1}$ are two faces of $\hat{S}_0^{k+1}$ sharing \textit{at most} $\left( N_{cull} + N_{random}\right) - N_{random}  = N_{cull}$ vertices \footnote{This can be zero, as illustrated in Figure \ref{fig:full2}}. A summary of the notation can be found in Table \ref{table:evonotation}.


\subsection{Detailed Description of Evolutionary step}
\label{sec:theo-evo-step}


In the previous section, we summarized and discussed geometrically the simplicial structure of the evolutionary step of the $k^{th}$ iteration of (the special case of) Extended RSDFO, a more rigorous mathematical discussions will be fleshed out in this section.

Let $f: M \rightarrow \mathbb{R}$ denote the objective function of a \emph{maximization} problem over Riemannian manifold $M$. We assume, without loss of generality, that $f$ is a strictly positive function by translation, i.e. $f(x) >0 $ for all $x\in M$. The case where we have a minimization problem is discussed in Remark \ref{rmk:evo_min}.


By line \ref{alg:prac:interim} of Algorithm \ref{alg:ereda}, at the $k^{th}$ iteration the algorithm produces a set of intermediate centroids $X^k \cup \hat{X}^{k+1} \subset M$. The search region of the $k^{th}$ iteration is thus given by the  subset $\hat{V}^{k+1} := \cup_{\alpha \in \hat{\Lambda}^{k+1}} \exp_{x_\alpha}\left( B(0,j_{x_\alpha}) \right)$ of $M$.

Let $\mathcal{E}_{\hat{V}^{k+1}} := \left\{ \left( \exp_{x_\alpha}, B(0,j_{x_\alpha}) \right) \right\}_{X_\alpha \in X^k \cup \hat{X}^{k+1}} \subset \mathcal{E}_M$ denote the finite subset of the orientation-preserving open cover $\mathcal{E}_M$ of $M$. Consider the ``special case" family of mixture densities over $\hat{V}^{k+1}$, denoted by $\overline{\mathcal{L}}_{\hat{V}^{k+1}}$, defined as in Equation \eqref{eqn:reduced-mixture-distn-mfold}. In the special case, the locally inherited component densities consist of a single element, therefore it suffices to consider the dynamics on the simplex of mixture coefficients.


We describe the three parts of the evolutionary step of Extended RSDFO (Algorithm \ref{alg:ereda}) in detail:

\textbf{Part one}: determine the fixed point of the natural gradient ascent in $\hat{S}_0^{k+1}$ on the modified maximization problem \cite{wierstra2008natural} : 
\begin{align*}
\max_{x \in M} f(x) \mapsto \max_{\xi^k \in \hat{S}_0^k} J^{\hat{V}^{k+1}} \left( \xi^k \right) \quad ,
\end{align*}
where  $\xi^k := \left( \varphi_\alpha \right)_{\alpha\in {\hat{\Lambda}^{k+1}}}$ denote the set of parameters of $\overline{\mathcal{L}}_{\hat{V}^{k+1}} \cong \hat{S}_0^{k+1}$. $J^{\hat{V}^{k+1}} \left( \xi^k \right)$ is the expected fitness of $f$ over $p(x,\xi^k) \in \overline{\mathcal{L}}_{\hat{V}^{k+1}}$ described in Equation \eqref{eqn:sr:theo}.

Let $\xi_0^k := \begin{cases}
      \varphi_\alpha^k & \text{if } \alpha \in \Lambda^k \subset \hat{\Lambda}^{k+1}\\
      0 & \text{otherwise. } 
    \end{cases}\, \in \overline{S}_0^{k} \subset \hat{S}_0^{k+1}$ denote the initial point on the subsimplex  $\overline{S}_0^{k}$ of $\hat{S}_0^{k+1}$. Let $\tilde{\nabla}_{\xi^k} J^{\hat{V}^{k+1}}\left( \xi^k \right)$ denote the natural gradient of $J^{\hat{V}^{k+1}}\left( \xi^k \right)$ on $\overline{\mathcal{L}}_{\hat{V}^{k+1}}$, and $\operatorname{proj}_{\hat{S}_0^{k}}$ the projection mapping onto the interim simplex of mixture coefficients $\hat{S}_0^{k}$ described above. We perform constraint natural gradient ascent \cite{calamai1987projected,amari1998natural} on $J^{\hat{V}^{k+1}} \left( \xi^k \right)$ at $\xi_0^k$, summarized by Algorithm \ref{alg:ng} below:

\begin{algorithm}[hbt!]
 \KwData{Initial point $\xi_0^k \in \overline{S}_0^{k} \subset \hat{S}_0^{k+1}$, objective function $J^{\hat{V}^{k+1}}\left( \xi^k \right)$, step size $s_i$. Set iteration count $i=0$}
 initialization\;
 \While{stopping criterion not satisfied}{
$\overline{\xi}_{i+1}^{k} =  \xi_i^k + s_i \cdot \tilde{\nabla}_{\xi^k} J^{\hat{V}^{k+1}}\left( \xi_i^k \right)$ \label{alg:ng:1} \;
$\xi_{i+1}^{k} = \operatorname{proj}_{\hat{S}_0^{k+1}} \left(\overline{\xi}_{i+1}^k \right)$ \label{alg:ng:2} \;
  i = i+1\;
 }
  \caption{Natural Gradient Ascent}
\label{alg:ng}
 \end{algorithm}
Given sufficiently small step sizes $s_i >0$, the fixed point of Algorithm \ref{alg:ng} can be determined explicitly.\footnote{The step sizes $s_i$'s are chosen to be sufficiently small such that the orthogonal projection $\operatorname{proj}_{\hat{S}_0^{k+1}} \left(\overline{\xi}_{i+1}^k \right)$ along the normal vector $\left[ 1,1,\ldots,1 \right] = \vec{1}$ of $\hat{S}_0^{k+1}$ onto the simplex $\hat{S}_0^{k+1}$ remains in the $\hat{S}_0^{k+1}$} In particular, a point $\xi_i^k \in \hat{S}_0^{k+1}$ is the fixed point of (lines \ref{alg:ng:1} and \ref{alg:ng:2} of) Algorithm \ref{alg:ng} if and only if $\operatorname{proj}_{\hat{S}_0^{k+1}} \left(\xi_i^k + s_i \cdot \tilde{\nabla}_{\xi^k} J^{\hat{V}^{k+1}}\left( \xi^k \right) \right) =\xi_i^k $. 

The above condition is satisfied if and only if $\tilde{\nabla}_{\xi^k} J^{\hat{V}^{k+1}}\left( \xi_i^k \right)$ is parallel to the normal vector: $\left[ 1,1,\ldots,1 \right]$ of $\hat{S}_0^{k+1}$ of the interim simplex $\hat{S}_0^{k+1}$.

Therefore to determine the fix point, we perform the following computation: for $k>0$, at the $k^{th}$ iteration of the special case of Extended RSDFO (Algorithm \ref{alg:ereda}), the inverse of interim Riemannian metric matrix (Equation \eqref{eqn:new-metric}) $\hat{G}_k^{-1}$ at a point $\xi_i^k = \left\{\varphi_\alpha^k \right\}_{\alpha \in \hat{\Lambda}^{k+1}}$ on the interim simplex $\hat{S}_0^{k+1}$ is given by :
\begin{equation}\hat{G}_k^{-1} :=  
 \left[  \begin {array}{ccccc} \frac{1}{\varphi_1^k + \epsilon_0 } & 0 & \cdots & \cdots &0 \\ 
0 & \frac{1}{\varphi_2^k + \epsilon_0 } & 0 & \cdots & 0 \\
\vdots & \ddots &  \ddots & \ddots &0\\ 
0 & \cdots & 0 & \frac{1}{\varphi^k_{\hat{\Lambda}^{k+1}-1} + \epsilon_0 } &0 \\
\noalign{\medskip}0 & 0 &\cdots &0& \frac{1}{\varphi^k_{\hat{\Lambda}^{k+1}} + \epsilon_0 }
\end {array}
 \right]\quad .
\end{equation}
The natural gradient \cite{amari1998natural} $\tilde{\nabla}_{\xi^k} J^{\hat{V}^{k+1}}\left( \xi_i^k \right)$ is therefore given by:
\begin{align*}
\tilde{\nabla}_{\xi^k} J^{\hat{V}^{k+1}}\left( \xi_i^k \right) &= \hat{G}_k^{-1} \cdot \nabla_{\xi^k} J^{\hat{V}^{k+1}}\left( \xi_i^k \right) \\
&= \hat{G}_k^{-1} \cdot \left(\frac{\partial}{\partial \varphi^k_\alpha} \int_x f(x) \sum_{\alpha \in \hat{\Lambda}^{k+1}}  \varphi_\alpha^k \tilde{F}_\alpha(x) dx \right)_{\alpha \in {\hat{\Lambda}^{k+1}}} \\
&= \hat{G}_k^{-1} \cdot \left( \int_x f(x) \tilde{F}_\alpha(x) dx  \right)_{\alpha \in {\hat{\Lambda}^{k+1}}} \\
&= \hat{G}_k^{-1} \cdot \left( E_\alpha^k \left[f\right] \right)_{\alpha \in {\hat{\Lambda}^{k+1}}} \\
&= \left( \frac{1}{\varphi_\alpha^k+\epsilon_0} \cdot E_\alpha^k \left[f\right] \right)_{\alpha \in {\hat{\Lambda}^{k+1}}} \quad .
\end{align*}

For simplicity and when the context is clear, for $\alpha \in \hat{\Lambda}^{k+1}$ we let  $E_\alpha := E_\alpha^k \left[f\right]$ denote the expected fitness of $f$ over the geodesic ball $ \exp_{x_\alpha}\left( B(0,j_{x_\alpha}) \right)$  on the $k^{th}$ iteration of Extended RSDFO  (see Equation \eqref{eqn:expectedfitness:ball}). Since $f$ is strictly positive, we may assume by translation that $E_\alpha > 0$ for all $\alpha$. Therefore the natural gradient $\tilde{\nabla}_{\xi^k} J^{\hat{V}^{k+1}}\left( \xi_i^k \right)$ is parallel to $\vec{1}$ if and only if for all $\alpha \in {\hat{\Lambda}^{k+1}}$:
\begin{align}
\label{eqn:parallel_theo}
\frac{1}{\varphi_\alpha^k+\epsilon_0} \cdot E_\alpha = c\cdot 1 \, ,  \, \, c\in \mathbb{R}_+ \quad .
\end{align}

Assume without loss of generality that $c =1$, the above condition is then satisfied if and only if $\varphi_\alpha^k = E_\alpha - \epsilon_0$, for $\alpha \in {\hat{\Lambda}^{k+1}}$. Moreover, since the mixture coefficients $\left( \varphi_\alpha^k \right)_{\alpha \in {\hat{\Lambda}^{k+1}}} \in \hat{S}_0^{k+1}$ belong to a simplex, they must sum to $1$: $\sum_{\alpha \in {\hat{\Lambda}^{k+1}}} \varphi_\alpha^k = 1$. The renormalized set of mixture coefficients thus becomes:
\begin{align}
\varphi_\alpha^k &= \frac{E_\alpha - \epsilon_0}{\sum_{\alpha \in {\hat{\Lambda}^{k+1}}} \left(E_\alpha - \epsilon_0 \right)} \label{eqn:renorm-mixture} \\
&=\frac{E_\alpha - \epsilon_0}{\sum_{\alpha \in {\hat{\Lambda}^{k+1}}} \left( E_\alpha\right) - \vert \hat{\Lambda}^{k+1} \vert \cdot \epsilon_0} \quad . \nonumber
\end{align}

An demonstration of the natural gradient of \textbf{Part One} (and similarly \textbf{Part three}) on the interim simplex $\hat{S}_0^k$ is give in Figure \ref{fig:fixedpoint} below:

\begin{figure}[hbt!]
    \centering
    \def\svgwidth{\columnwidth}
    \includegraphics[scale=0.7]{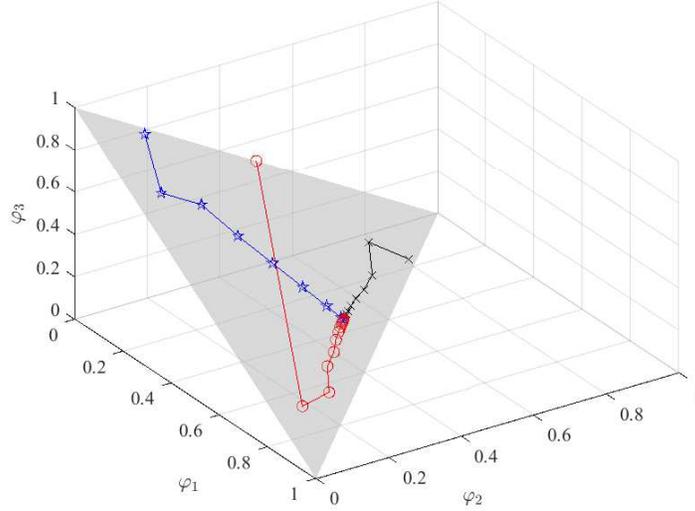}
    \caption{A demonstration of natural gradient ascent under the metric $G$ (Equation \eqref{eqn:new-metric}) in Part One of evolutionary step of Algorithm \ref{alg:ereda}. The shaded surface represents the interim simplex $\hat{S}_0^k$ with three vertices $\varphi_1,\varphi_2,\varphi_3$. This search paths in the figure illustrates how the  fixed point $\hat{\xi}^{k+1}$ of natural gradient ascent under $G$ is attained regardless of where the initial point is located in the simplex, whether it is an interior point or a boundary point.}
    \label{fig:fixedpoint}
\end{figure}

\textbf{Part two}: Let $\hat{\xi}^{k+1} := \left( \varphi_\alpha^k \right)_{\alpha \in {\hat{\Lambda}^{k+1}}} = \left( \frac{E_\alpha - \epsilon_0}{\sum_{\alpha \in {\hat{\Lambda}^{k+1}}} \left(E_\alpha - \epsilon_0\right)} \right)_{\alpha \in {\hat{\Lambda}^{k+1}}}$ denote the renormalized interim point described in Equation \eqref{eqn:renorm-mixture} above. We now project this point to a $N_{cull}$-dimensional subsimplex $\overline{S}_0^{k+1}$ of $\hat{S}_0^{k+1}$. 

We first sort the interim mixture coefficients  $\hat{\xi}^{k+1} := \left( \varphi_\alpha^k \right)_{\alpha \in {\hat{\Lambda}^{k+1}}}$ in ascending order (or descending order for minimization problem). Let the indices $\Lambda^{k+1} \subset \hat{\Lambda}^{k+1}$ denote the the largest (or smallest for minimization problem) $N_{cull}$ elements. The renormalized interim point is thus projected to the subsimplex $\overline{S}_0^{k+1}$ by:
\begin{align*}
{\xi}_0^k := \begin{cases}
      \frac{\varphi_\alpha^k}{\sum_{\alpha \in \Lambda^{k+1}} \varphi_\alpha^k} & \text{if } \alpha \in \Lambda^{k+1} \subset \hat{\Lambda}^{k+1} \\
      0 & \text{otherwise. } 
    \end{cases}\, \in \overline{S}_0^{k+1} \subset \hat{S}_0^{k+1} \quad .
\end{align*}
The subsimplex $\overline{S}_0^{k+1}$ therefore represents the set of mixture coefficients of the next iteration.
%

\textbf{Part three}: Finally, we determine the fixed point of natural gradient ascent (Algorithm \ref{alg:ng}) on the subsimplex $\overline{S}_0^{k+1} \subset \hat{S}_0^{k +1 }$ with the projected interim point $\xi_0^k$ as the initial point.
By the derivations in part one, we compute the fixed point of the natural gradient ascent in the subsimplex $\overline{S}_0^{k+1}$ in a similar fashion. The inverse of Riemannian metric matrix $G_{k+1}^{-1}$ at a point $\xi_i^{k+1} = \left\{\varphi_\alpha^k \right\}_{\alpha \in \Lambda^{k+1}}$ on\textit{ the subsimplex} $\overline{S}_0^{k+1}$ is given by :

\begin{equation*}G_{k+1}^{-1} := 
 \left[  \begin {array}{ccccc} \frac{1}{\varphi_1^k + \epsilon_0 } & 0 & \cdots & \cdots &0 \\ 
0 & \frac{1}{\varphi_2^k + \epsilon_0 } & 0 & \cdots & 0 \\
\vdots & \ddots &  \ddots & \ddots &0\\ 
0 & \cdots & 0 & \frac{1}{\varphi^k_{\Lambda^{k+1}-1} + \epsilon_0 } &0 \\
\noalign{\medskip}0 & 0 &\cdots &0& \frac{1}{\varphi^k_{\Lambda^{k+1}} + \epsilon_0 }
\end {array}
 \right] \quad ,
\end{equation*}
and the fixed point $\xi^{k+1} := \left( \varphi_\alpha^{k+1} \right)_{\alpha \in {\Lambda^{k+1}}}$ in $\overline{S}_0^{k+1}$ is therefore given by (similar to Equation \eqref{eqn:renorm-mixture}):
\begin{align}
\varphi_\alpha^{k+1} = \frac{E_\alpha - \epsilon_0}{\sum_{\alpha \in {\Lambda^{k+1}}} \left(E_\alpha - \epsilon_0 \right)} \label{eqn:mixture-end-evo} \quad .
\end{align}

\begin{remark}
\label{rmk:coeff-coincide}
The renormalized fixed point $\xi^{k+1}$ from Equation \eqref{eqn:mixture-end-evo} is an interior point of the simplex $\overline{S}_0^{k+1}$. Notice if for all $\alpha \in \Lambda^{k+1}$, the constant $\epsilon_0 <0$ is sufficiently small, such that it satisfies: $\vert \Lambda^{k+1} \vert \cdot \epsilon_0 \ll E_\alpha$, then we would have $\varphi_\alpha^k \cong \frac{E_\alpha}{\sum_{\alpha \in {\Lambda^{k+1}}} E_\alpha}$, and we retrieve Equation \eqref{eqn:mixture-evo-reduced}. This fixed point therefore directly reflects the relative expected fitness of each component $\exp_{x_\alpha}\left(B(0,j_{x_\alpha})\right)$ of the search region $V^{k+1} = \cup_{\alpha \in \Lambda^{k+1}} \exp_{x_\alpha}\left( B(0,j_{x_\alpha}) \right)$.
Hence the metric introduced in Equation \eqref{eqn:new-metric} describes the geometry of the evolutionary step of Extended RSDFO and  the use of the relative fitness of mixture components in  Equation \eqref{eqn:mixture-evo-reduced} is rigorously derived from first principles. 
\end{remark}
 
\begin{remark}
\label{rmk:evo_min}
In the case where we have a minimization problem over $M$, we would mirror the assumptions and operations of the maximization case. We assume, by translation, that the objective function is strictly negative: i.e. $f(x) <0$ for all $x \in M$, equivalently this means $-f(x)>0$ for all $x\in M$.

The natural gradient ascent in \textbf{Part One} (and similarly \textbf{Part three}) are replaced by natural gradient \emph{descent} over the following converted minimization problem:
\begin{align*}
\min_{x \in M} f(x) \mapsto \min_{\xi^k \in \hat{S}_0^k} J^{\hat{V}^{k+1}} \left( \xi^k \right) \quad .
\end{align*}
Since mixture coefficients $\left( \xi^k \right)$ in the simplices points towards the negative quadrant (see for example Figure \ref{fig:fixedpoint}) under natural gradient \emph{descent} direction, the normal vector of the simplices is taken to be $-\vec{1}$ instead. Equation \eqref{eqn:parallel_theo} thus becomes:
\begin{align*}
\frac{1}{\varphi_\alpha^k+\epsilon_0} \cdot E_\alpha = c\cdot -1 \, ,  \, \, c\in \mathbb{R}_+ \quad .
\end{align*}
The assumption that $f$ is strictly negative implies the expected value of $f$ over any geodesic ball $ \exp_{x_\alpha}\left( B(0,j_{x_\alpha}) \right)$ is strictly negative. In other words, the negative expected value is always positive: $- E_\alpha >0$ for all $\alpha$, which in turn implies  $\varphi_\alpha^k = -E_\alpha - \epsilon_0 > 0$. Equation \eqref{eqn:renorm-mixture} (similarly Equation \eqref{eqn:mixture-end-evo} and Equation \eqref{eqn:mixture-evo-reduced}) will then become:
\begin{align}
\label{eqn:min_mixturecoeff}
\varphi_\alpha^k &= \frac{-E_\alpha - \epsilon_0}{\sum_{\alpha' \in {\Lambda}} \left(-E_\alpha - \epsilon_0 \right)}  \cong 
\frac{-E_\alpha}{\sum_{\alpha \in {\Lambda}} -E_\alpha} \quad .
\end{align} 
Note that since we assumed $f(x)$ is strictly negative, 
$\varphi_\alpha^k$ can be equivalently expressed as
$|E_{\alpha}|/(\sum_{\alpha \in \Lambda} |E_\alpha|)$.
The rest of the arguments and computation are the same.
\end{remark}

\section{Convergence of Extended RSDFO on Compact Connected Riemannian Manifolds}
\label{sec:conv}


In this section we discuss the convergence behavior of Extended RSDFO on compact connected Riemannian manifolds. In particular, we show how Extended RSDFO converges globally eventually within \textit{finitely many} steps on compact connected Riemannian manifolds under the assumption that the optima is attainable in the manifold. The formal definitions and detailed proof will be provided in Section \ref{appendix:proof:conv}.

Global convergence of SDFOs on Euclidean spaces, such as Estimation of Distribution Algorithms (EDA) on $\mathbb{R}^n$ studied in \cite{zhang2004convergence}, is derived based on the explicit relations between the sampling probability distribution on the $k^{th}$ and the $\left(k+1\right)^{th}$ iteration of the EDA algorithm. This relation in turn depends on the chosen selection scheme, and more importantly: the fact that the selection distributions share the same support in the Euclidean space. Similarly, convergence results of Evolutionary Strategies \cite{beyer2014convergence} on Euclidean spaces share the same underlying assumption: selection distributions across iterations share the same support. Therefore, the selection distributions belong to the same statistical manifold.
%
%

%

However, such global convergence results of EDAs cannot be translated directly to Riemannian manifolds. Let $M$ be a Riemannian manifold and two points $x_k, x_{k+1}$ in $M$, the tangent spaces $T_{x_k} M$ and $T_{x_{k+1}} M$ centered $x_k, x_{k+1}$  respectively, are different (disjoint) spaces. Locally inherited parametrized probability densities (described in Section \ref{app:localdistn}) over $T_{x_k} M$ and $T_{x_{k+1}} M$ have different supports thus belong to different families. On the other hand, in Extended RSDFO, family of parametrized densities on separate tangent spaces can be related by mixture densities (described in Section \ref{sec:global:distnmfold}), and the quality of solutions is monotonically non-decreasing from iteration to iteration (Proposition \ref{rmk:monoincreasing}).


Suppose a new boundary exploration point is generated for each iteration of Extended RSDFO, and suppose the algorithm terminates only if no boundary point is available (i.e. when the boundary of the explored region is empty). The explored region of Extended RSDFO consists of the union of geodesic balls in $M$ with non-zero radius. Let $j_M$ denote the smallest of the radii, then the sequence of ``exploration centroids", generated on the boundary of the previous explored region, must be at least $j_M$ apart.

The sequence of boundary exploration points cannot continue indefinitely, otherwise this sequence will have no limiting point in $M$, contradicting the fact the $M$ is compact. Therefore the explored region generated by Extended RSDFO must exhaust the manifold $M$ in finitely many steps. Together with the use of elitist selection of the mixture components, the global optima will eventually be attainable in the explored region within finitely many steps.

\subsection{Detailed Exposition of Convergence Behaviour of Extended RSDFO}
\label{appendix:proof:conv}

In this section, we provide the detailed proof of the convergence of Extended RSDFO in connected compact Riemannian manifolds discussed in Section \ref{sec:conv}. In particular, we show that if Extended RSDFO generates a boundary exploration point for each iteration  (Section \ref{sec:termin}), then it converges eventually globally in finitely many steps in compact connected Riemannian manifolds.

Let $M$ be a Riemannian manifold. For $k>0$, let $X_A^k$ denote the finite set of accumulated centroids described in Equation \eqref{eqn:accum:centroids}. Let $\mathcal{E}_M$ denote the orientation-preserving open cover of $M$ described in Section \ref{sec:global:distnmfold} with set of centroids given by  $X_M := \left\{ x_\alpha \right\}_{\alpha \in \Lambda_M} \subset M$. The sets of search centroids $X_M$ and $X_A^k$ describes an open cover of $M$:

\begin{definition}
For each $k>0$, let $X_M^k := X_M \cup X_A^k$. The \textbf{accumulative open cover} of $M$ described by $X_M^k$  is given by: 
\begin{align*}
\mathcal{E}^k_M := \left\{ \left( \exp_{x_\alpha}, B(\vec{0},j_{x_\alpha})  \right) \right\}_{\alpha \in \Lambda^k_M} \quad ,
\end{align*}
where $j_{x_\alpha} \leq \operatorname{inj}(x_\alpha)$.
\end{definition}
Let $V^k$ denote the search region of the $k^{th}$ iteration of Extended RSDFO defined in Equation \eqref{eqn:searchvks}. Using the notation of Table \ref{table:evonotation} in Section \ref{sec:prac:mixturecoeff}, $V^k$ is defined by the set of centroids $X^k \subset X_M^k$ indexed by the finite subset $\Lambda^k \subset \Lambda_M^k$ 


Let the function $f: M \rightarrow \mathbb{R}$ denote the objective function of an optimization problem over $M$. Assume without loss of generality that the optimization problem is a maximization problem and $f$ is strictly positive. Furthermore, we assume that $f$ is bounded above in $M$ and that there exists an attainable global optima $x^* \in M$ such that: $f(x^*) \geq f(x) $ for all $x\in M$. 

The existence of a global optima implies for all $k>0$, there exists $\alpha^* \in \Lambda^k_M$ such that $E_{\alpha^*} := \int_{\exp_{x_\alpha^*}\left(B(\vec{0},j_{x_\alpha^*} ) \right)}f(x) p(x \vert\theta^{\alpha^*}) dx \geq E_\beta$ for all $\beta \in\Lambda_M$. We shall denote this by $E^* := E_{\alpha^*}$ for simplicity.

Motivated by the convergence condition described in \cite{zhang2004convergence}, we describe a global convergence condition of Extended RSDFO:
\begin{definition}
Extended RSDFO \textbf{converges globally eventually} if it satisfies the following condition:
\begin{align}
\label{eqn:conv-cond}
\lim_{k\rightarrow \infty} \sup_{\alpha \in \Lambda^k} E_\alpha^k = E^* \quad ,
\end{align}
where $E_\alpha^k := \int_{\exp_{x_\alpha}\left(B(\vec{0},j_{x_\alpha} ) \right)}f(x) p(x \vert\theta^\alpha) dx$ for $\alpha \in \Lambda^k$.
\end{definition}

This means asymptotically, the best component of the accumlative open cover $\mathcal{E}^k_M$ will be contained in the preserved centroids contained in $V^k$. In other words, there exists an integer $K_0$ such that for $k \geq K_0$ the set of retained centroids will contain the ``best" neighbourhood $\exp_{x_\alpha^*}\left(B(\vec{0},j_{x_\alpha^*} ) \right)$. 

Recall from Section \ref{sec:termin} that Extended RSDFO terminates when the boundary of the explored region is empty. The following theorem guarantees Extended RSDFO will terminate in finitely many steps in compact connect Riemannian manifolds.

\begin{restatable}{theorem}{thmconv}
\label{thm:conv}
Let $(M,g)$ be a compact connected Riemannian manifold. Suppose for each iteration Extended RSDFO generates at least one exploration point from the boundary of the explored region \footnote{See Section \ref{sec:prac:expdistn}} when the boundary is nonempty.  Then Extended RSDFO satisfies the global convergence condition Equation (\ref{eqn:conv-cond}) on $M$ within finitely many steps. That is, there exists finite integer $N > 0$ such that: 
\begin{align*}
\sup_{\alpha \in \Lambda^N} E_\alpha^N = \lim_{k\rightarrow \infty} \sup_{\alpha \in \Lambda^k} E_\alpha^k = E^* \quad ,
\end{align*}
where $E_\alpha^k := \int_{\exp_{x_\alpha}\left(B(\vec{0},j_{x_\alpha} ) \right)}f(x) p(x \vert\theta^\alpha) dx$ for $\alpha \in \Lambda^k$.
\end{restatable}

\begin{proof}

%

For each $k \geq 0$, consider the region in $M$ explored by Extended RSDFO up to the $k^{th}$ iteration: $W^k = \cup_{j = 1}^k \hat{V}^j \subset M$ (see Equation \eqref{eqn:explored:wk}). $W^K$ is the union of closed geodesic balls in $M$ hence a nonempty closed subset of $M$. If $W^k \subsetneq M$, then the boundary $\partial \left( W^k \right)$ of $W^k$ must be nonempty. \footnote{Otherwise $W^k = \operatorname{int}(W^k)$, meaning $W^k$ is a clopen proper subset of $M$, which contradicts the fact that $M$ is connected.} This means we can always find a search centroid on the boundary $\partial \left( W^k \right)$ \textit{unless} $W^k = M$.


\begin{claim}
Extended RSDFO explored $M$ in finitely many steps. In other words, there exists finite integer $N>0$ such that $W^N = M$.
\end{claim}
\begin{claimproof}
Let $k \geq 0$, we generate a sequence of points $\left\{y^k\right\}_{k\in \mathbb{N}}$ in $M$ as follows:
\begin{itemize}
    \item If $W^k \subsetneq M$, then by the argument above the boundary of the explored region is nonempty: $\partial \left( W^k \right) \neq \emptyset$. Hence by the assumption Extended RSDFO can generate a new exploration centroid from $\partial \left( W^k \right)$. In particular, at the $k^{th}$ iteration, there exists  $y^{k+1} \in X^k \cup \hat{X}^{k+1} \subset M$ in the set of interim centroids such that $y^{k+1} \in \partial \left( W^k \right)\subset M$. Let $j_M := \inf_{x\in M} j_x \leq \inf_{x\in M} \operatorname{inj}\left(x\right)$, then $j_M >0$ and we obtain the following relation: 
\begin{align*}
W^{k} \subset \exp_{y^{k+1}}\left(B(\vec{0},j_M) \right) \cup W^k \subset W^{k+1} \quad ,
\end{align*} 
where $\exp_{y^{k+1}}\left(B(\vec{0},j_M )\right)$ denotes the closed geodesic ball around $y^{k+1}$ with radius $j_{M} \leq j_{y^{k+1}} \leq \operatorname{inj}\left(j_{y^{k+1}} \right)$. Note that geodesic balls in $M$ are also metric balls of the same radius \cite{lee2006riemannian}.

\item Otherwise if $W^k = M$, then $\partial \left( W^k \right) = \emptyset$. We therefore set $y^{k+1} = y^k$ and $W^{k+1} = W^{k} = M$.
\end{itemize}

Since $M$ is compact, it admits the Bolzano-Weierstrass Property \cite{rudin1964principles}: every
sequence in $M$ has a convergent subsequence. In particular the sequence $\left\{y^k\right\}_{k\in \mathbb{N}}$ has a convergent subsequence $\left\{y^{k^\ell}\right\}_{{k^\ell} \in \mathbb{N}}$. This means there exists $\hat{y} \in M$ such that:
\begin{align*}
    \hat{y} = \lim_{\ell \rightarrow \infty} y^{k^\ell} \quad .
\end{align*}
Let $d_g : M \times M \rightarrow \mathbb{R}$ denote the Riemannian distance induced by the Riemannian metric $g$ of $M$. The above equation implies there exists finite $k^\alpha, k^\beta >0$, $k^\alpha < k^\beta$ such that $d_g\left(y^{k^\alpha},\hat{y}\right), d_g\left(y^{k^\beta},\hat{y}\right) < \frac{j_M}{2}$, hence $d_g\left(y^{k^\alpha},y^{k^\beta}\right) < j_M$ \footnote{In the normal neighbourhood, the Euclidean ball in the tangent space is mapped to the geodesic ball of the same radius. \cite{lee2006riemannian}}. Therefore we have:
\begin{align*}
    y^{k^\beta} \in \operatorname{int}\left(\exp_{y^{k^\alpha}}\left(B(\vec{0},j_M )\right)\right) \subset \exp_{y^{k^\alpha}}\left(B(\vec{0},j_M )\right) \subset W^{k^\alpha} \subset W^{k^\beta -1} \quad .
\end{align*}
which implies $y^{k^\beta} \in \operatorname{int}\left(W^{k^\beta -1}\right)$.

Finally, the result above means $W^{k^\beta -1} = M$: Suppose by contrary that $W^{k^\beta -1} \neq M$, then by construction of the sequence $\left\{y^k\right\}_{k\in \mathbb{N}}$: $y^{k^\beta} \in \partial\left(W^{k^\beta -1}\right) := W^{k^\beta -1} \setminus \operatorname{int}\left(W^{k^\beta -1}\right)$, contradicting the result above.

Therefore there exists finite $k^\beta > 0$ such that $W^{k^\beta -1} = M$, which completes the proof of the claim.

\end{claimproof}

%



Let $\Lambda_A^k = \cup_{j=0}^k \hat{\Lambda}^{j+1}$ denote the index of accumulated centroids (see Section \ref{sec:prac:expdistn} in particular Equation \eqref{eqn:accum:centroids}), by the claim above, there exists $N>0$ such that:
\begin{align*}
W^N = \lim_{k\rightarrow\infty} W^k = \lim_{k\rightarrow\infty} \bigcup_{\alpha \in \Lambda_A^k} \exp_{x_\alpha}\left(B(\vec{0},j_{x_\alpha}) \right) = M
\end{align*}


Suppose $W^N$ is indexed by $\Lambda^N$. Since the the fittest neighbourhoods are preserved by the evolutionary step (see Section \ref{sec:theo-evo-step}, and Proposition \ref{rmk:monoincreasing}) of Extended RSDFO, we obtain the following relation:
\begin{align*}
\sup_{\alpha \in \Lambda^N} E_\alpha^N = \lim_{k\rightarrow \infty} \sup_{\alpha \in \Lambda^k} E_\alpha^k &= \lim_{k \rightarrow \infty}\sup_{\alpha \in \Lambda^k} \int_{\tilde{B}_\alpha  \subset V^k} f(x)  p(x\vert\theta^\alpha) dx  \\
&= \lim_{k \rightarrow \infty}\sup_{\alpha \in \Lambda_A^k} \int_{\tilde{B}_\alpha \subset W^k} f(x)  p(x\vert \theta^\alpha) dx  \quad , \\
& = \sup_{\alpha \in \Lambda_A^k} \int_{\tilde{B}_\alpha  \subset M} f(x)  p(x \vert \theta^\alpha) dx \\ 
&= E^* \quad .
\end{align*}
where  $\tilde{B}_\alpha :=  \exp_{x_\alpha}\left(B(\vec{0},j_{x_\alpha})\right)$ for simplicity.
\end{proof}

\section{Experiments}
\label{sec:experiment}
%
%
%

In this section we compare Extended RSDFO against three manifold optimization algorithms from the literature reviewed in \ref{ch:litreview:manopt}: manifold trust-region method (RTR) \cite{absil2007trust,absil2009optimization}, Riemannian CMA-ES (RCMA-ES) \cite{colutto2010cma}, and Riemannian Particle Swarm Optimization (R-PSO) \cite{borckmans2010modified, borckmans2010oriented}. The three algorithms are detailed in \ref{sec:rtr}, \ref{sec:rpso} and \ref{sec:rcma} respectively. For the remainder of the section, we will discuss the local restrictions and implementation assumptions of the  manifold optimization algorithms in the literature, and discuss why they are not necessary for Extended RSDFO (Section \ref{sec:manopt:assmption}). The experiments in this section are implemented with ManOpt version 4.0 \cite{boumal2014manopt} in MATLAB R2018a. For the following experiments we demonstrate the optimization methods on \textbf{minimization} problems over lower dimensional manifold such as: $M = S^2$ the unit $2$-sphere (Section \ref{sec:sphere}), Grassmannian manifolds $Gr_{\mathbb{R}}(p,n)$ (Section \ref{sec:gr}), and Jacob's ladder (Section \ref{sec:jl}). Jaccob's ladder, in particular, is a manifold of potentially infinite genus and cannot be addressed by traditional (constraint) optimization techniques on Euclidean spaces, this experiment necessitates the development of manifold optimization algorithms. Expected fitness in the experiments are approximated by Monte Carlo sampling with respect to their corresponding probability distributions. Finally in Section \ref{sec:discuss:exp}, we discuss the experimental results in detail in the end of the section.

We begin by describing the setup of the algorithms.

\textbf{ \textit{Setup of Extended RSDFO}} (Section \ref{ch:paper2}): the RSDFO core is given by the generic framework described in Algorithm \ref{alg:greda}. Let $M$ denote a Riemannian manifold. The families of parametrized probability distributions $S_x$ on each tangent space $T_x M$ are set to be multivariate Gaussian densities centered at $\vec{0} \subset T_x M$ restricted to $B(\vec{0},\operatorname{inj}(x)) \subset U_x \subset T_{x} M$ (properly renormalized), denoted by $N(\cdot \vert \vec{0},\Sigma)$ . The initial statistical parameter is set to be $\left(\vec{0},\Sigma\right) = \left(\vec{0},\text{Id}_{\text{dim}(M)}\right)$ (spherical Gaussian), and the covariance matrix is iteratively re-estimated (see Algorithm \ref{alg:greda}). Each iteration of local RSDFO is given a fixed budget of sample ``parents" and the estimation is based on a fixed number of ``off-springs". The local RSDFO algorithms are terminated if the improvement of solutions is lower than the threshold $10^{-14}$. Extended RSDFO is terminated when all local RSDFO algorithms around the current centroids terminate.
The non-increasing non-zero function $\tau(k)$ modulating the amount of exploration on the $k^{th}$ iteration of Extended RSDFO (see Section \ref{sec:additionalparameter}) is be given by:
\begin{align*}
    \tau(k) = \frac{6}{10}\cdot e^{-\left(0.015\right)\cdot k} \quad .
\end{align*} 

Furthermore, for the sake of computational efficiency, the exploration distribution $U = \operatorname{unif}\left( \partial\left(W^k \right) \right)$ in line\ref{alg:prac:generate} of Extended RSDFO (see Section \ref{sec:prac:expdistn}) will be simplified as follows:
For each step we pick $5$ explored centroids randomly. For each chosen centroids generate $50$ points on the \emph{sphere} of injectivity radius: $\exp_{x_\alpha} \left(B(\vec{0},\operatorname{inj}(x_\alpha)) \right)$. \footnote{Equivalently, this means we set $j_{x_\alpha} = \operatorname{inj}(x_\alpha)$ in Section \ref{sec:prac:expdistn}.} If a sufficient number of boundary points is accepted (according to Equation \eqref{eqn:expdistn:rej}), then the appropriate number of boundary points will be added to the current centroids. However, if all the sampled points are rejected, then Extended RSDFO will not sample new boundary points. The algorithm will proceed on with the current set of centroids.

\textbf{ \textit{Setup of RTR}} (\ref{sec:rtr}): The parameters of RTR are inherited from the classical Euclidean version of Trust Region \cite{nocedal2006numerical}, and are set to be their default values in ManOpt's implementation summarized in Table \ref{Table:RTR} below.
\begin{table}[hbt!]
\begin{tabular}{|l|l|l|}
\hline
Parameter           & Value                          & Meaning                                                        \\ \hline
$\overline{\Delta}$ & $\pi$                          & Upperbound on step length                                      \\ \hline
$\Delta_0$          & $ \frac{\overline{\Delta}}{8}$ & Initial step length                                            \\ \hline
$\rho'$             & 0.1                            & Parameter that determines whether the new solution is accepted \\ \hline
\end{tabular}
\caption{Parameters of RTR.}
\label{Table:RTR}
\end{table}

\textbf{\textit{ Setup of RCMA-ES}} (\ref{sec:rcma}) For RCMA-ES, the parameters of inherited directly from the Euclidean version \cite{hansen2006cma}. The parameters in the following experiments are set to be the default parameters described in \cite{colutto2010cma} Let $N$ denote the dimension of the manifold, the parameters of RCMA-ES are summarized in Table \ref{Table:cmaparamer:exp} below:
\begin{table}[hbt!]
\centering
\begin{tabular}{|l|l|p{0.3\textwidth}|}
\hline
Parameter          & Value                                                                                                                                                                            & Meaning                                          \\ \hline
$m_1$               & Depends on the manifold                                                                                                                                                                            & Number of sample parents per iteration                  \\ \hline
$m_2$               & $\frac{m_1}{4}$                                                                                                                                                                  & Number of offsprings.                            \\ \hline
$w_i$               & $\log\left( \frac{m_2 +1}{i}\right) \cdot \left(\sum_{j=1}^{m_2} \log\left( \frac{m_2 +1}{j}\right) \right)^{-1}$                                                                & Recombination coefficient                        \\ \hline
$m_{\text{eff}}$    & $ \left(\sum_{i=1}^{m_2} w_i^2 \right)^{-1}$                                                                                                                                     & Effective samples                                \\ \hline
$c_c$               & $\frac{4}{N+4}$                                                                                                                                                                  & Learning rate of anisotropic evolution path      \\ \hline
$c_\sigma$          & $\frac{m_{\text{eff}} +2}{N + m_{\text{eff}}+ 3}$                                                                                                                                & Learning rate of isotropic evolution path        \\ \hline
$\mu_{\text{cov}} $ & $m_{\text{eff}}$                                                                                                                                                                 & Factor of rank-$m_2$-update of Covariance matrix \\ \hline
$c_{\text{cov}}$    & $ \frac{2}{ \mu_{\text{cov}}\left(N+ \sqrt{2}\right)^2} + \left(1- \frac{1}{\mu_{\text{cov}}}\right)\min\left(1, \frac{2\mu_{\text{cov}}-1}{(N+2)^2 + \mu_{\text{cov}}} \right)$ & Learning rate of covariance matrix update        \\ \hline
$d_\sigma =$        & $ 1+ 2\max\left(0,\sqrt{\frac{m_{\text{eff}}-1}{N+1}} \right) + c_\sigma$                                                                                                        & Damping parameter                                \\ \hline
\end{tabular}
\caption{Parameters of RCMA-ES described in \cite{colutto2010cma}.}
\label{Table:cmaparamer:exp}
\end{table}

\textbf{\textit{Setup of R-PSO}} (\ref{sec:rpso}): Finally, the parameters of R-PSO are set to be the default parameters in ManOpt \cite{boumal2014manopt}  and the function of inertial coefficient is defined according to \cite{borckmans2010modified}:
\begin{align*}
    &\text{nostalgia coefficient} = 1.4, \quad \text{social coefficient} = 1.4, \quad \\
    &\text{inertia coefficient } = \text{monotonic decreasing linear function from } 0.9 \text{ to } 0.4 \quad .
\end{align*}

The parameters of the manifold optimization algorithms in the literature are chosen according to their original papers ((RTR) \cite{absil2009optimization}, (RCMA-ES) \cite{colutto2010cma}, and (R-PSO) \cite{borckmans2010modified}) as the dimensions of the experiments in this paper are similar the ones described in their original papers. 

\subsection{On the Assumptions of Manifold Optimization Algorithms in the Literature}
\label{sec:manopt:assmption}

Manifold optimization algorithms in the literature (described in Section \ref{sec:principle:manopt} and \ref{ch:litreview:manopt}) assume the manifolds to be complete. That is, they require the Riemannian exponential map to be defined on the entire tangent space. Given a tangent vector (search direction) of an arbitrary length, this assumption allows us to ``project" our search iterate along any such direction for time $1$.

This assumption is used in gradient-based line search such as RTR (\ref{sec:rtr}, \cite{absil2007trust}), in particular if the upperbound of step length is set to be larger than the injectivity radius, the algorithm can generate search iterates beyond the current normal neighbourhood.

In Extended RSDFO, this assumption is unnecessary: As all computations and estimations are done in a strictly local fashion around the normal neighbourhood, the search direction would not leave the geodesic ball of injectivity radius. This same observation applies to Riemannian CMA-ES (\ref{sec:rcma}, \cite{colutto2010cma}) as well. This means their assumption that the manifold is complete is not necessary for Extended RSDFO.


On the other hand, Riemannian PSO (\ref{sec:rpso}, \cite{borckmans2010modified}) requires a stronger assumption to operate properly: the Riemannian logarithm map to be defined between any two points. This assumption essentially requires any two points to be able to ``communicate" on the manifold via the $\log$ map, which is generally impossible.

In particular, the Riemannian logarithm map is \textit{only} defined in the region where Riemannian exponential map is a diffeomorphism (in other words it is only defined within the normal neighbourhood and undefined else where). Therefore for the Riemannian adapation of PSO described in \cite{borckmans2010modified}, the full PSO step is only possible if all points are within \textit{one single} normal neighourhood. This assumption is not necessary for Extended RSDFO and RCMA-ES.





For ``smaller" manifolds such as $n$-spheres and Grassmannian manifolds, the Riemannian logarithm map between (almost) any two points is well-defined due to a ``sufficiently global parametrization" or embedding from an ambient Euclidean space.\footnote{For Grassmannian manifolds see \cite{absil2004riemannian}.} However, for ``larger" manifolds such as Jacob's ladder, some components of the Riemannian PSO step will be impossible, as we shall discuss in Section \ref{sec:jl}.




\clearpage

\subsection{Sphere $S^2$}
\label{sec:sphere}
The first experiment considers the objective function discussed in \cite{colutto2010cma} on $(\theta,\phi)$, the spherical coordinates on $S^2$:
\begin{align*}
f: S^2 &\rightarrow \mathbb{R} \\
f(\theta,\phi) &= -2\,\cos \left( 6 \cdot\theta \right) +2\,\cos \left( 12 \cdot \phi \right) + \left( \frac{\pi}{12}-\phi \right)^{2} + \frac{7}{2}{\theta}^{2}+4 \quad .
\end{align*}
$f$ is a multi-modal problem with global minimum located at $(\theta^*,\phi^*) = \left(0,\frac{\pi}{12} \right)$ with objective value $0$. The heat map of the objective function is illustrated in Figure \ref{fig:s2:heat}:

\begin{figure}[hbt!]
    \centering
    \includegraphics[width=0.5\linewidth]{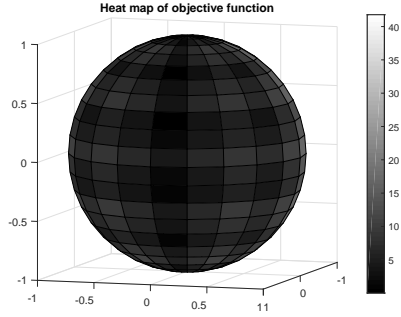}
    \caption{Heat map of $f$ on $S^2$}
    \label{fig:s2:heat}
\end{figure}


Injectivity radius on $S^2$ is $\operatorname{inj}(x) = \pi$ for all $x\in S^2$\cite{tron2013riemannian}, and the Riemannian exponential map of $S^2$ is given by \cite{absil2009optimization}:
\begin{align*}
\exp_x(v) = \left. \gamma_v(t) \right|_{t=1} = \left. \cos\left(\vert\vert v \vert \vert\cdot t \right) x + \sin\left(  \vert\vert v \vert \vert\cdot t\right) \frac{v}{\vert\vert v \vert \vert} \right|_{t=1} 
\end{align*}
where $\gamma_v(t)$ is the geodesic on $S^2$ starting at $x$ with initial velocity $v \in T_x S^2$ .  For the following experiment, the additional parameters of Extended RSDFO are given by:
\begin{align*}
    N_{random} = 2, \quad N_{cull} = 2 \quad .
\end{align*}
Each optimization algorithm is given a budget of 10000 function evaluations. Extended RSDFO initializes with $2$ initial centroids $X^0$ randomly generated on the sphere $S^2$. The local RSDFO core is given a budget of $50$ ``parents", and the parameter estimation will be based on the best $10$ ``offsprings".  RTR and RCMA-ES is initialized with one centroid randomly chosen from $X^0$. RCMA-ES is given a budget of $m_1 = 40$ parents per iteration. R-PSO initialized with $10$ copies of $X^0$, which allows $20$ ``agents" to perform $20$ function evaluations per iteration. The expected fitness of the stochastic algorithms are approximated using $10$ Monte Carlo samples with respect to their corresponding probability distributions. Each algorithm is repeated $200$ times with the parameters described in the beginning of this section. The gradient and the Hessian of the objective function is not provided. Each estimation of the gradient or the Hessian in RTR counts as a function evaluation respectively, even though in practice it may take up additional resources. The algorithms (and the RSDFO steams in Extended RSDFO)  terminate when they converge locally or when they are sufficiently ($10^{-6}$) close to the global optimum with objective value $0$.

Furthermore, we perform an additional set of experiments with a slightly relaxed exploration distribution described in Section \ref{sec:prac:expdistn}. In particular, we choose non-zero real number $1>\epsilon_b > 0$, such that for each iteration $k>0$ the ``exploration'' centroids (described in Section \ref{sec:prac:expdistn}) is generated from $\partial\left(\bigcup_{\alpha \in \Lambda_A^k} \tilde{B}_\alpha \right)$, with $\tilde{B}_\alpha := \exp_{x_\alpha} \left(B(\vec{0},\epsilon_b \cdot \operatorname{inj}(x_\alpha)) \right) \subset  \exp_{x_\alpha} \left(B(\vec{0},\operatorname{inj}(x_\alpha)) \right)$. 

The result of 200 experimental runs is summarized in Table \ref{Table:exp_s2} below. The second and third column of the Table \ref{Table:exp_s2} shows the number of local minimum and global minimum attained by the optimization methods within the 200 experiments respectively.

\begin{table}[hbt!]
\centering
\begin{tabular}{|l|l|l|l|}
\hline
Method & No. of local min & No. of global min & avg. $f$ eval \\ \hline
Ext. RSDFO with $\epsilon_b = 1$  & 72 & 128 &   5154  \\ \hline
RCMA-ES    & 104 & 96 &   5518.6 \\ \hline
R-PSO    & 33 & 167 & 10000   \\ \hline
RTR     & 192 & 8  &  3093.9  \\ \hline
Ext. RSDFO with $\epsilon_b = 0.4$   & 46  & 154 & 4648 \\ \hline    
\end{tabular}
\caption{Experimental results on $S^2$ with 200 runs.}
\label{Table:exp_s2}
\end{table}

From Table \ref{Table:exp_s2} above, we observe that Extended RSDFO with a basic RSDFO core (Algorithm \ref{alg:greda}) performs better than RCMA-ES even in the sphere. While R-PSO out-performs the rest of the rival algorithms at the expense of higher number of objective function evaluations, the slightly relaxed version of Extended RSDFO with $\epsilon_b = 0.4$ is close comparable using less than half the computational resources. We will discuss the experimental results in further detail in Section \ref{sec:sphere:gr:discuss}.


Experimental results of a sample run is provided in Figure \ref{fig:s2exp:a} below. In this experiment both versions of Extended RSDFO and RCMA-ES attained the global minimum, whilst R-PSO and RTR are stuck in local minima. The search centroids of model-based algorithms (Extended RSDFO and RCMA-ES) and sampled points of R-PSO are presented in Figure \ref{fig:s2exp:b} and Figure \ref{fig:s2exp:c} respectively. We observe that the movement of search centroids of model-based algorithms (Extended RSDFO and RCMA-ES) exhibit varing but converging sequence towards the global optima. On the other hand, the search agents of R-PSO gravitates towards the local min, similar to its Euclidean counterpart. 

\begin{figure*}[hbt!]
    \centering
    \begin{subfigure}[t]{\textwidth}
        \centering
        \includegraphics[width=1\textwidth]{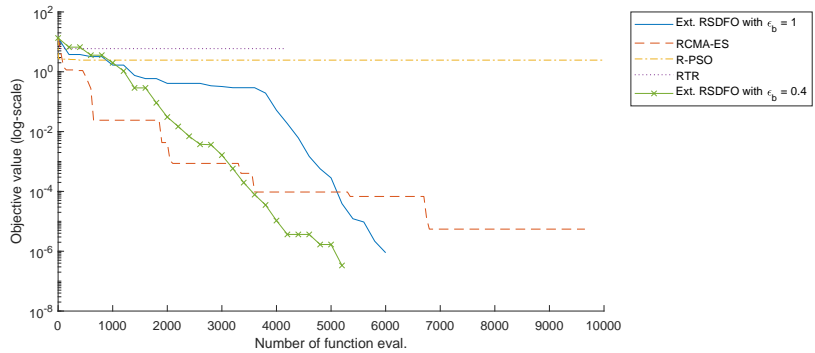}
        \caption{A sample comparison of quality of solution obtained by each algorithm.}    \label{fig:s2exp:a}
    \end{subfigure}%
    ~ \\
    \begin{subfigure}[t]{0.49\textwidth}
        \centering
        \includegraphics[width=1\textwidth]{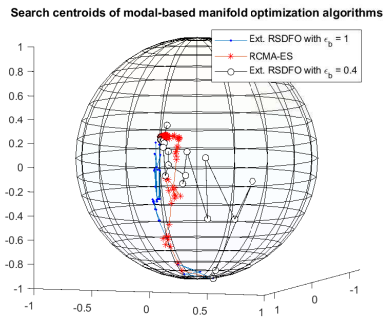}
        \caption{Movement of search centroids in model-based algorithms. Notice that even the algorithms share the same initial search centroid position at the bottom of the sphere, their search paths differ as they depends on the estimated statistical model. }   
        \label{fig:s2exp:b}
    \end{subfigure}
        ~ 
    \begin{subfigure}[t]{0.49\textwidth}
    \centering
       \includegraphics[width=1\textwidth]{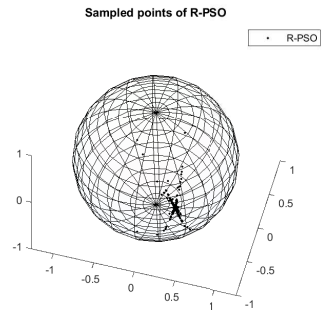}
        \caption{Search agents generated by R-PSO.}\label{fig:s2exp:c}
    \end{subfigure}
    \caption{Experimental results of one of the experiments on $S^2$. }
    \label{fig:s2:exp}
\end{figure*}

\clearpage
\subsection{Grassmannian Manifolds}
\label{sec:gr}
For experiments on Grassmannian manifolds $Gr_{\mathbb{R}}(p,n)$, we consider a composition of Gramacy-Lee Function \cite{benchmarkfcns,gramacy2012cases}, and the Riemannian distance function $d_{Gr}$ on $Gr_{\mathbb{R}}(p,n)$:
\begin{align*}
 f: Gr_{\mathbb{R}}(p,n) &\rightarrow \mathbb{R} \\
    f(x) &= \frac{\sin\left(10\pi \cdot d_{Gr}\left(x,I_{n\times p}\right)\right)}{2\cdot d_{Gr}\left(x,I_{n\times p}\right)} + \left(d_{Gr}\left(x,I_{n\times p}\right)-1\right)^4 \quad ,
\end{align*}
where $d_{Gr}$ denote the Riemannian distance function on $Gr_{\mathbb{R}}(p,n)$. The objective function is illustrated in Figure \ref{fig:grcost}, notice the $x$-axis is given by $d_{Gr}\left(x,I_{n\times p}\right)$.
\begin{figure}[hbt!]
    \centering
    \def\svgwidth{\columnwidth}
    \includegraphics[width=0.55\textwidth]{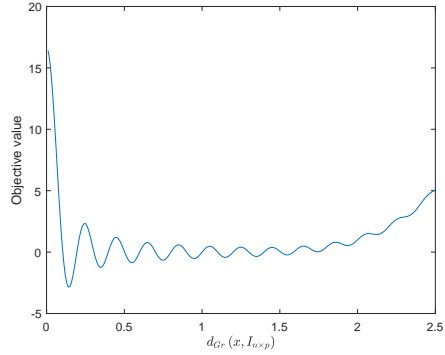}
    \caption{Illustration of objective function $f$ on $Gr_{\mathbb{R}}(p,n)$, notice the $x$-axis is given by the Riemannian distance $d_{Gr}$ of $x$ relative to $I_{n\times p}$.}
    \label{fig:grcost}
\end{figure}

The injectivity of Grassmannian manifolds $Gr_{\mathbb{R}}(p,n)$ is $\frac{\pi}{2}$ for all $x\in Gr_{\mathbb{R}}(p,n)$ \cite{prastaro1996geometry,tron2013riemannian}, and the Riemannian exponential map of is given by \cite{absil2004riemannian}:
\begin{align*}
\exp_x(v) = \left. \gamma_v(t) \right|_{t=1} = \left. \cos\left(\vert\vert v \vert \vert\cdot t \right) x + \sin\left(  \vert\vert v \vert \vert\cdot t\right) \frac{v}{\vert\vert v \vert \vert} \right|_{t=1}
\end{align*}
where $\gamma_v(t)$ is the geodesic on $Gr_{\mathbb{R}}(p,n)$ starting at $x$ with initial velocity $v \in T_x Gr_{\mathbb{R}}(p,n)$.  For the following experiment, the additional parameters of Extended RSDFO are given by:
\begin{align*}
    N_{random} = 2, \quad N_{cull} = 2 \quad .
\end{align*}
The initialization of the following experiments is similar to that of $S^2$ discussed in the previous section, the specification is included for the sake of completeness: Extended RSDFO initializes with $2$ initial centroids $X^0$ randomly generated on the sphere $S^2$. The local RSDFO core is given a budget of $120$ and $200$ ``parents" with parameter estimation based on the best $40$, $50$ ``offsprings" on $Gr_{\mathbb{R}}(2,4)$ and $Gr_{\mathbb{R}}(2,5)$ respectively. RTR and RCMA-ES is initialized with one centroid randomly chosen from $X^0$. RCMA-ES is given a budget of $m_1 = 80$ parents per iteration.  R-PSO initialized with $20$ copies of $X^0$, which allows $40$ ``agents" performing $40$ function evaluations per iteration. Each algorithm is repeated $100$ times with the parameters described in the beginning of this section. The gradient and the Hessian of the objective function is not provided. The expected fitness of the stochastic algorithms are approximated using $40$,$50$ Monte Carlo samples with respect to their corresponding probability distributions on $Gr_{\mathbb{R}}(2,4)$ and $Gr_{\mathbb{R}}(2,5)$ respectively. Each estimation of the gradient or the Hessian in RTR counts as a function evaluation respectively.

Furthermore, we once again perform an additional set of experiments with a slightly relaxed exploration distribution described in Section \ref{sec:prac:expdistn}. Following the discussion in the previous section, we choose non-zero real number $1>\epsilon_b > 0$, such that for each iteration $k>0$ the ``exploration'' centroids (described in Section \ref{sec:prac:expdistn}) is generated from $\partial\left(\bigcup_{\alpha \in \Lambda_A^k}  \tilde{B}_\alpha \right)$, with $\tilde{B}_\alpha := \exp_{x_\alpha} \left(B(\vec{0},\epsilon_b \cdot \operatorname{inj}(x_\alpha)) \right) \subset  \exp_{x_\alpha} \left(B(\vec{0},\operatorname{inj}(x_\alpha)) \right)$. 

We perform two sets of experiments on $Gr_{\mathbb{R}}(2,4)$, $Gr_{\mathbb{R}}(2,5)$ respectively and each optimization algorithm is given a budget of $24000$ and $40000$ function evaluations respectively. Due to the complexity of the problem,  the convergence criteria is slightly relaxed: the algorithms (and the RSDFO steams in Extended RSDFO) are said to converge globally if they are ``sufficiently close" ($10^{-8}$) to the global optimum with objective value $-2.87$, meaning the search iterate is within the basin of the global optimum. 

The result of 100 experimental runs on $Gr_{\mathbb{R}}(2,4)$ and $Gr_{\mathbb{R}}(2,5)$ is summarized in Tables \ref{Table:exp_gr_24}, \ref{Table:exp_gr_25} respectively below. The second and third column of the tables shows the number of local minimum and global minimum attained by the optimization methods within the 100 experiments respectively.

\begin{table}[hbt!]
\centering
\begin{tabular}{|l|l|l|l|}
\hline
Method & Local min & Global min & avg. $f$ eval \\ \hline
Ext. RSDFO with $\epsilon_b = 1$  & 4 & 96 &   24000  \\ \hline
RCMA-ES    & 61 & 39 &   9786.4  \\ \hline
R-PSO    & 3 & 97 &  24000  \\ \hline
RTR     & 100 & 0  &   22719.54 \\ \hline
Ext. RSDFO with $\epsilon_b = 0.5$   &  5 & 95 & 24000 \\ \hline    
\end{tabular}
\caption{Experimental results on $Gr_{\mathbb{R}}(2,4)$ with 100 runs.}
\label{Table:exp_gr_24}
\end{table}


\begin{table}[hbt!]
\centering
\begin{tabular}{|l|l|l|l|}
\hline
Method & Local min & Global min & avg. $f$ eval \\ \hline
Ext. RSDFO with $\epsilon_b = 1$  & 70 & 30 &   40000  \\ \hline
RCMA-ES    & 91 & 9 &  13271.2  \\ \hline
R-PSO    & 44 & 56 &  40000  \\ \hline
RTR     & 100 & 0  &  39445.47  \\ \hline
Ext. RSDFO with $\epsilon_b = 0.5$   & 59 & 41 & 40000 \\ \hline    
\end{tabular}
\caption{Experimental results on $Gr_{\mathbb{R}}(2,5)$ with 100 runs.}
\label{Table:exp_gr_25}
\end{table}

From Table \ref{Table:exp_gr_24} and Table \ref{Table:exp_gr_25}, we observe that both versions of Extended RSDFO are comparable to R-PSO, which out-performs the rival algorithms. While these three algorithms achieves a high success rate in determining the global solution of the experiments on  $Gr_{\mathbb{R}}(2,4)$, the success rates of the algorithms dropped as the dimension increases on  $Gr_{\mathbb{R}}(2,5)$. We will discuss the experimental results in further detail in Section \ref{sec:sphere:gr:discuss}.

Experimental results of a sample run on $Gr_{\mathbb{R}}(2,4)$ and $Gr_{\mathbb{R}}(2,5)$ is provided in Figure \ref{fig:grexp:a}  and Figure \ref{fig:grexp:b} respectively below. In the experiment on $Gr_{\mathbb{R}}(2,4)$, all algorithms \textit{except} RTR and RCMA-ES attained the global minimum. The plot shows how the convergence rate of the algorithms vary; Extended RSDFO takes more steps to explore the manifold before converging to an optimum. In the experiment on $Gr_{\mathbb{R}}(2,5)$, both versions of Extended RSDFO attained the global optimum, while the other algorithms converged to local optima. It is worth noting that, whilst RCMA-ES convergences sharply, it is prone to be stuck in local minima at higher dimensions.

\begin{figure*}[hbt!]
    \begin{subfigure}[t]{\textwidth}
        \centering
        \includegraphics[width=1\textwidth]{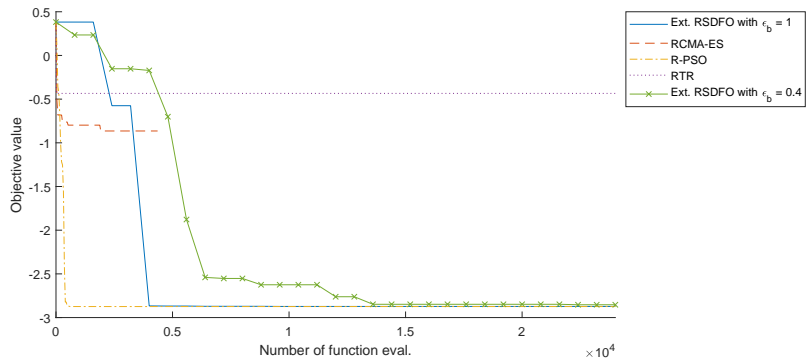}
        \caption{Comparison of quality of solution obtained by each algorithm in $Gr_{\mathbb{R}}(2,4)$.}   
        \label{fig:grexp:a}
    \end{subfigure}
        ~ 
    \begin{subfigure}[t]{\textwidth}
    \centering
       \includegraphics[width=1\textwidth]{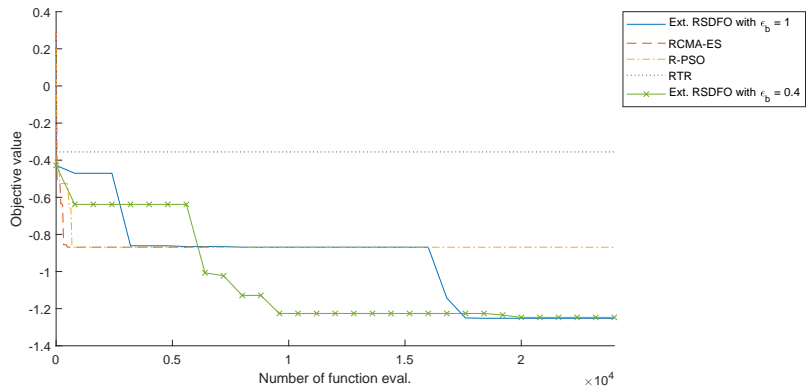}
        \caption{Comparison of quality of solution obtained by each algorithm in $Gr_{\mathbb{R}}(2,5)$.}\label{fig:grexp:b}
    \end{subfigure}
    \caption{Experimental results of one of the experiments on $Gr_{\mathbb{R}}(p,n)$. }
    \label{fig:gr:exp}
\end{figure*}

\clearpage

\subsection{Jacob's ladder}
\label{sec:jl}

In this section we discuss the experimental results of an optimization problem on a representation of Jacob's ladder surface -- a manifold of infinite genus \cite{ghys1995topologie}. Jacob's ladder will be represented by the connected sum of countable number of tori \cite{spivak1979comprehensive}, illustrated in Figure \ref{fig:donutchain} below. 
\footnote{The general geometrical structure of manifolds of infinite genus is beyond the scope of this paper. Interested readers are referred to \cite{phillips1981geometry,feldman1995infinite} and the relevant publications.}

\begin{figure}[hbt!]
    \centering
    \def\svgwidth{\columnwidth}
    \includegraphics[width=\textwidth]{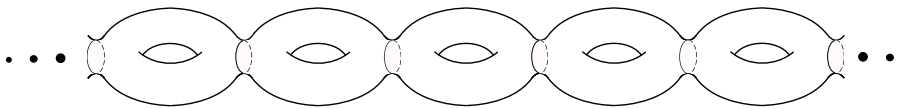}
    \caption{Illustration of Jacob's ladder surface as an infinite connected sum of tori.}
    \label{fig:donutchain}
\end{figure}

In certain conditions, manifold optimization can be performed in the framework of constraint optimization in a Euclidean space. Indeed, due to Whitney's embedding theorem \cite{whitney1944self,whitney1944singularities}, all manifolds can be embedded in a sufficiently large ambient Euclidean space. As traditional optimization techniques are more established and well-studied in Euclidean spaces, one would be more inclined to address optimization problems on Riemannian manifold by first finding an embedding onto the manifold, and then applying familar classical optimization techniques.

Riemannian manifolds discussed in the manifold optimization literature \cite{absil2009optimization,absil2019collection} are mostly compact manifolds such as $n$-spheres, Steifel manifolds, or Grassmannian manifolds. These manifolds all have global vectorial (matrix) representations, which provide a global parametrization of the manifold. Moreover, these manifolds are all compact manifolds, which can be expressed as zero set of polynomials also known as real affine algebraic varieties \cite{tognoli1973congettura,bochnak1989algebraic}. This means compact manifolds can be written as polynomial constraints in an ambient Euclidean space. 

To motivate the need for Riemannian manifold optimization, we present an optimization problem on Jacob's ladder as an example of a Riemannian manifold that does not have a global parametrization. Furthermore, Jacob's ladder is a surface of countably infinite genus, meaning it cannot be expressed as polynomial constraints in the manifold \interfootnotelinepenalty=10000 \footnote{The Euler characteristic $\chi$ of a manifold can be express as the alternating sum of Betti numbers:$\chi = b_0 - b_1 + b_2 - \ldots$, where $b_i$ denote the $i^{th}$ Betti number. In the case of Jacob's ladder we have $b_0 = 1 = b_2$, and $b_j = 0$ for $j > 3$ due to the dimension of Jacob's ladder. Let $g$ denote the genus of the surface, then the Euler characteristic is given by $\chi = 2-2\cdot g$. By combining the two equations we obtain $b_1 = 2\cdot g$. Furthermore, the (sum of) Betti number of algebraic varieties is bounded above \cite{milnor1964betti}:, which means a manifold of infinite genus cannot be a real affine algebraic variety.} It is therefore difficult to formulate optimization problems on Jacob's ladder as a constraint optimization problem. 

For the rest of the discussion and in the following experiment we assume no upperbound on the number of tori in the Jacob's ladder chain. A priori the optimization algorithms does not know the size of the tori chain and so the manifold in the optimization optimization problem cannot be formulated as constraints beforehand. In this situation, the only natural approach is to use Riemannian manifold optimization methods that ``crawl" on the manifold.



Jacob's ladder represented by an infinite connected sum of tori is a non-compact complete Riemannian manifold, as the connected sum of complete manifolds is also complete. The local geometry of each torus in the Jacob's ladder is given by $S^1_R \times S^1_r$, where $R>r$ and $S^1_R$ and $S^1_r$ denote the major and minor circles of radius $R$ and $r$ respectively. A torus can be locally parametrized by the angles $\left(\theta,\varphi\right)$, where $\theta$ and $\varphi$ denote the angle of the minor circle and major circle respectively. The local exponential map at a point $x$ on a torus can be decomposed as $\left(exp_x^{S^1_R}, exp_x^{S^1_r} \right)$, this is due to the direct sum nature of the product Riemannian metric described in Section \ref{Subsection:DualGeoLv}. If the geodesic ball of injectivity radius around a point $x$ intersects two tori in the Jacob's ladder, then the exponential map follows the transition function described below:

Let $T_1$ and $T_2$ denote two tori of the Jacob's ladder, and let $\psi_1,\psi_2$ denote smooth coordinate functions of $T_1$ and $T_2$ respectively.  On the gluing part of the adjunction space, the transition map $\psi_2 \circ \psi_1^{-1}$ is given by the identity map. This is illustrated in Figure \ref{fig:donutchainarrow}.
\begin{figure}[hbt!]
    \centering
    \def\svgwidth{\columnwidth}
    \includegraphics[width=0.5\textwidth]{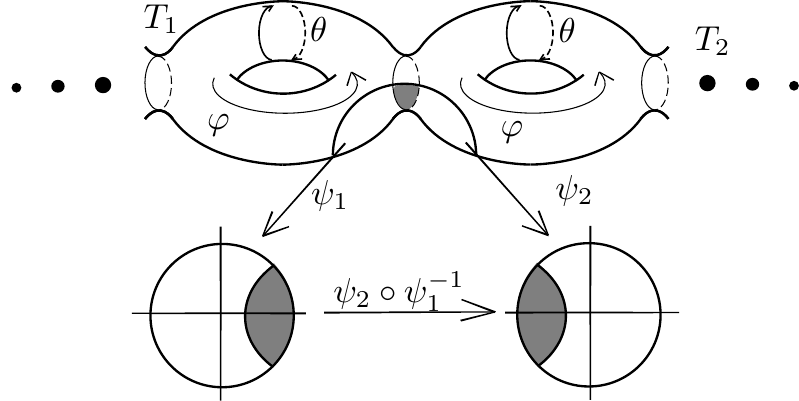}
    \caption{Illustration of minor and major rotational angles in the Jacob's ladder. The vertical smaller rotational angle represents the minor angle $\theta$, whereas the horizontal arc represents the major angle $\varphi$.}
    \label{fig:donutchainarrow}
\end{figure}

Denote by $T_i$ the $i^{th}$ torus in Jacob's ladder, denoted by $M$, described above. Due to the lack of global parametrization of $M$, we consider the \textbf{minimization} problem on $M$ with the following optimization function, composed of a ``global" part and a ``local part:
\begin{align}
f:M &\rightarrow \mathbb{R} \\
x &\mapsto f_G(x) \cdot f_L(x).
\end{align} 

$f_G$ denote the ``global" part of the objective function, which acts on the ``torus number" of a point $x\in M$. To be precise, the function is given by:
 \begin{align*}
 f_G: M &\rightarrow \mathbb{R} \\
n(x) &\mapsto := \begin{cases}
        0.05 & \text{if } n(x) = 0,15, and -25 \\
        \left[ -\left(\frac{\sin \left(\frac{4}{5}n(x) - 15\right)}{\frac{4}{5}\left(n(x)-15\right)} \right)^2 + 1.05 \right]\cdot \left| \frac{n(x)}{20}\right|^{\frac{1}{10}} 
                & \text{if } n(x) > \frac{15}{2} \\
        \left[ -\left(\frac{\sin \left(\frac{4}{5}n(x) +25 \right)}{\frac{4}{5}\left(n(x) + 25\right)} \right)^2 + 1.05 \right]\cdot \left| \frac{n(x)}{20}\right|^{\frac{1}{10}} 
                & \text{if } n(x) < -\frac{15}{2} \\
        -\left(\frac{\sin \left(\frac{4}{5}n(x) \right)}{\frac{4}{5}n(x)} \right)^2 \cdot \left| \frac{n(x)}{20}\right|^{\frac{1}{10}} & \text{otherwise,}
    \end{cases}
\end{align*}
where $n(x)$ is the integer denoting the numerical label of the torus in which the point $x \in M$ is located. The numerical label increases as we move towards the positive $x$-axis direction in the ambient Euclidean space and decreases as we move towards the negative $x$-axis direction. For instance, in the current implementation the torus centered at the origin is labelled torus number zero, and the next one towards the positive $x$-axis direction in the ambient Euclidean space is subsequently labelled $1$. The function $f_G$ is illustrated in Figure \ref{fig:fg}.


The local component of the objective function is given by Levy Function N.13 \cite{simulationlib}:
\begin{align*}
    f_L: T_i \setminus \bigcup_{j\neq i} T_j &\rightarrow \mathbb{R} \\
    \theta, \varphi &\mapsto \sin^2\left(3 \pi \theta\right) + \left(\theta-1 \right)^2 \left(1+\sin^2\left(3 \pi\varphi \right)\right)+ \left(\varphi-1\right)^2\left(1+\sin^2\left(2 \pi \varphi \right) \right) + 35 \quad .
\end{align*}
The function $f_L$ is illustrated in Figure \ref{fig:fl} on the $\left(\theta,\varphi\right)$ axes. Levy N.13 is a non-convex, differentiable, multimodal function \cite{benchmarkfcns} with one global minimum at radian angles $\left( \theta^*, \varphi^* \right) = \left( 1,1\right)$ with objective value $f_L \left( \theta^*, \varphi^* \right) = 0 + 35 = 35$.

The ``global-global" optima (the desired solution) is thus located in torus number $0,15,-25$ with minor-major rotational angles $\left( \theta^*, \varphi^* \right) = \left( 1,1\right)$ with objective value $0.05 \cdot 35 = 1.75$. The torus numbered $0,15,-25$ are therefore called ``optimal" for the following experiments.

\begin{figure}[hbt!]
    \centering
    \begin{minipage}{0.45\textwidth}
        \centering
        \includegraphics[width=0.9\textwidth]{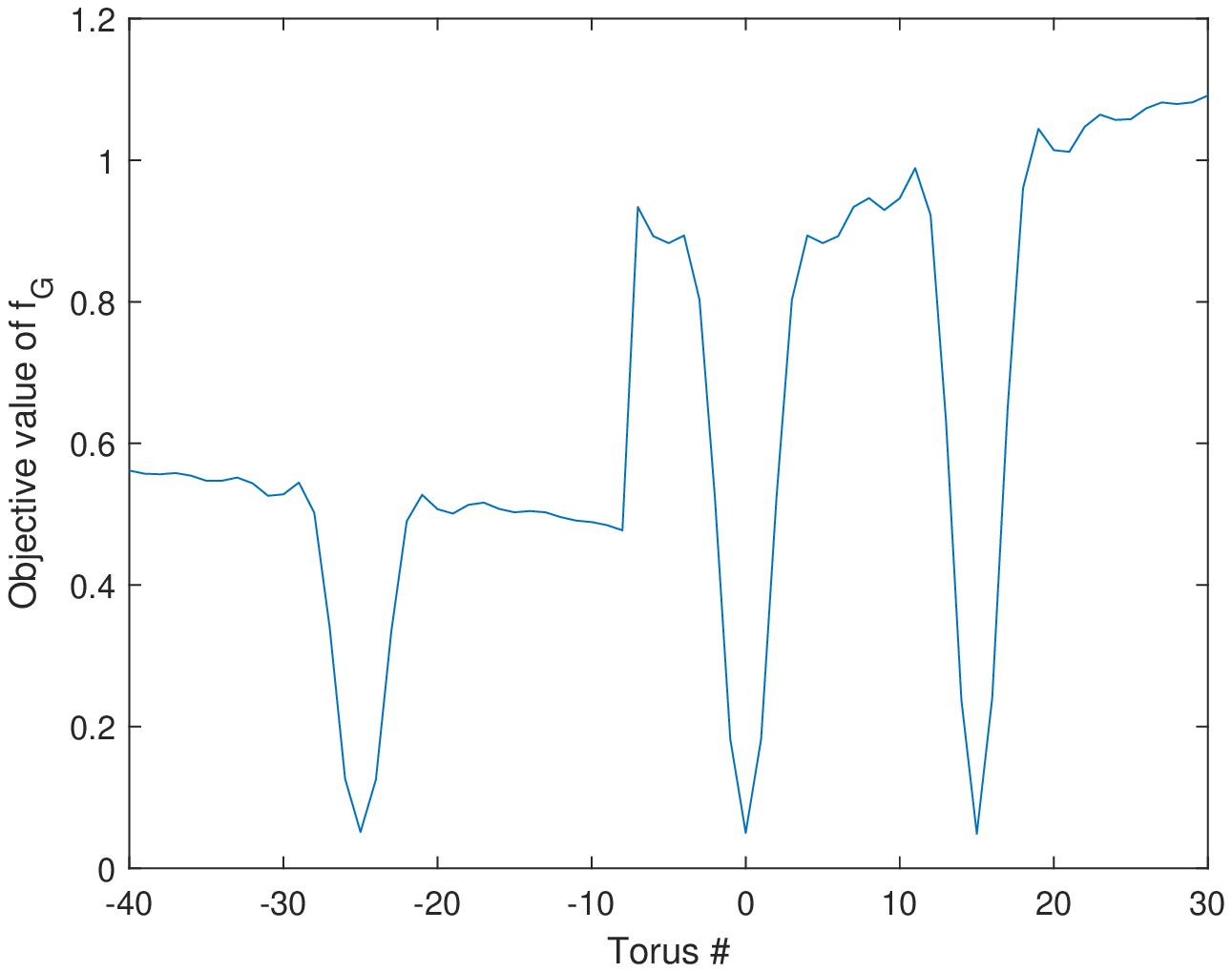} 
        \caption{$f_G$: the ``global" part of the objective function $f$ on Jacob's ladder.}
        \label{fig:fg}
    \end{minipage}\hfill
    \begin{minipage}{0.45\textwidth}
        \centering
        \includegraphics[width=0.9\textwidth]{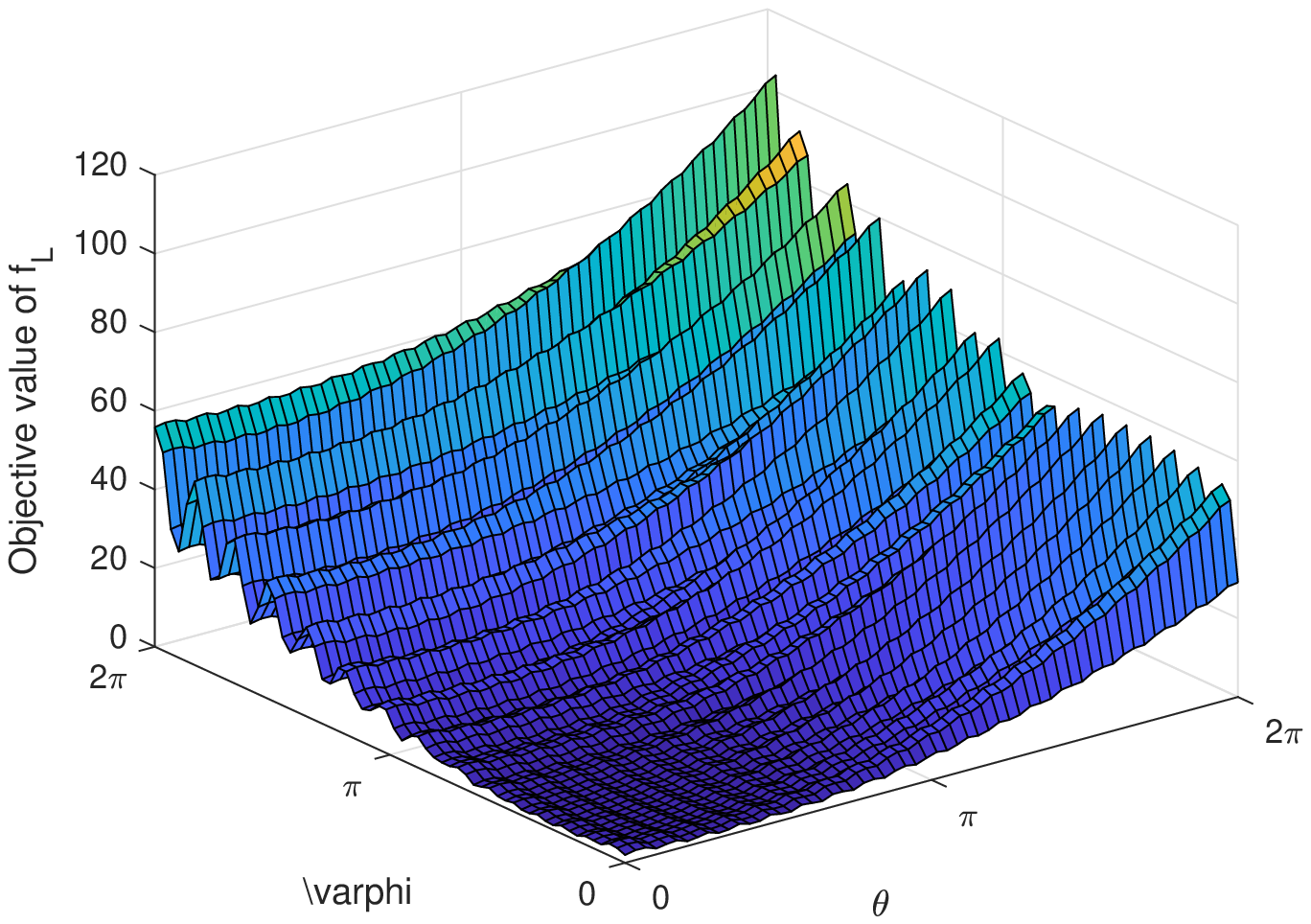} 
        \caption{$f_L$: the ``local" part of the objective function $f$ on Jacob's ladder.}
        \label{fig:fl}
    \end{minipage}
\end{figure}

For the following experiment, the additional parameters of Extended RSDFO are given by:
\begin{align*}
    N_{random} = 6 , \quad N_{cull} = 3 \quad .
\end{align*}

For each execution, Extended RSDFO initializes with a set of $5$ initial centroids $X^0 := \left\{x^0_i \right\}_{i=1}^5$ randomly generated on the Jacob's ladder between torus number $-30$ to $30$. The initial set of centroids $X^0$ is generated using a heuristic approach. RTR and RCMA-ES is initialized with one centroid randomly chosen from $X^0$. RCMA-ES is given a budget of $m_1 = 40$ parents per iteration.  R-PSO initialized with $8$ copies of $X^0$, which allows $40$ ``agents" performing $40$ function evaluations per iteration. The expected fitness of the stochastic algorithms are approximated using $10$ Monte Carlo samples with respect to their corresponding probability distributions . Each algorithm is repeated $100$ times with the parameters described in the beginning of this section. For each execution, each algorithm is budgeted by $20,000$ function evaluations. The algorithms (and the RSDFO steams in Extended RSDFO) terminate when they converge locally or when they are sufficently ($10^{-8})$) close to the optimal value at $1.75$. The gradient and the Hessian of the objective function is not provided, that is we are treating it as a black-box optimization problem. Each estimation of the gradient or the Hessian in RTR counts as a function evaluation respectively. The results of the experiments are summarized in Table \ref{table:JL_exp}. The meanings of the column labels of Table \ref{table:JL_exp} are given in Table \ref{table:JL_meaning}.

\begin{remark}
By the discussion in Section \ref{sec:manopt:assmption}, two out of three components (nostalgia and social component) of the R-PSO step would be impossible without a Riemannian logarithm map between any two points on the manifold. 

For the following experiments, since the true Riemannian logarithm map is difficult to compute, we provide R-PSO with a heuristic version of the Riemannian logarithm map on the Jacob's ladder for the sake of comparison. In particular we allow ``communication" between points beyond each other's normal neighbourhood. 

To be precise, we provide a heuristic Riemannian logarithm map between points that are either within the same torus, or if the points are ``close enough" (difference of torus number is $1$). If two points are within the same tori, we compute the Riemannian logarithm map using the product geometry of $S^1 \times S^1$. For points that are ``close enough", we provide a ``general direction" if they are not in the same torus. That is, we first embed the two points in an ambient Euclidean space, determine the direction of the torus of the target point from the original point, then project the vector $\pm \vec{1}$ onto the tangent space of the major circle of the original torus. Otherwise if the pair of points $x,y$ is too far apart, we set $\log_x(y) = \vec{0}$. This is a sufficiently good approximation to the true Riemannian logarithm map on Jacob's ladder. The experimental results is given in Table \ref{table:JL_exp}.


It is important to note that this provides an unfair advantage to R-PSO over the other algorithms, as the other algorithms do not have this additonal requirement (see Section \ref{sec:manopt:assmption}). Moreover, this generates additional computational time not recorded by function evaluations.
\end{remark}

\begin{table}[hbt!]
\begin{tabular}{|c|l|p{0.5\textwidth}|}
\hline
Column label & Meaning              & Description                                                                               \\ \hline
$\ell+\ell$  & Local-local optima   & Algorithm converges without reaching optimal angles nor the optimal tori                  \\ \hline
$\ell+g$     & Local-global optima  & Algorithm attains the optimal angles in a local torus \textbf{but} does not reach a global optimal torus \\ \hline
$g + \ell$   & global-local optima  & Algorithm reaches a global optimal torus \textbf{but} fails to find the local optimal angles within it.                  \\ \hline
$g+g$        & global-global optima (desired solution) & Algorithm converges to \textbf{both} an optimal torus and the optimal angles within the torus. This is the best possible outcome.      \\ \hline
avg. $f$ eval & Average function evaluation & The average function evaluation needed to either converge or reach the global optima. \\ \hline
\end{tabular}
\caption{Detailed description of column labels in Table \ref{table:JL_exp}}
\label{table:JL_meaning}
\end{table}

\begin{table}[hbt!]
\centering
\begin{tabular}{|l|l|l|l|l|l|}
\hline
Method    & $\ell+\ell$ & $\ell+g$ & $g+ \ell$ & $g+g$ & avg. $f$ eval \\ \hline
Ext-RSDFO &   2     &             27          &          10       & 61  & 8483.8  \\ \hline
RCMA-ES   &      1    &            76           &       0          & 23  & 
          4230.1  \\ \hline
RTR       &          96 &           0            &              4        & 0  &  3086.2 \\ \hline
R-PSO       & 2             &            5           &  56                     & 37  & 20000  \\ \hline
\end{tabular}
\caption{Experimental results on Jacob's ladder}
\label{table:JL_exp}
\end{table}


Experimental results of a sample run is provided in Figure \ref{fig:jl:exp} below. In Figure \ref{fig:jlexp:a}, we observe that only Extended RSDFO attained the global-global solution. Figure \ref{fig:jlexp:b} provides a close-up view of Figure \ref{fig:jlexp:a}, which illustrates the performance of Extended RSDFO and R-PSO towards the global-global optima in the experiment. We observe that Extended RSDFO attained the global-global solution with around $15000$ function evaluations, while R-PSO was only able to find a global-local solution (optimal torus but not the optimal angle) with the given budget of $20000$ function evaluations.


\begin{figure*}[hbt!]
    \begin{subfigure}[t]{0.5\textwidth}
        \centering
        \includegraphics[width=1\textwidth]{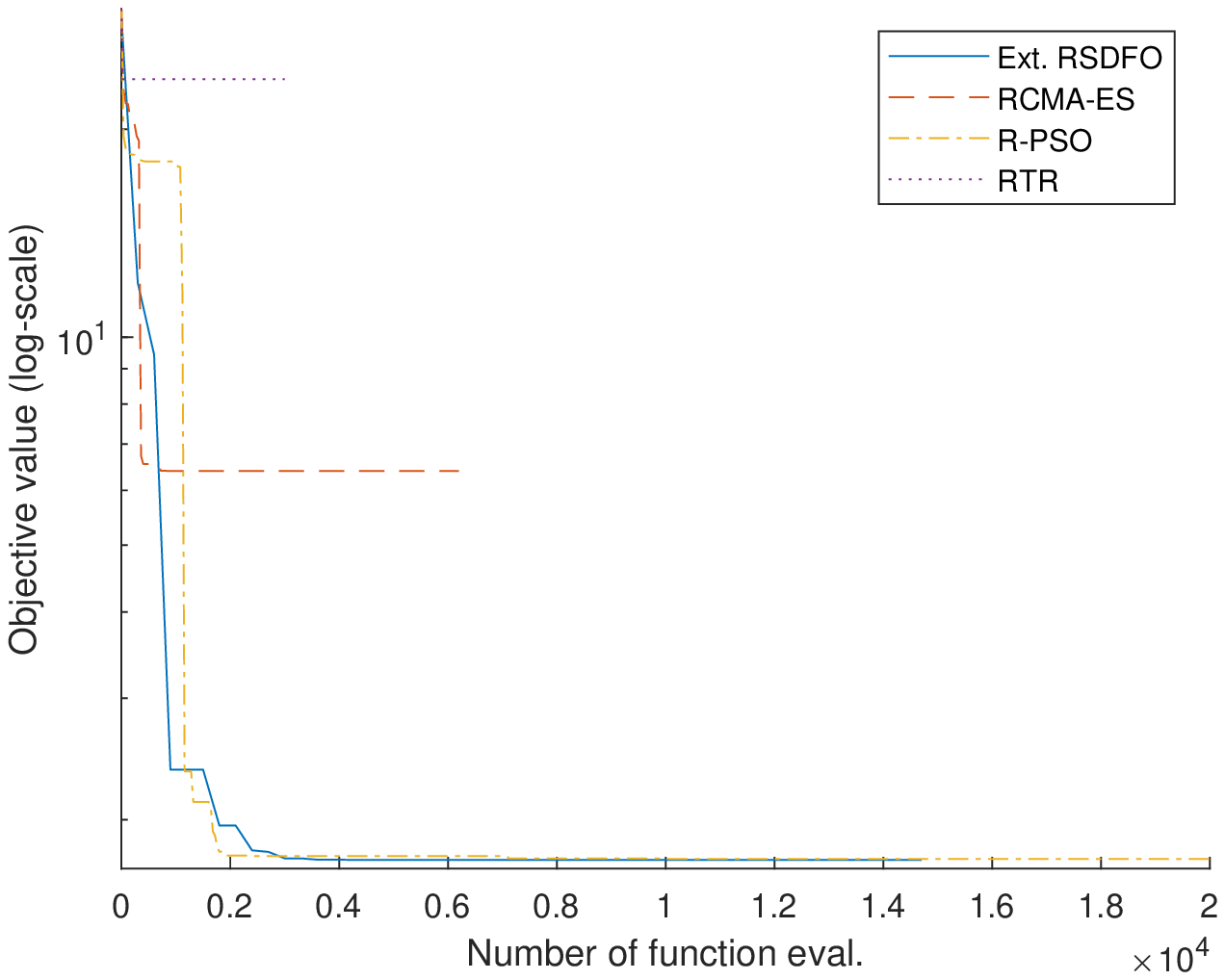}
        \caption{Comparison of quality of solution obtained by each algorithm on Jacob's ladder }   
        \label{fig:jlexp:a}
    \end{subfigure}
        ~ 
    \begin{subfigure}[t]{0.5\textwidth}
    \centering
       \includegraphics[width=1\textwidth]{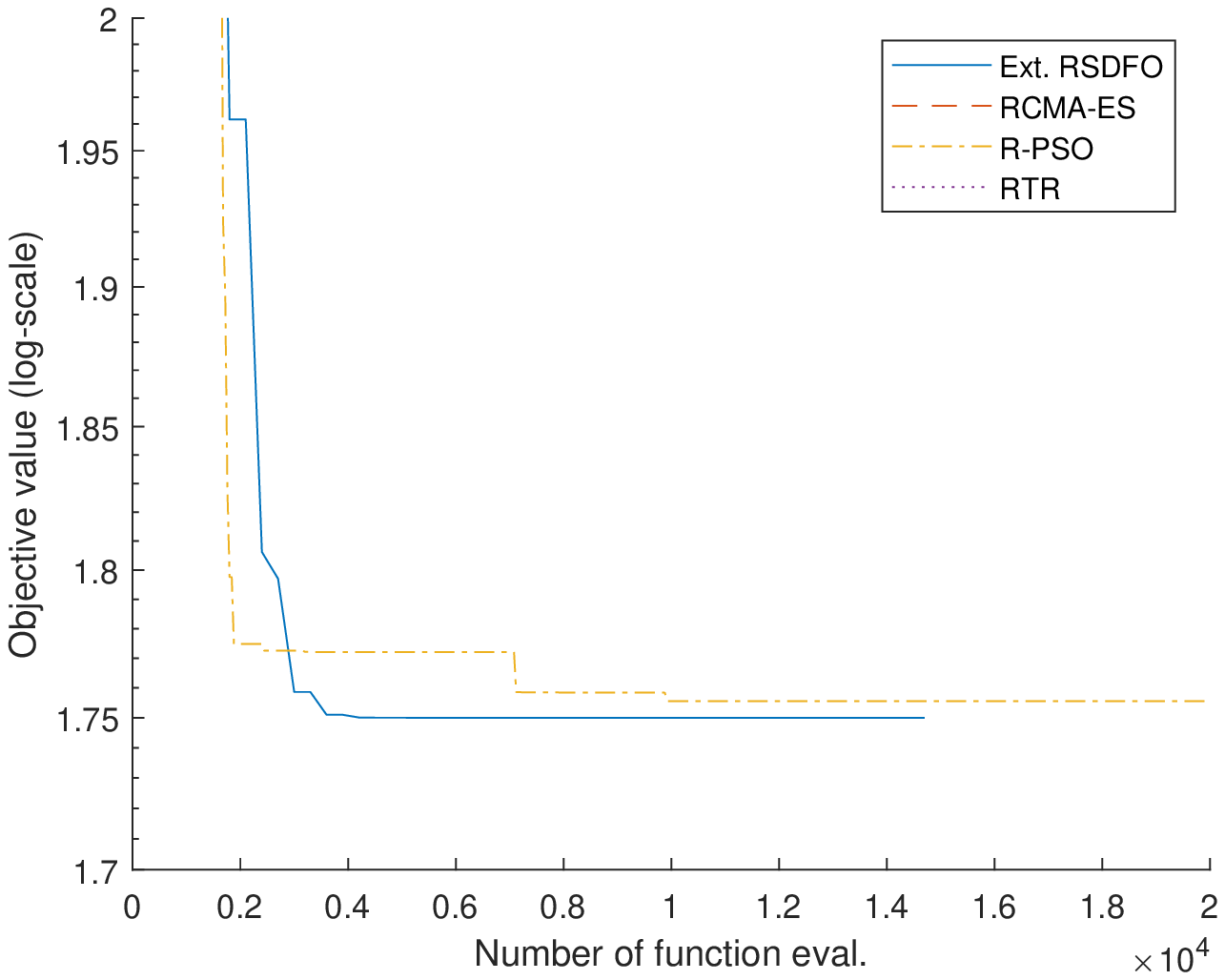}
        \caption{Close-up of Figure \ref{fig:jlexp:a} towards the global-global optima.}\label{fig:jlexp:b}
    \end{subfigure}
       \caption{Experimental results of one of the experiments on Jacob's ladder. }
    \label{fig:jl:exp}
\end{figure*}


\subsection{Discussion}
\label{sec:discuss:exp}

From the experimental results above, we observe that gradient-based line search methods (such as RTR) is less effective compared to model-based approaches (Extended RSDFO, RCMA-ES) and meta-heuristic evolutionary strategies (R-PSO) when solving multi-modal optimization problems on Riemannian manifolds. We discuss the experiments in further detail in this section.

\subsubsection{On Sphere and Grassmannian manifold}
\label{sec:sphere:gr:discuss}


The Riemannian manifolds considered in the first two experiments are $2$-sphere $S^2$ and low dimensional Grassmannian manifolds $Gr_{\mathbb{R}}(p,n)$. These two types of are ``small'', in the sense that they are low dimensional, compact, coverable by a small number of geodesic balls, and most importantly they admit a global parametrization (matrix representation).

The optimization on $S^2$ was taken from \cite{colutto2010cma}, and was used for the sake of comparison. We observe that Extended RSDFO with the basic version of RSDFO core (Algorithm \ref{alg:greda}) performs better than RCMA-ES even in the sphere. On the other hand, R-PSO out-performs the rest of the algorithms, this is partly due to the abundance of computational resources (available computational budget relative to the size of the search manifold)  and the amount of search agents available (relative to the size of the manifold). 

The manifold $S^2$ is small, in the sense that it can be covered by geodesic balls centered around a handful of points. In particular, the implemented exploration distribution (described in the beginning of this section) often fails to find a boundary point before the the point rejection threshold is reached after a few iterations. This resulted in less effective searches (in rejected points) and higher sensitivity to initial condition. 

In light of this, we ran Extended RSDFO with additional flexibility in selecting new exploration centroids. By setting the geodesic balls to be $\epsilon_b = 0.4$ times the injectivity radius, we allow more overlapping between search regions generated by the search centroids, thus enhancing the exploration of search region. Indeed, the ``relaxed version" of Extended RSDFO with $\epsilon_b = 0.4$ is less sensitive to initial conditions and has shown improvements over the version with $\epsilon_b = 1$. The relaxed version is also closely competitive with R-PSO in determing the global optima.




From the experiments on $Gr_{\mathbb{R}}(2,4)$, we observe a similar behavior as the experiment on $S^2$. For model-based algorithms: Extended RSDFO out performs RCMA-ES by searching with multiple centroids at the expense of additional resources. On the other hand both versions of Extended RSDFO are comparable to the meta-heuristic R-PSO.


On $Gr_{\mathbb{R}}(2,4)$ (a $4$- dimensional manifold), the both versions of Extended RSDFO and R-PSO achieve a high success rate in determining the global solution. However, as we increase the dimension to $Gr_{\mathbb{R}}(2,5)$ (a $6$- dimensional manifold), the success rates of the algorithms dropped.

The drop in success rate for model-based algorithms such as Extended RSDFO and RCMA-Es is is partly due to the typical exponential increase in the amount of resources necessary to correctly estimate the underlying (local) stochastic modal. In particular, if we assume the amount of resources increases exponentially: for $Gr_{\mathbb{R}}(2,4)$, the average function evaluation of Extended RSDFO is $24000 \approx \left(12.45 \right)^4$, this means on $Gr_{\mathbb{R}}(2,5)$ the expected function evaluation requirement would be approximately $\left( 12.45\right)^6 \approx 3724055$, which is two orders of magnitude higher than our budget of $40000$.

On the other hand, a manifold of higher dimension also means a larger amount of search centroids is needed to provide better exploration of the search space. This in turn affects the solution quality of both R-PSO and Extended RSDFO as the search agents remain unchanged.




From the two sets of experiments on $Gr_{\mathbb{R}}(2,4)$ and $Gr_{\mathbb{R}}(2,5)$ on, we observe that R-PSO is more scalable to the dimensional increases. On the other hand, model-based algorithms such as Extended RSDFO and RCMA-ES both require additional resources to estimate the correct model as the dimension increase.

\subsubsection{On Jacob's ladder}
\label{sec:jl:discuss}
The optimization problem on Jacob's ladder described above was deliberately set to be difficult: the objective function $f$ is multi-modal, non-convex, non-separable with three global optima. There are two notably challenging features of the objective function $f = f_G \cdot f_L$ from the ``global" and ``local" part respectively.

First of which comes from the the ``global'' part of the objective function: $f_G$, which has a plateau for points with torus number within the range $\left[-19,-8 \right]$. The aforementioned algorithms are prone to be stuck on centroids with tori number $-19\ldots -8$, as $f_G(n_1) < f_G(n_2)$ for $n_1 \in \left[-19,-8 \right]$ and $n_2 \in \left\{-3,-2,2,3,\ldots \right\}$. In particular, the search centroids within the plateau have slightly better function value than the other tori around the global optima, while being far away from global optima.

Secondly, the ``local" part of the objective function $f_L$ also introduces complications due to the local parametrization of the tori. As illustrate in Fig \ref{fig:fl_1}, \ref{fig:fl_2} below, as the search landscape of $f_L$ wraps around the tori, the solution estimate depends on how the search path travels between the tori in Jacob's ladder. The algorithms might wrongly estimate the overall objective function values of the neighboring torus depending on where the search path is.

\begin{figure}[hbt!]
    \centering
    \begin{minipage}{0.45\textwidth}
        \centering
        \includegraphics[width=1\textwidth]{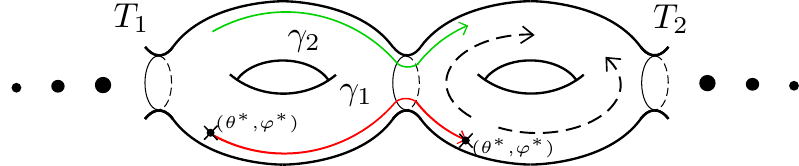} 
        \caption{Illustration of the ``trap'' of $f_L$. The local optimal angles $\left(\theta^*,\varphi^*\right)$ are marked by the points labelled with $\times$. The objective value increases along dotted lines on $T_2$ (as well as $T_1$ and the other tori), as the ``objective landscape", described by Figure \ref{fig:fl}, ``wraps around'' the local coordinates of the tori. If an algorithm searches along $\gamma_2$ from $T_1$ to $T_2$, it may think $T_1$ is better. On the other hand, if an algorithm searches along $\gamma_1$, then $T_2$ may seem like the better torus.}
        \label{fig:fl_1}
    \end{minipage}\hfill
    \begin{minipage}{0.45\textwidth}
        \centering
        \includegraphics[width=1\textwidth]{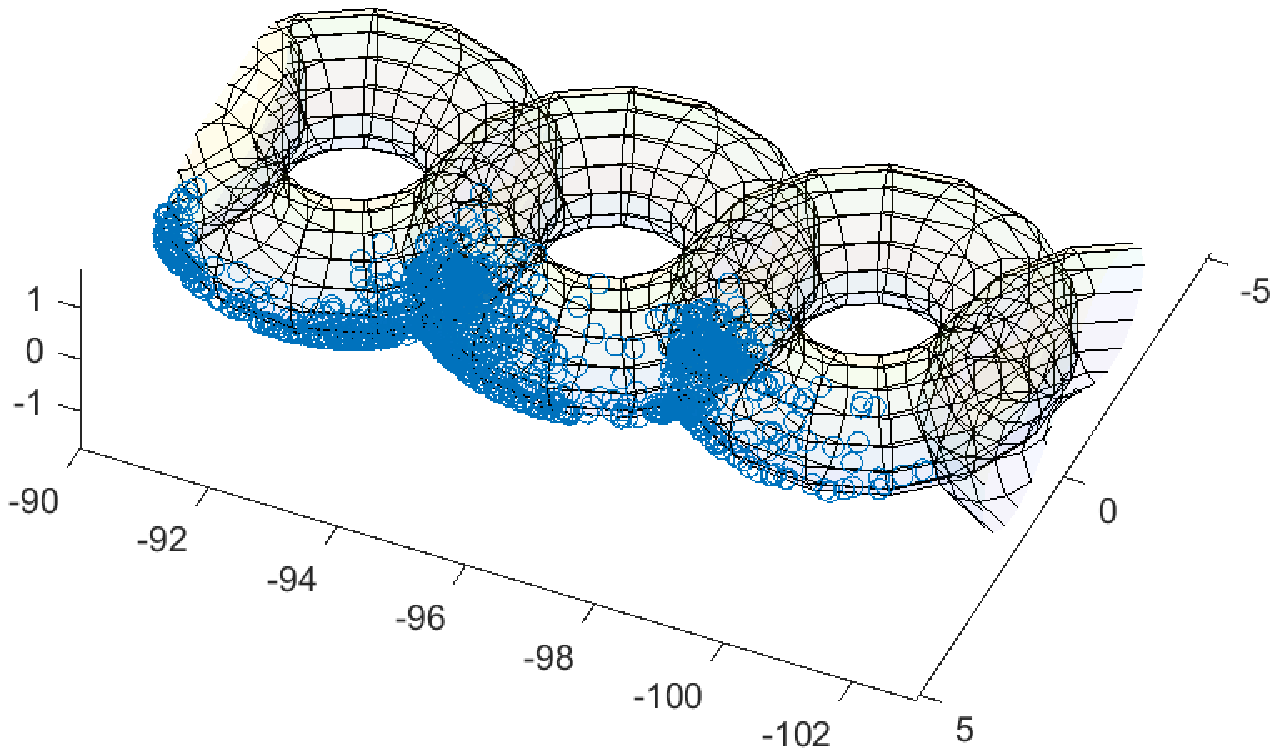} 
        \caption{An illustration of path $\gamma_1$ of Figure \ref{fig:fl_1} from one of the experiments. This figure shows the final converging steps of Extended RSDFO. The search points generated from the converging distributions (from the RSDFO core) are concentrated around $\left(\theta^*,\varphi^*\right) = \left(1,1\right)$ of each tori.}
        \label{fig:fl_2}
    \end{minipage}
\end{figure}


From the experiments, we observe that Extended RSDFO achieves the highest global-global convergence rate amongst the manifold optimization algorithms in the experiments. 

The current implementation of Extended RSDFO, with the basic version of SDFO (Algorithm \ref{alg:eda}) and a simplified version of exploration point selection (discussed in the beginning of this section), is competitive with R-PSO and RCMA-ES on the global and local scale respectively:

On one hand, Extended RSDFO has shown to be competitive with the current state-of-the-art manifold optimization heuristics R-PSO  (even with the ``extra" Riemannian logarithm map) in determining the global optimal tori. Whilst R-PSO is effective in determining the global optimal tori, it often fails to converge to the local optimal angles. At the same time, Extended RSDFO is more accurate when converging to the local optimal angles. 

On the other hand, Extended RSDFO is also shown to be competitive with the current state-of-the-art model-based algorithm RCMA-ES in terms of convergence to the local optimal angles within each tori, while more effective when determing the global optimal tori. RCMA-ES out-performs the other algorithms in determining the local optimal angles, but it is unable to move far away from the initial torus. That is, when given an initial point $\mu_0$ in torus number $i$, RCMA-ES can at-most explore torus numbers within $\left[ i-2, i+2\right]$. 

In this sense, Extended RSDFO gets the ``best of both worlds'' of R-PSO and RCMA-ES in the expense of additional computational resources. Extended RSDFO generally requires more resources (function evaluation, evaluation of boundary points, exponential maps) compared to RCMA-ES when evaluating the results across different centroids.




Finally, RTR is efficient in finding the ``local-local" optima, but will almost always get stuck in them (unless we initialize at the global optimal tori).

\section{Conclusion}
This paper addresses multi-modal, black-box manifold optimization problems on Riemannian manifolds using population-based stochastic optimization methods from a geometrical perspective. We summarize our investigations and the implications as follows:

%
%
%



\begin{enumerate}
\item \textbf{Geometrical framework for stochastic optimization on Riemannian
manifolds:}

The statistical geometry of locally inherited probability densities and mixture densities over Riemannian manifolds (Section \ref{app:localdistn}, Section \ref{Subsection:MixtureDensity} respectively) provided us with a geometrical framework for stochastic optimization on Riemannian manifolds. This framework relates the statistical parameters of the statistical manifold decision space and local point estimations on the Riemannian manifold. Mixture densities, in particular, overcomes the local restrictions of parametric probability distributions on Riemannian manifolds and Riemannian adapted manifold optimization methods.

The significance of parametrized mixture densities extends to statistical analysis on manifolds. This is because closed form expression of parametrized densities beyond normal neighbourhhods whose statistical parameters are consistent with the local statistical point estimiations are generally difficult to obtain.

\item \textbf{Novel population-based stochastic optimization algorithm based on the geometrical framework:} Using the product statistical geometry of mixture densities, we proposed Extended RSDFO (Algorithm \ref{alg:ereda}), a population-based stochastic optimization algorithm on Riemannian manifolds, which overcomes the local restrictions and implicit assumptions (Section \ref{sec:manopt:assmption}) of manifold optimization in the literature. 

Under the aforementioned geometrical framework, the components and properties of Extended RSDFO admit geometrical interpretations. This enabled theoretical analysis of the geometry of the evolutionary step (Proposition \ref{rmk:monoincreasing}, Section \ref{sec:theo}) and the convergence behavior (Theorem \ref{thm:conv}).

\item \textbf{Experiments on Jacob's ladder:} We compared Extended RSDFO against state-of-the-art manifold optimization methods in the literature, such as Riemannian Trust-Region method \cite{absil2007trust,absil2009optimization}, Riemannian CMA-ES \cite{colutto2010cma} and Riemannian Particle Swarm Optimization  \cite{borckmans2010modified,borckmans2010oriented}, using optimization problems defined on the $n$-sphere, Grassmannian manifolds, and Jacob's ladder. 

Jacob's ladder, in particular, is a novel synthetic example designed to motivate and necessitate manifold optimization (as oppose to classical constraint optimization) \footnote{Jacob's ladder is a non-compact manifold of countably infinite genus, which cannot be expressed as polynomial constraints and does not have a global representation in an ambient Euclidean space. Optimization problems on Jacob's ladder therefore cannot be addressed by traditional (constraint) optimization techniques on Euclidean spaces, which necessitates the development of manifold optimization algorithms.}. 

From the experimental results, in particular the ones on Jacob's ladder, we observed that population-based approaches out-perform gradient-based methods on mutli-modal manifold optimization problems. Even though Extended RSDFO is implemented with a basic RSDFO as the ``core" algorithm in the experiments, we demonstrated that Extended RSDFO is comparable to R-PSO (meta-heuristic) in finding global optimal, whilst also comparable to the accuracy of RCMA-ES (modal-based algorithm)  in terms of attaining the local optima. In this sense, Extended RSDFO has the ``best of both worlds" scenario at the expense of possibly additional computational resources. 
 
\end{enumerate}

\newpage 

\bibliographystyle{elsarticle-harv}

\bibliography{main.bbl}
\newpage
\appendix

\section{Riemannian Adaptation of Optimization Algorithms in the literature}
\label{ch:litreview:manopt}
In this section we illustrate the Riemannian adaptation approach with algorithms from the literature. More precisely, we describe Riemannian adaptations of trust-region method \cite{absil2007trust}, Particle Swarm Optimization (PSO) \cite{borckmans2010modified} and Covariance Matrix Adaptation Evolution Strategy (CMA-ES) \cite{colutto2010cma}.

For the remainder of the section, let $M$ denote the Riemannian manifold search space, and let $f : M \rightarrow \mathbb{R}$ denote the objective function of a \textbf{minimization} problem over $M$. For each $x \in M$, let $\tilde{f}_x := f\circ \exp_x: T_x M \rightarrow \mathbb{R}$ denote the locally converted objective function on each tangent space $T_x M$ of $M$. For each $x \in M$, let  $U_x := \exp_x^{-1} N_x \subset T_x M$, centered at $\vec{0} \in T_x M$, denote the pre-image of a normal neighbourhood $N_x$ of $x$ under the Riemannian exponential map.

We begin by reviewing the notion of Riemannian gradient and Riemannian Hessian \cite{lee2001introduction}.

\subsection{Riemannian Gradient and Hessian}
\label{sec:riemanniangradhess}
Given a Riemannian manifold $(M,g)$ and affine connection $\nabla$ on $M$. Let $f : M \rightarrow \mathbb{R}$ denote a smooth real-valued function on $M$. For any point $x\in M$ the \textbf{Riemannian gradient} of $f$ at $x$, denoted by $\operatorname{grad}f(x) \in T_x M$, is the tangent vector at $x$ \textit{uniquely} associated to the differential of $f$. In particular, $\operatorname{grad}f(x)$ is the unique tangent vector satisfying:
\begin{align*}
\langle \operatorname{grad}f(x), v \rangle_g = \left. df(v)\right|_x \quad, \quad  v \in T_x M \quad .
\end{align*}
Thus $\operatorname{grad}f \in \mathcal{E}(TM)$ is a vector field on $M$. In local coordinates  $\left( x^1, \ldots, x^n \right)$ of $M$, Riemannian gradient $\operatorname{grad}f$ can be expressed as:
\begin{align*}
\operatorname{grad}f = \sum_{i,j = 1}^n g^{ij} \frac{\partial f}{\partial x^i} \frac{\partial}{\partial x^j} \quad ,
\end{align*}
where $\left[ g^{ij} \right]$ denote the inverse of Riemannian metric matrix associated to $g$ and $\left(\frac{\partial}{\partial x^j}\right)_{j=1}^n$ denote the local coordinate frame. Fixing the local coordinate frame  $\left(\frac{\partial}{\partial x^i}\right)_{i=1}^n$  corresponding to the local coordinate system, we may associate $\operatorname{grad}f$ to a column vector for each $x\in M$:
\begin{align*}
\tilde{\nabla} f(x) := \sum_{i=1}^n g^{ij} \frac{\partial f}{\partial x^i} =  G^{-1} (x) \nabla f (x) \quad ,
\end{align*}
where $G$ denote the Riemannian metric matrix associated to $g$ and $\nabla f (x)$ denote the gradient in Euclidean spaces from elementary calculus. $\tilde{\nabla} f(x)$ is otherwise known as \textbf{natural gradient} \cite{amari1998natural} in the machine learning community.

On the other hand, the \textbf{Riemannian Hessian} of a smooth function $f$ is the $(0,2)$ tensor field given by:
\begin{align*}
\operatorname{Hess}f: TM \times TM &\rightarrow \mathbb{R} \\
\left( X, Y\right) &\mapsto \nabla \nabla f =  \langle \nabla_X \operatorname{grad} f, Y \rangle_g \quad .
\end{align*}

\subsection{Riemannian Gradient-based Optimization}
\label{sec:rtr}
In this section we consider Riemannian adaptation \cite{absil2007trust} of Trust Region method \cite{nocedal2006numerical} as an example of how gradient-based optimization methods from Euclidean spaces can be translated to the context of Riemannian manifolds. Other gradient-based algorithms in the literature are translated under the same principle described in Section \ref{sec:principle:manopt} \cite{absil2009optimization}. The following side-by-side comparison of Euclidean Trust-Region \cite{nocedal2006numerical} and Riemannian Trust Region \cite{absil2007trust} illustrates the Riemannian adaption process of gradient-based methods.

The parameters and the numerical values of both versions of Trust-Region, summarized in Table \ref{Table:review:RTR} below, are taken from the corresponding references \cite{nocedal2006numerical,absil2007trust}. 
It is worth noting that while the numerical values may differ across different implementations, the purpose of this comparison is to illustrate the Riemannian adaptation process highlighted in the algorithms below.

\begin{table}[hbt!]
\begin{tabular}{|l|l|l|}
\hline
Parameter           & Value                          & Meaning                                                        \\ \hline
$\overline{\Delta}$ & $\pi$                          & Upperbound on step length                                      \\ \hline
$\Delta_0$          & $\Delta_0 \in (0,\Delta)$ & Initial step length                                            \\ \hline
$\rho'$             & $\rho' \in \left[ 0,\frac{1}{4} \right)$                & Parameter that determines whether the new solution is accepted \\ \hline
\end{tabular}
\caption{Parameters of Trust-region.}
\label{Table:review:RTR}
\end{table}

\begin{minipage}{0.49\textwidth}
\begin{algorithm}[H]
    \caption{Euclidean Trust-Region }\label{alg:etr}
 \KwData{Initial point $x_0 \in \mathbb{R}^n$, $\Delta_0 \in (0,\Delta)$.}
 \For{$k = 0,1,2,\ldots$}{
	\HiLi{ Obtain $\eta_k$ by solving \textbf{Euclidean} Trust-Region sub-problem. \label{alg:tr:subproblem}} \;
	\HiLi{ Compute reduction ratio $p_k$ \label{alg:tr:ratio} }\;
	\uIf{$p_k < \frac{1}{4}$}{
	$\Delta_{k+1} = \frac{1}{4} \Delta_k$ \;
	\uElseIf{$p_k > \frac{3}{4}$ \textbf{and} $\vert \vert \eta_k \vert \vert = \Delta_k$}{
	$\Delta_{k+1} = \min(2\cdot \Delta_k, \Delta)$\;
	}
	\uElse{}{
	$\Delta_{k+1} = \Delta_k$ \;
	}
	}
	\uIf{$p_k > \rho'$}{
	\HiLi{ $x_{k+1} = x_k + \eta_k$} \label{alg:tr:progress }\;
	\uElse{}{
	$x_{k+1} = x_k$ \;
	}
	}
 } 
\end{algorithm}
\end{minipage}
\hfill
\begin{minipage}{0.49\textwidth}
\begin{algorithm}[H]
    \caption{Riemannian Trust-Region}\label{alg:RTR}
  \KwData{Initial point $x_0 \in M$, $\Delta_0 \in (0,\Delta)$.}
 \For{$k = 0,1,2,\ldots$}{
	\HiLi {Obtain $\eta_k$ by solving \textbf{Riemannian} Trust-Region sub-problem. } \;
	\HiLi{ Compute reduction ratio $p_k$} \;
	\uIf{$p_k < \frac{1}{4}$}{
	$\Delta_{k+1} = \frac{1}{4} \Delta_k$ \;
	\uElseIf{$p_k > \frac{3}{4}$ \textbf{and} $\vert \vert \eta_k \vert \vert = \Delta_k$}{
	$\Delta_{k+1} = \min(2\cdot \Delta_k, \Delta)$\;
	}
	\uElse{}{
	$\Delta_{k+1} = \Delta_k$ \;
	}
	}
	\uIf{$p_k > \rho'$}{
	\HiLi{ $x_{k+1} = \exp_{x_k}\left(\eta_k \right)$ }\;
	\uElse{}{
	$x_{k+1} = x_k$ \;
	}
	}
 } 
\end{algorithm}
\end{minipage}

The two algorithms differ in the Trust-Region sub-problem in line \ref{alg:tr:subproblem}, computation of reduction ratio in line \ref{alg:tr:ratio} and the translation to the new search iterate in line \ref{alg:tr:progress } (highlighted lines). Since we have already discussed the translation of search iterate in Section \ref{sec:principle:manopt}, it remains to discuss the translation of trust-region subproblem and reduction ratio from Euclidean spaces to Riemannian manifolds.

The Euclidean trust-region subproblem at the current search iterate $x_k$ on $\mathbb{R}^n$ is given by:
\begin{align*}
\eta_k := \argmin_{\eta \in \mathbb{R}^n} m_k(\eta) =  \argmin_{\eta \in \mathbb{R}^n} f(x_k) + \langle \nabla f(x_k), \eta \rangle + \frac{1}{2} \eta^\top H_f(x_k) \eta \, , \quad \text{s.t. } \vert\vert \eta_k \vert\vert \leq \Delta_k
\end{align*}
where $\nabla f(x_k)$ is the gradient on Euclidean space from elementary calculus,  $\langle\cdot,\cdot\rangle$ denote the standard inner product on $\mathbb{R}^n$ and $H_f(x_k)$ is the Hessian matrix of $f$ at $x_k$.

On Riemannian manifold $M$, the Riemannain trust-region subproblem is therefore translated \textit{locally} around the search iterate $x_k \in M$ from Euclidean version using tools described in Section \ref{sec:principle:manopt}:
\begin{align*}
\eta_k := \argmin_{\eta \in T_{x_k} M} m_{x_k}(\eta) =  \argmin_{\eta \in T_{x_k} M} f(x_k) + \langle \operatorname{grad} f(x_k), \eta \rangle_g + \frac{1}{2} \operatorname{Hess}f(x_k)\left(\eta,\eta\right) \, , \quad \text{s.t. } \sqrt{\langle \eta,\eta \rangle}_g \leq \Delta_k \quad ,
\end{align*}
where $\operatorname{grad} f(x_k)$ and $\operatorname{Hess}f(x_k)\left(\eta,\eta\right)$ denote the Riemannian gradient and Hessian function described in Section \ref{sec:riemanniangradhess}, and $\langle\cdot,\cdot\rangle_g$ is the Riemannian metric. Note that the above formulation is strictly local within the tangent space $T_{x_k} M$ and can be equivalently formulated as minimization problem on the tangent space with traditional gradient and Hessian of $\tilde{f}$:
\begin{align*}
\eta_k := \argmin_{\eta \in T_{x_k} M} m_{x_k}(\eta) =  \argmin_{\eta \in T_{x_k} M} \tilde{f}(x_k) + \langle \nabla \tilde{f}(x_k), \eta \rangle_g + \frac{1}{2} \eta^\top H_{\tilde{f}}(x_k) \eta \, , \quad \text{s.t. } \sqrt{\langle \eta,\eta \rangle}_g \leq \Delta_k \quad .
\end{align*}

Finally, the translation of reduction ratio $p_k$ is described as follows:
\begin{align*}
p_k = \underbrace{ \frac{f(x_k) - f(x_k + \eta_k)}{m_k(0) - m_k(\eta_k)} }_{\text{Euclidean case}} \longrightarrow \underbrace{ \frac{f(x_k) - f\left(\exp_{x_k}(\eta_k) \right)}{m_{x_k}(0) - m_{x_k}(\eta_k)} }_{\text{Riemannian case}} \quad .
\end{align*}

\begin{remark}[Retraction]
Since exponential map is difficult to compute, authors of \cite{absil2009optimization} proposed a slightly relaxed map from the tangent spaces to $M$ called the retraction:
\begin{definition}
A \textbf{retraction} on a manifold is a (smooth) function $R_x: T_x M \rightarrow M$ satisfying
\begin{enumerate}
\item $R_x(\vec{0}) = x$, where $\vec{0} \in T_x M$ and
\item $DR_x(\vec{0}) = \text{id}_{T_x M}$ .
\end{enumerate}
\end{definition}
This establishes the sufficient conditions to locally preserve the gradient at $x\in M$. This generalization is similar to the use of orientation-preserving diffeomorphism in Section \ref{SectionProbOnMfold}
\end{remark}

\subsection{Riemannian Particle Swarm Optimizaiton}
\label{sec:rpso}
The adaptation of Particle Swarm Optimization (PSO) \cite{eberhart1995particle}, a population-based meta-heuristic, from Euclidean spaces to Riemannian manifolds \cite{borckmans2010modified} is similar to the translation of trust-region method described in the previous section. In particular, the Riemannian version and the Euclidean version differ only in the evolution of search iterates while the structure of the algorithm remains the same. Furthermore, the computation remains largely the same: Riemannian adapted version perform local computations on \textit{pre-images of normal neighbourhoods} within the tangent space, just as in the Euclidean case.

In this section, we consider Riemannian PSO discussed in \cite{borckmans2010modified}. In classical PSO on Euclidean space $\mathbb{R}^n$ discussed in \cite{eberhart1995particle}, search agents (particles) are randomly generated on $\mathbb{R}^n$, and on the $k^{th}$ iteration each particle $x_i^k$ ``evolves" under the following equation (Equation (5) and (6) of \cite{borckmans2010modified}):
\begin{align}
    v_i^{k+1} &= \underbrace{w^k\cdot v_i^k}_{\text{inertia}} + \underbrace{c\cdot \alpha_i^k \left( y_i^k - x_i^k\right)}_{\text{nostalgia}} + \underbrace{s\cdot \beta_i^k \left(\hat{y}^k - x_i^k \right)}_{\text{social}} \label{eqn:pso:1}\\
    x_i^{k+1} &= x_i^k + v_i^{k+1} \label{eqn:pso:2} \quad .
\end{align}
The meaning of the symbols are summarized in Table \ref{Table:pso} below.
\begin{table}[hbt!]
\centering
\begin{tabular}{|l|p{0.8\textwidth}|}
\hline
Symbols                & Meaning                                                                                \\ \hline
$x_i^k$                & position of the $i^{th}$ particle at the $k^{th}$ iteration                            \\ \hline
$v_i^{k}$              & search direction (overall velocity) of the $i^{th}$ particle at the $k^{th}$ iteration \\ \hline
$w^k$                  & inertial coefficient, assigned by a predefined real valued function on the iteration counter $k$                                                                   \\ \hline
$c$                    & weight of the nostalgia component, a predefined real number                                                                \\ \hline
$s$                    & weight of the social component, a predefined real number                                                                     \\ \hline
$\alpha_i^k,\beta_i^k$ & random numbers generate with $\left[0,1 \right]$ for each of the $i^{th}$ particle at the $k^{th}$ iteration of the nostalgia and social components respectively. The purpose of the random components is to avoid premature convergence \cite{eberhart1995particle}                     \\ \hline
$y_i^k$                & ``personal best" of the $i^{th}$ particle up to the $k^{th}$ iteration                 \\ \hline
$\hat{y}^k$            & overall best of all the particles up to the $k^{th}$ iteration                         \\ \hline
\end{tabular}
\caption{Meaning of symbols in PSO step (Equation \eqref{eqn:pso:1}).}
\label{Table:pso}
\end{table}

The translation of the second PSO equation (Equation \eqref{eqn:pso:2}) to Riemannian manifolds is straightforward: using the Riemannian exponential map, we simply rewrite:
\begin{align*}
    x_i^{k+1} = x_i^k + v_i^{k+1} \rightarrow x_i^{k+1} = \exp_{x_i^k} \left(v_i^{k+1}\right) \quad ,
\end{align*}
where the completeness assumption of $M$ is used when $\vert\vert v_i^{k+1}\vert \vert > \operatorname{inj} \left(x_i^k\right)$.

The adaptation of the first PSO equation (Equation \eqref{eqn:pso:1}) involves the difference of points on the manifold in the nostalgia and social component. In \cite{borckmans2010modified}, the nostalgia component (similarly the social component) of the first PSO equation is adapted using the Riemannian logarithm map:
\begin{align*}
    c\cdot \alpha_i^k \left( y_i^k - x_i^k\right) \longrightarrow c\cdot \alpha_i^k \cdot \log_{x_i^k}\left(y_i^k \right) \quad .
\end{align*}
This construction is not a problem when $y := y_i^k$ and $x := x_i^k$ are ``close enough''. Indeed, when $y$ is within the normal neighourhood of $x$ (and vice versa), $\log_x\left(y\right)$ is the tangent vector at $x$ that points towards $y$ such that if we follow the geodesic starting at $x$ with velocity $\log_x\left(y\right)$ for time $1$ we obtain $y$. Therefore, the full evolutionary step of Riemannian adaptation of PSO \textit{only} works when \textit{all} the points (particles) are within a \textit{single} normal neighbourhood.

\subsection{Riemannian CMA-ES}
\label{sec:rcma}
In this section, we review the Riemannian adaptation of Covariance Matrix Adaptation Evolutionary Strategies (CMA-ES) \cite{hansen1996adapting,kern2004learning, hansen2006cma}, a state-of-the-art SDFO algorithm that has been adapted to the context of spherical manifolds \cite{colutto2010cma} using the same principle described in Section \ref{sec:principle:manopt}. Further studies of properties of step length of Riemannian CMA-ES were made in \cite{arnold2014use}.

We begin by summarizing the essence and structure of CMA-ES in Algorithm \ref{alg:cma-es} below. The steps involved in the Riemannian adaption process is once again highlighted in the algorithm. It is worth noting that while a variety of different implementations of CMA-ES exists in the literature, the purpose of this section is to illustrate the Riemannian adaptation process highlighted in the algorithms. The parameters of CMA-ES, described in \cite{hansen2016cma,colutto2010cma}, is summarized in Table \ref{Table:cmaparamer} below:
\begin{table}[hbt!]
\centering
\begin{tabular}{|l|l|p{0.3\textwidth}|}
\hline
Parameter          & Default value \cite{hansen2016cma}                                                                                                                                                                            & Meaning                                          \\ \hline
$m_1$               & Depends on the search space                                                                                                                                                                            & Number of sample parents per iteration                  \\ \hline
$m_2$               & $\frac{m_1}{4}$                                                                                                                                                                  & Number of offsprings.                            \\ \hline
$w_i$               & $\log\left( \frac{m_2 +1}{i}\right) \cdot \left(\sum_{j=1}^{m_2} \log\left( \frac{m_2 +1}{j}\right) \right)^{-1}$                                                                & Recombination coefficient                        \\ \hline
$m_{\text{eff}}$    & $ \left(\sum_{i=1}^{m_2} w_i^2 \right)^{-1}$                                                                                                                                     & Effective samples                                \\ \hline
$c_c$               & $\frac{4}{N+4}$                                                                                                                                                                  & Learning rate of anisotropic evolution path      \\ \hline
$c_\sigma$          & $\frac{m_{\text{eff}} +2}{N + m_{\text{eff}}+ 3}$                                                                                                                                & Learning rate of isotropic evolution path        \\ \hline
$\mu_{\text{cov}} $ & $m_{\text{eff}}$                                                                                                                                                                 & Factor of rank-$m_2$-update of Covariance matrix \\ \hline
$c_{\text{cov}}$    & $ \frac{2}{ \mu_{\text{cov}}\left(N+ \sqrt{2}\right)^2} + \left(1- \frac{1}{\mu_{\text{cov}}}\right)\min\left(1, \frac{2\mu_{\text{cov}}-1}{(N+2)^2 + \mu_{\text{cov}}} \right)$ & Learning rate of covariance matrix update        \\ \hline
$d_\sigma $        & $ 1+ 2\max\left(0,\sqrt{\frac{m_{\text{eff}}-1}{N+1}} \right) + c_\sigma$                                                                                                        & Damping parameter                                \\ \hline
\end{tabular}
\caption{Parameters of CMA-ES described in \cite{colutto2010cma}.}
\label{Table:cmaparamer}
\end{table}
\begin{algorithm}[hbt!]
 \KwData{Initial mean $\mu_0 \in \mathbb{R}^n$, set initial covariance matrix $C$ the identity matrix $I$, step size $\sigma = 1$, $p_c = 0$, $p_\sigma = 0$.}
 \While{stopping criterion not satisfied}{
	\HiLi{Sample $\left\{v_i \right\}_{i=1}^{m_1} \sim N\left(\vec{0},(\sigma)^2\cdot C \right) $} \;
	\HiLi{Update mean (search center) $\mu_{k+1} = \mu_k + \sum_{i=1}^{m_2} w_i v_i$} \;
	Update $p_\sigma$ with $C$, $\sigma$, $m_{\text{eff}}$, $c_\sigma$ (isotropic evolution path):
	\begin{align*}
	p_\sigma = \left(1 - c_\sigma \right) p_\sigma + \sqrt{\left( 1- \left( 1-c_\sigma\right)^2\right)} \frac{\sqrt{m_{\text{eff}}}}{\sigma} \cdot C^{\frac{-1}{2}} \cdot  \sum_{i=1}^{m_2} w_i v_i
	\end{align*}	 \;
	Update $p_c$ with  $\sigma$, $m_{\text{eff}}$, $c_c$ (anisotropic evolution path): 
	\begin{align*}
	p_c = \left(1 - c_c \right) p_c + \sqrt{\left( 1- \left( 1-c_\sigma\right)^2\right)} \frac{\sqrt{m_{\text{eff}}}}{\sigma} \cdot  \sum_{i=1}^{m_2} w_i v_i
	\end{align*}
	\;
	Update covariance matrix $C$ with $\mu_{\text{cov}}$, $c_{\text{cov}}$, $p_\sigma$, $p_c$:
	\begin{align*}
C =	c_{\text{cov}} \cdot \left( 1- \frac{1}{\mu_{\text{cov}}} \right) \frac{1}{\sigma^2} \sum_{i=1}^{m_2} w_i v_i v_i^\top + \left(1 - c_{\text{cov}} \right)\cdot C + \frac{c_{\text{cov}}}{\mu_{\text{cov}}} p_c p_c^\top
	\end{align*}		
	\;
	Update step size $\sigma$ with $c_\sigma$, $d_\sigma$, $p_\sigma$:
	\begin{align*}
	\sigma = \sigma \cdot \exp{\left(\frac{c_\sigma}{d_\sigma}\left(\frac{\vert p_\sigma \vert}{E\vert \operatorname{N}\left(0,I\right)\vert} - 1 \right)\right)}
	\end{align*}		
	 \;
	k = k+1 \;
 } 
 \caption{CMA-ES on Euclidean space }
 \label{alg:cma-es}
\end{algorithm}
\\
%


\clearpage 

Riemannian CMA-ES \cite{colutto2010cma} is a special case of RSDFO (Algorithm \ref{alg:greda}). The adaptation of CMA-ES from Euclidean spaces to Riemannian manifolds \cite{colutto2010cma} employs the Riemannian adaptation approach described in Section \ref{sec:principle:manopt}: all the local computations and update of search information are done within (a subset of) the tangent  space of the search iterate just as in the Euclidean case, while the structure of the algorithm remains unchanged. Therefore, it remains to describe the changes in sampling process and translation to new search center highlighted in Algorithm \ref{alg:cma-es} above.

For $k>0$, at the $k^{th}$ iteration of RCMA-ES, a new set of sample vectors is generated from the tangent space $T_{\mu_k} M$ \textit{within the pre-image of normal neighbourhood $N_{\mu_k}$} around the search iterate $\mu_k \in M$:
\begin{align*}
\left\{v_{i,\mu_k} \right\}_{i=1}^{m_1} \subset T_{\mu_k} M \sim N\left(\vec{0},(\sigma)^2\cdot C \right) \quad , \quad \text{s.t. } \vert\vert v_{i,\mu_k} \vert\vert \leq \operatorname{inj}(\mu_k) \quad . 
\end{align*}

The obtained vectors are sorted according to their corresponding function values with respect to the locally converted objective function $\tilde{f}_{\mu_k} = f \circ \exp_{\mu_k}$. The new search iterate (center) is thus obtained via the exponential map similar to Riemannian trust-region and Riemannian PSO:
\begin{align*}
\mu_{k+1} = \exp_{\mu_k} \left(\sum_{i=1}^{m_2} w_i v_i \right) \in N_{\mu_k}\quad ,
\end{align*}
where $N_{\mu_k}$ is the normal neighbourhood of $\mu_k \in M$.

However, there is one caveat: Since tangent spaces are disjoint spaces, the search information has to be carried to the new search iterate in a way that coincides with the manifold structure of the search space. Multi-dimensional search information such as covariance matrix $C$ and the evolutionary paths $p_\sigma$, $p_c$ thus have to be parallel transported to the new search iterate. Since all operations are done within a \textit{single} normal neighbourhood, the parallel transport is described by Equation \eqref{eqn:parallelinnormalnhbd}.

Riemannian adaptation of CMA-ES is thus summarized by Algorithm \ref{alg:rcma-es} below, where the adaptation process described above are highlighted.

\begin{algorithm}[hbt!]
 \KwData{Initial mean $\mu_0 \in \mathbb{R}^n$, set initial covariance matrix $C$ the identity matrix $I$, step size $\sigma = 1$, $p_c = 0$, $p_\sigma = 0$.}
 \While{stopping criterion not satisfied}{
	\HiLi{Sample $\left\{v_{i,\mu_k} \right\}_{i=1}^{m_1} \subset T_{\mu_k} M \sim N\left(\vec{0},(\sigma)^2\cdot C \right) $} \;
	\HiLi{Update mean (search center) $\mu_{k+1} = \exp_{\mu_k} \left(\sum_{i=1}^{m_2} w_i v_i \right)$} \;
	Update $p_\sigma$ with $C$, $\sigma$, $m_{\text{eff}}$, $c_\sigma$ (isotropic evolution path):
	\begin{align*}
	p_\sigma = \left(1 - c_\sigma \right) p_\sigma + \sqrt{\left( 1- \left( 1-c_\sigma\right)^2\right)} \frac{\sqrt{m_{\text{eff}}}}{\sigma} \cdot C^{\frac{-1}{2}} \cdot  \sum_{i=1}^{m_2} w_i v_i
	\end{align*}	 \;
	Update $p_c$ with  $\sigma$, $m_{\text{eff}}$, $c_c$ (anisotropic evolution path): 
	\begin{align*}
	p_c = \left(1 - c_c \right) p_c + \sqrt{\left( 1- \left( 1-c_\sigma\right)^2\right)} \frac{\sqrt{m_{\text{eff}}}}{\sigma} \cdot  \sum_{i=1}^{m_2} w_i v_i
	\end{align*}
	\;
	Update covariance matrix $C$ with $\mu_{\text{cov}}$, $c_{\text{cov}}$, $p_\sigma$, $p_c$:
	\begin{align*}
C =	c_{\text{cov}} \cdot \left( 1- \frac{1}{\mu_{\text{cov}}} \right) \frac{1}{\sigma^2} \sum_{i=1}^{m_2} w_i v_i v_i^\top + \left(1 - c_{\text{cov}} \right)\cdot C + \frac{c_{\text{cov}}}{\mu_{\text{cov}}} p_c p_c^\top
	\end{align*}		
	\;
	Update step size $\sigma$ with $c_\sigma$, $d_\sigma$, $p_\sigma$:
	\begin{align*}
	\sigma = \sigma \cdot \exp{\left(\frac{c_\sigma}{d_\sigma}\left(\frac{\vert p_\sigma \vert}{E\vert \operatorname{N}\left(0,I\right) \vert} - 1 \right)\right)}
	\end{align*}		 \;
	\HiLi{Parallel transport $C$, $p_\sigma$, $p_c$ from $T_{\mu_k} M$ to $T_{\mu_{k+1}} M$} \;
	k = k+1 \;
 } 
 \caption{Riemannian CMA-ES}
 \label{alg:rcma-es}
\end{algorithm}

\clearpage

\section{Proofs of results in Section 4.1}
\label{app:proof:induceddualisticgeo}
\naturality*
\begin{proof} 
Let $\tilde{S}$ be a smooth n-dimensional manifold, and let $S$ be a smooth n-dimensional manifold with torsion-free dualistic structure $(g,\nabla,\nabla^*)$ then the following condition is satisfied:

\begin{align*}
X\langle Y,Z \rangle_g = \langle \nabla_X Y,Z \rangle_g + \langle Y, \nabla^*_X Z \rangle_g\quad, \quad \forall X,Y,Z \in \mathcal{E}\left(TS\right)\quad .
\end{align*}

Let $\varphi:\tilde{S} \rightarrow S$ be a diffeomorphism. The pullback of $g$ along $\varphi$ is thus given by $\tilde{g}=\varphi^* g$, which defines a Riemannian metric on $\tilde{S}$. Consider pullback of $\nabla$ via $\varphi$ given by:
\begin{align*}
\varphi^* \nabla:\mathcal{E}(T\tilde{S}) \times \mathcal{E}(T\tilde{S}) &\rightarrow \mathcal{E}(T\tilde{S}) \\ 
(\varphi^* \nabla)(\tilde{X},\tilde{Y}) &= \varphi_*^{-1}\nabla (\varphi_* \tilde{X}, \varphi_* \tilde{Y}) = \varphi_*^{-1}\nabla_{\varphi_* \tilde{X}} \varphi_* \tilde{Y}  \quad ,
\end{align*}

where $\mathcal{E}(T\tilde{S})$ denote the set of smooth sections of tangent bundle over $\tilde{S}$, $\varphi_*$ denote the push-forward of $\varphi$, and $\varphi^*$ denote the pullback of $\varphi$.
Since pullback of torsion-free connection by diffeomorphism is a torsion-free connection, the pullback connections $\varphi^*\nabla$ and $\varphi^*\nabla^*$ are torsion-free connections on the tangent bundle over $\tilde{S}$.

It remains to show that $(\tilde{\nabla},\tilde{\nabla}^*) := (\varphi^* \nabla,\varphi^* \nabla^*)$ is a $\tilde{g}$-conjugate pair of connections on $\tilde{S}$. In particular, we show that the pair $(\tilde{\nabla},\tilde{\nabla}^*)$ satisfies the following equation:
\begin{gather*}
\tilde{X}\langle \tilde{Y},\tilde{Z} \rangle_{\tilde{g}} = \langle \tilde{\nabla}_{\tilde{X}} \tilde{Y},\tilde{Z} \rangle_{\tilde{g}} + \langle \tilde{Y}, \tilde{\nabla}^*_{\tilde{X}} \tilde{Z} \rangle_{\tilde{g}},\quad \forall \tilde{X},\tilde{Y},\tilde{Z} \in \mathcal{E}(T\tilde{S}) \quad .
\end{gather*}
Let $\tilde{X},\tilde{Y},\tilde{Z} \in \mathcal{E}(T\tilde{S})$, and let $ \tilde{p} \in \tilde{S}$ be an arbitrary point:
\begin{align*}
\langle \tilde{\nabla}_{\tilde{X}} \tilde{Y},\tilde{Z} \rangle_{\tilde{g}}(\tilde{p}) &= \langle (\varphi^* \nabla)_{\tilde{X}} \tilde{Y},\tilde{Z}\rangle_{\tilde{g}}(\tilde{p}) = \langle \varphi_*^{-1} \nabla_{\varphi_* \tilde{X}} \varphi_* \tilde{Y}_{\tilde{p}}, \tilde{Z}_p\rangle_{\tilde{g}} \\ 
&= \langle \varphi_*^{-1} \left(\nabla_{\varphi_* \tilde{X}} \varphi_* \tilde{Y}\right)_{\varphi(\tilde{p})}, \varphi_*^{-1} \left(\varphi_* \tilde{Z}\right)_{\varphi(\tilde{p})}\rangle_{\tilde{g}} \\
&= \left(\varphi^{*^{-1}}\tilde{g}\right)_{\varphi(\tilde{p})} \left(\nabla_{\varphi_* \tilde{X}} \varphi_* \tilde{Y}_{\varphi(\tilde{p})}, \varphi_* \tilde{Z}_{\varphi(\tilde{p})}\right) \\
&= \langle  \nabla_{\varphi_* \tilde{X}} \varphi_* \tilde{Y}_{\varphi(\tilde{p})}, \varphi_* \tilde{Z}_{\varphi(\tilde{p})} \rangle_{g} = \langle  \nabla_{\varphi_* \tilde{X}} \varphi_* \tilde{Y}, \varphi_* \tilde{Z} \rangle_{g} (\varphi(\tilde{p})) \quad .
\end{align*}
Similarly, we have the following ``symmetric" argument:
\begin{gather*}
\langle \tilde{Y},\tilde{\nabla}^*_{\tilde{X}} \tilde{Z} \rangle_{\tilde{g}}(\tilde{p}) =  \langle   \varphi_* \tilde{Y},\nabla_{\varphi_* \tilde{X}} \varphi_* \tilde{Z} \rangle_{g} (\varphi(\tilde{p})) \quad .
\end{gather*}
Since $(\nabla,\nabla^*)$ is $g$-conjugate pair of connection on $S$, we have for each $\tilde{p} \in \tilde{S}$:
\begin{align*}
\langle  \nabla_{\varphi_* \tilde{X}} \varphi_* \tilde{Y}_{\varphi(\tilde{p})}, \varphi_* \tilde{Z}_{\varphi(\tilde{p})} \rangle_{g} + \langle \varphi_* \tilde{Y}_{\varphi(\tilde{p})}, \nabla^*_{\varphi_* \tilde{X}}\varphi_* \tilde{Z}_{\varphi(\tilde{p})} \rangle_{g} =\varphi_* \tilde{X} \left( \langle \varphi_* \tilde{Y}, \varphi_* \tilde{Z} \rangle_g \right) (\varphi(\tilde{p})) \quad .
\end{align*}
Hence:
\begin{align*}
\varphi_* \tilde{X} \left( \langle \varphi_* \tilde{Y}, \varphi_* \tilde{Z} \rangle_g \right) (\varphi(\tilde{p})) &= \varphi_* \tilde{X} \left( \langle \varphi_* \tilde{Y}, \varphi_* \tilde{Z} \rangle_g \circ \varphi \circ \varphi^{-1} \right)( \varphi(\tilde{p})) \\
&= \varphi_* \tilde{X} \left((\varphi^* g)(\tilde{Y},\tilde{Z}) \circ \varphi^{-1} \right) (\varphi(\tilde{p})) \\
&= \tilde{X} \left( \langle \tilde{Y},\tilde{Z} \rangle_{\tilde{g}} \circ \varphi^{-1} \circ \varphi \right) \circ \varphi^{-1} (\varphi(\tilde{p})) \\
&= \tilde{X} \langle \tilde{Y},\tilde{Z} \rangle_{\tilde{g}}(\tilde{p}) \quad .
\end{align*}

Combining the above results we obtain the following:
\begin{gather*}
\langle \tilde{\nabla}_{\tilde{X}} \tilde{Y},\tilde{Z} \rangle_{\tilde{g}}(\tilde{p}) + \langle \tilde{Y},\tilde{\nabla}^*_{\tilde{X}} \tilde{Z} \rangle_{\tilde{g}}(\tilde{p}) = \tilde{X} \langle \tilde{Y},\tilde{Z} \rangle_{\tilde{g}}(\tilde{p}) \quad , \quad \forall \tilde{X},\tilde{Y},\tilde{Z}\in \mathcal{E}(T\tilde{S}) \quad .
\end{gather*}

Therefore $\tilde{S}$ can be equipped with the induced torsion-free dualistic structure $(\tilde{g},\tilde{\nabla},\tilde{\nabla}^*)=(\varphi^*g,\varphi^*\nabla,\varphi^*\nabla^*)$ as desired.
\end{proof}

\curvaturenaturality*
\begin{proof}
Let $\varphi: (\tilde{S},\tilde{g},\tilde{\nabla},\tilde{\nabla}^*) \rightarrow (S,g,\nabla,\nabla^*)$ be a local isometry, where $(\tilde{g},\tilde{\nabla},\tilde{\nabla}^*) = (\varphi^* g,\varphi^* \nabla, \varphi^* \nabla^*)$, then by the proof of Proposition \ref{naturality}:
\begin{gather*}
\tilde{\nabla}_{\tilde{X}} \tilde{Y} = (\varphi^* \nabla)_{\tilde{X}} \tilde{Y} = \varphi_*^{-1}(\nabla_{\varphi_* \tilde{X}} \varphi_* \tilde{Y}), \quad \forall \tilde{X},\tilde{Y}\in \mathcal{E} \left(T \tilde{S}\right) \quad .
\end{gather*}
Hence we have:
\begin{align*}
\varphi_*(\tilde{\nabla}_{\tilde{X}} \tilde{\nabla}_{\tilde{Y}} \tilde{Z}) &= \nabla_{\varphi_* \tilde{X}} \nabla_{\varphi_* \tilde{Y}} \varphi_* \tilde{Z} \quad , \\
\varphi_*(\tilde{\nabla}_{\left[\tilde{X},\tilde{Y} \right]} \tilde{Z}) &= \nabla_{\varphi_* \left[\tilde{X},\tilde{Y} \right]} \varphi_* \tilde{Z} = \tilde{\nabla}_{\left[\varphi_* \tilde{X},\varphi_* \tilde{Y} \right]} \varphi_* \tilde{Z} \quad . 
\end{align*}
The above Equation is obtained by:
\begin{align*}
\varphi_* \left(\tilde{\nabla}_{\tilde{X}} \tilde{\nabla}_{\tilde{Y}} \tilde{Z} \right) &= \varphi_* \tilde{\nabla} \left(\tilde{X},  \tilde{\nabla}_{\tilde{Y}} \tilde{Z} \right) = \nabla \left(\varphi_* \tilde{X}, \varphi_* \tilde{\nabla}_{\tilde{Y}} \tilde{Z} \right) \\
&= \nabla \left( \varphi_* \tilde{X}, \nabla \left( \varphi_* \tilde{Y},\varphi_* \tilde{Z} \right) \right) = \nabla_{\varphi_* \tilde{X}} \nabla_{\varphi_* \tilde{Y}} \varphi_* Z \quad .
\end{align*}
For vector fields $X,Y,Z \in \mathcal{E}\left(TS\right)$:, the Riemannian curvature tensor on $S$ is given by: 
\begin{gather*}
\langle R(X,Y)Z,W \rangle = \langle \nabla_X \nabla_Y Z, W \rangle - \langle \nabla_Y \nabla_X Z, W \rangle - \langle \nabla_{\left[X,Y \right]} Z, W \rangle \quad .
\end{gather*}
Let $\tilde{X}, \tilde{Y}, \tilde{Z}, \tilde{W} \in \mathcal{E}\left( T\tilde{S}\right)$. The pullback curvature tensor $\varphi^*R$ is thus given by:
\begin{align*}
\langle (\varphi^* R)(\tilde{X},\tilde{Y})\tilde{Z}, \tilde{W} \rangle & = \langle\tilde{\nabla}_{\varphi_* \tilde{X}} \tilde{\nabla}_{\varphi_* \tilde{Y}} \varphi_* \tilde{Z}, \tilde{W} \rangle \\ 
& \quad - \langle \tilde{\nabla}_{\varphi_* \tilde{Y}} \tilde{\nabla}_{\varphi_* \tilde{X}} \varphi_* \tilde{Z}, \tilde{W} \rangle - \langle \tilde{\nabla}_{\left[\varphi_* \tilde{X},\varphi_* \tilde{Y} \right]} \varphi_* \tilde{Z}, \tilde{W} \rangle \\
& = \varphi_*(\nabla_X \nabla_Y Z) - \varphi_*(\nabla_Y \nabla_X Z) - \varphi_*(\nabla_{\left[X,Y \right]} Z) \\ 
& = \varphi_* (\langle R(X,Y)Z,W \rangle) \quad .
\end{align*}
And by symmetry, $\varphi^*R^*$ satisfies:
\begin{align*}
\langle (\varphi^* R^*)(\tilde{X},\tilde{Y})\tilde{Z}, \tilde{W} \rangle = \varphi_* (\langle R^*(X,Y)Z,W \rangle) \quad .
\end{align*}
Therefore if $S$ is dually flat, meaning $R = 0 = R^*$, we have the following equality:
\begin{gather*}
R^* \equiv 0 \Leftrightarrow R \equiv 0 \Leftrightarrow \varphi^* R \equiv 0 \Leftrightarrow \varphi^* R^* \equiv 0 \quad .
\end{gather*}
\end{proof}

\computenaturality*
\begin{proof}
By Corollary \ref{df} if $\left(S,g,\nabla,\nabla^* \right)$ is dually flat, then so is $\left( \tilde{S}, \varphi^* g, \varphi^*\nabla, \varphi^* \nabla^* \right)$. Let $(\theta_i)_{i=1}^n$ denote local $\nabla$-affine coordinates on $\left(S,g,\nabla,\nabla^*\right)$, and let $\left(\overline{\theta_i}:= \theta_i \circ \varphi\right)_{i=1}^n$ denote the  pulled-back local coordinate system on $\tilde{S}$ via $\varphi:\tilde{S} \rightarrow S$.

Let $\left(\partial_i := \frac{\partial}{\partial \theta_i}\right)_{i=1}^n$ denote the local coordinate frame for $TS$ corresponding to local $\nabla$-affine coordinates $\left(\theta_i \right)_{i=1}^n$, and let $\mathcal{U} \subset T\tilde{S}$ denote the tangent subspace spanned by vector fields $\left( \tilde{\partial}_i := \varphi^{-1}_* \partial_i \right)_{i=1}^n$. For $i,j = 1,\ldots, n$, since $\left[\varphi^{-1}_* \partial_i, \varphi^{-1}_* \partial_j \right] = \varphi^{-1}_* \left[\partial_i, \partial_j \right]  =0$, the vector fields $\left( \tilde{\partial}_i := \varphi^{-1}_* \partial_i \right)_{i=1}^n$ commute, hence the space $\mathcal{U} $, spanned by $\left(\tilde{\partial}_i \right)_{i=1}^n$, is involutive. Therefore by theorem of Frobenius, $\mathcal{U}$ is completely integrable, hence there exists $(\varphi^* \nabla)$-affine coordinates $\left(\tilde{\theta}_i \right)_{i=1}^n$ on $\tilde{S}$ with corresponding local frame $\left( \tilde{\partial}_i \right)_{i=1}^n$. 
\begin{remark}
\label{Rmk:pullbackcoord}
The pulled-back coordinates and the corresponding local coordinate frame described in this fashion does not depend on the dual flatness of $S$ and $\tilde{S}$. Of course, when $S$ and $\tilde{S}$ are not dually flat, the coordinate systems may no longer be $\nabla$,$\tilde{\nabla}$-affine. 
\end{remark}

Let $\left(\partial^i := \frac{\partial}{\partial \eta_i} \right)_{i=1}^n$ denote local coordinate frame of $TS$ corresponding to local $\nabla^*$-affine coordinates $\left(\eta_i \right)_{i=1}^n$. We can define local $(\varphi^* \nabla^*)$-affine coordinates $\left(\tilde{\eta}_i \right)_{i=1}^n$ on $\tilde{S}$ with corresponding local coordinate frame $\left( \tilde{\partial}^i := \varphi^{-1}_* \partial^i \right)_{i=1}^n$. It is immediate that $\left(\tilde{\theta}_i\right)_{i=1}^n,\left(\tilde{\eta}_i\right)_{i=1}^n$ are $\varphi^* g$-dual coordinates on $\tilde{S}$: For $i,j = 1,\ldots, n$:
\begin{align*}
\langle \tilde{\partial}_i,\tilde{\partial}^j \rangle_{\tilde{g}} &= \left( \varphi^* g\right) \left(\varphi_*^{-1} \partial_i, \varphi^{-1}_* \partial^j \right) \\
&= g\left(\varphi_*\varphi_*^{-1}\partial_i, \varphi_* \varphi_*^{-1} \partial^j \right) \\
&= g\left(\partial_i,\partial^j \right) = \delta_i^j \quad .
\end{align*}
Moreover, if we consider smooth pulled back coordinates $\left(\overline{\theta_i}:= \theta_i \circ \varphi\right)_{i=1}^n$ and the corresponding local coordinate frame $\left(\overline{\partial_i}\right)_{i=1}^n :=\left(\frac{\partial}{\partial \overline{\theta}_i} \right)_{i=1}^n$ on $T\tilde{S}$. Then we obtain for $i,j = 1,\ldots, n$:
\begin{align*}
\tilde{\partial}_i(\overline{\theta}_j) &= \varphi^{-1}_* \partial_i (\theta_j \circ \varphi) \\
    &= \partial_i (\theta_j \circ \varphi \circ \varphi^{-1}) = \partial_i \theta_j = \delta_i^j \quad .
\end{align*}
This implies for each $i = 1,\ldots, n$, there exists a constant  $c_i$ such that $\tilde{\theta}_i = \overline{\theta}_i + c_i$, which in turn implies $\tilde{\partial}_i = \overline{\partial}_i$ for all $i = 1,\ldots n$.

Recall from Section \ref{Subsection:computingInducedDualisticGeo} Equation \eqref{eqn:div}: the divergence function $\overline{D}$ generated by the induced potential function $\tilde{\psi}$ in the local pulled-back coordinates $\left(\overline{\theta_i} \right)_{i=1}^n$:
\begin{align*}
\overline{D}: \tilde{S} \times \tilde{S} &\rightarrow \mathbb{R}_+ \\
(\tilde{p},\tilde{q}) &\mapsto \overline{D}(\tilde{p},\tilde{q}) =  \tilde{\psi}(\tilde{p}) + \tilde{\psi}^\dagger(\tilde{q}) - \langle \overline{\theta}(\tilde{p}), \overline{\eta}(\tilde{q}) \rangle \quad ,
\end{align*}
where $\tilde{\psi}^\dagger$ is a smooth function on $\tilde{S}$ representing the Legendre-Fr\'echet transformation of $\bar{\psi}$ with respect to the pair of $\tilde{g}$-dual local coordinates $\left(\overline{\theta}_i\right)_{i=1}^n,\left(\overline{\eta}_i\right)_{i=1}^n$ on $\tilde{S}$.

Let $\overline{g}$ denote the Hessian metric generated by $\overline{D}$, then  $\left. \overline{g}_{ij} \right|_{\tilde{p}} := \left. g^{\overline{D}}_{ij}\right|_{\tilde{p}}$. By definition and the above arguments $\left.\tilde{g}_{ij}\right|_{\tilde{p}} = \tilde{\partial}_i \tilde{\partial}_j \tilde{\psi} = \overline{\partial}_i \overline{\partial}_j \tilde{\psi} =  - \overline{\partial}^1_i \overline{\partial}^2_j \overline{D} = \left. g^{\overline{D}}_{ij}\right|_{\tilde{p}}$ for $i,j = 1,\ldots, n$, hence $\tilde{g} = \overline{g}$.

Let $\left(\nabla^{\overline{D}}, \nabla^{\overline{D}^*} \right)$ denote pair of $\tilde{g}$-dual connections defined by $\overline{D}$.  We now show $\left(\tilde{\nabla},\tilde{\nabla}^* \right) = \left(\overline{\nabla},\overline{\nabla}^* \right)$. Let $\tilde{X},\tilde{Y},\tilde{Z} \in \mathcal{E} \left(T \tilde{S}\right)$ and $\tilde{p} \in \tilde{S}$, the following is satisfied by construction:
\begin{align*}
\tilde{X}\langle \tilde{Y},\tilde{Z} \rangle_{\tilde{g}} &= \langle \tilde{\nabla}_{\tilde{X}} \tilde{Y},\tilde{Z} \rangle_{\tilde{g}} + \langle \tilde{Y}, \tilde{\nabla}_{\tilde{X}}^* \tilde{Z} \rangle_{\tilde{g}} \\
\tilde{X}\langle \tilde{Y},\tilde{Z} \rangle_{\overline{g}} &= \langle \overline{\nabla}_{\tilde{X}} \tilde{Y},\tilde{Z} \rangle_{\overline{g}} + \langle \tilde{Y}, \overline{\nabla}_{\tilde{X}}^* \tilde{Z} \rangle_{\overline{g}} \quad .
\end{align*}
Since $\tilde{g} = \overline{g}$, the two equations are equal. For $x\in \tilde{S}$, let $\left(\overline{\partial}_i = \tilde{\partial}_i\right)_{i=1}^n$ denote the local frame of $T_p \tilde{S}$, and let $\tilde{X} = \tilde{\partial}_i = \overline{\partial}_i$, $\tilde{Y} = \tilde{\partial}_j = \overline{\partial}_k$, and $\tilde{Z} = \tilde{\partial}_k = \overline{\partial}_k$ then:
\begin{align*}
\langle \tilde{\nabla}_{\tilde{X}} \tilde{Y},\tilde{Z} \rangle_{\tilde{g}} 
	+ \langle \tilde{Y}, \tilde{\nabla}_{\tilde{X}}^* \tilde{Z} \rangle_{\tilde{g}} 
&= \langle \overline{\nabla}_{\tilde{X}} \tilde{Y},\tilde{Z} \rangle_{\overline{g}} 
	+ \langle \tilde{Y}, \overline{\nabla}_{\tilde{X}}^* \tilde{Z} \rangle_{\overline{g}} 
\\
= \langle \tilde{\nabla}_{\tilde{\partial}_i} \tilde{\partial}_j,\tilde{\partial}_k \rangle_{\tilde{g}} 
	+ \langle \tilde{\partial}_j, \tilde{\nabla}_{\tilde{\partial}_i}^* \tilde{\partial}_k \rangle_{\tilde{g}} 
&= \langle \overline{\nabla}_{\overline{\partial}_i} \overline{\partial}_j, \overline{\partial}_k \rangle_{\overline{g}} 
	+ \langle \overline{\partial}_j, \overline{\nabla}_{\overline{\partial}_i}^* \overline{\partial}_k \rangle_{\overline{g}} \\
= 0 + \langle \tilde{\partial}_j, \tilde{\nabla}_{\tilde{\partial}_i}^* \tilde{\partial}_k \rangle_{\tilde{g}} 
&= 0+ \langle \overline{\partial}_j, \overline{\nabla}_{\overline{\partial}_i}^* \overline{\partial}_k \rangle_{\overline{g}}
\end{align*}
Since $\tilde{\partial}_i = \overline{\partial}_i$ for all $i$, and $\tilde{g} = \overline{g}$, we have for all $\tilde{p} \in \tilde{S}$:
\begin{align*}
 \langle \tilde{\partial}_j, \tilde{\nabla}_{\tilde{\partial}_i}^* \tilde{\partial}_k \rangle_{\tilde{g}}
&=  \langle \tilde{\partial}_j, \overline{\nabla}_{\tilde{\partial}_i}^* \tilde{\partial}_k \rangle_{\tilde{g}} \quad, \quad \forall i,j,k\quad.
\end{align*}
Therefore $\tilde{\nabla}^* = \overline{\nabla}^*$, and by symmetry $\left(\tilde{\nabla},\tilde{\nabla}^* \right) = \left(\overline{\nabla},\overline{\nabla}^* \right)$.
Furthermore, we can determine the explicit expression of the Christoffel symbols of the induced connection $\overline{\nabla}^* = \tilde{\nabla}^* = \varphi^*\nabla^*$ at $\tilde{p} \in \tilde{S}$ in pulled-back coordinates $\left(\overline{\theta}_i = \theta_i \circ \varphi\right)_{i=1}^n$ as follows.
Let $x\in S$ be an arbitrary, and let $\tilde{p} := \overline{\varphi}(p) = \varphi^{-1}(p) \in \tilde{S}$, then the Christoffel symbols coincide: 
\begin{align*}
\left. \tilde{\Gamma}^*_{ijk}\right|_{\tilde{p}} &= \overline{\partial}_i \overline{\partial}_j \overline{\partial}_k \tilde{\psi} (\tilde{p}) =  \overline{\partial}_i \tilde{g}_{jk}(\tilde{p}) \\
&= \tilde{\partial}_i \tilde{g}_{jk}(\tilde{p}) =  \left( \varphi^{-1}_* \partial_i\right) \tilde{g}_{jk} (\tilde{p}) \\
&= \partial_i \left( g_{jk} \circ \varphi \circ \varphi^{-1} \right) (\varphi(\tilde{p})) = \left. \Gamma^*_{ijk}\right|_p \quad .
\end{align*}
\end{proof}

\section{Proofs of Lemma 5.2} 
\label{app:proof:OPcover}

\lemmamakediffeoOP*
\begin{proof}
Since $f:M\rightarrow N$ is a local diffeomorphism, the pushforward $f_* :T_p M \rightarrow T_{f(p)} N $ is a linear isomorphism for all $x\in M$. 
Since $f_*$ is a linear isomorphism, the determinant of the matrix $Df$ is non-zero: $\text{det}Df \neq 0$. If $f$ is orientation-preserving, then there's nothing left to prove. Hence we will now assume $f$ is orientation reversing, in other words: $\text{det}Df < 0$.  

Let $\left(x^1,\ldots, x^n \right)$ denote local coordinates in $M$, let $\overline{f}$ denote coordinate representation of $f$, then we can write:
\begin{align*}
\overline{f}\left(x^1(p),\ldots, x^n(p)\right) = \left( \overline{f}_1(x),\ldots, \overline{f}_n(x)\right) \quad ,
\end{align*}
where $x := \left( x^1(p), \ldots, x^n(p) \right)$.

Choose $a\in \left[1,\ldots, n\right]$, and let $\tilde{f}: M \rightarrow N$ denote the diffeomorphism from $M$ to $N$ defined by the following coordinate representation:
\begin{align*}
\tilde{f}\left(x^1(p),\ldots, x^n(p)\right) = \left( \overline{f}_1(x),\ldots, \overline{f}_{a+1} (x), \overline{f}_{a}(x),\ldots,\overline{f}_n(x)\right) \quad .
\end{align*}

In other words, we define $\tilde{f}$ by swap the $a^{th}$ and ${a+1}^{th}$ coordinates of $f$.
The matrix representation of $\tilde{f}_*$ in standard coordinates is thus given by:
\begin{align*}
D\tilde{f} = I' \cdot Df \quad ,
\end{align*}
where $I'$ is the matrix given by:
\[
I' = 
\left[
\begin{array}{ccc}
I_{a-1} & & \\
& \begin{array}{cc}
	0 & 1\\
	1 & 0
  \end{array} & \\
& & I_{n-(a+1)}
\end{array}
\right] \quad ,
\]
where $I_k$ is the identity matrix of the dimension $k$, the sub-matrix $ \left[ \begin{array}{cc}
	0 & 1\\
	1 & 0
  \end{array} \right] $ is located at the $\left(a,a\right)^{th}$ to the $\left(a+1,a+1\right)^{th}$ position of $I'$, and the rest of the entries are all zero.
Since $\text{det}Df < 0$, this means $\text{det}D\tilde{f}  = \text{det}I' \cdot \text{det}Df  =  -1 \cdot\text{det}Df  > 0$. 
\end{proof}

\end{document}